\documentclass[leqno, 11pt]{amsart}
\usepackage{amssymb, amsmath}
\usepackage{xcolor}
\numberwithin{equation}{section}

\usepackage{color,ulem,textcomp}

\usepackage{cancel}

\def\refer#1{~\ref{#1}}
\def\refeq#1{~(\ref{#1})}
\def\ccite#1{~\cite{#1}}

\def\suite#1#2#3{(#1_{#2})_{#2\in {#3}}}

\def\longformule#1#2{
\displaylines{ \qquad{#1} \hfill\cr \hfill {#2} \qquad\cr } }
\def\inte#1{
\displaystyle\mathop{#1\kern0pt}^\circ }
\def\sumetage#1#2{
\sum_{\scriptstyle {#1}\atop\scriptstyle {#2}} }
\def\minetage#1#2{
\min_{\scriptstyle {#1}\atop\scriptstyle {#2}} }
\def\sumetagetr#1#2#3{
\sum_{\scriptstyle {#1}\atop{\scriptstyle {#2}\atop\scriptstyle {#3}} }}

\def\supetage#1#2{
\sup_{\scriptstyle {#1}\atop\scriptstyle {#2}} }

\def\dive{\mathop{\rm div}\nolimits}

\let\al=\alpha
\let\b=\beta
\let\g=\gamma

\let\e=\varepsilon

\let\lam=\lambda

\let\s=\sigma

\let\vf=\varphi

\let\D=\Delta
\let\Lam=\Lambda

\let\wt=\widetilde

\let\convf=\rightharpoonup

\def\cA{{\mathcal A}}
\def\cB{{\mathcal B}}
\def\cC{{\mathcal C}}
\def\cD{{\mathcal D}}
\def\cE{{\mathcal E}}
\def\cF{{\mathcal F}}

\def\cL{{\mathcal L}}
\def\cM{{\mathcal M}}
\def\cN{{\mathcal N}}

\def\cQ{{\mathcal Q}}
\def\cR{{\mathcal R}}

\def\cT{{\mathcal T}}

\def\with{\quad\hbox{with}\quad}
\def\andf{\quad\hbox{and}\quad}

\def\virgp{\raise 2pt\hbox{,}}
\def\cdotpv{\raise 2pt\hbox{;}}

\def\eqdefa{\buildrel\hbox{\footnotesize def}\over =}

\def\C{{\mathop{\bf C\kern 0pt}\nolimits}}
\def\DD{{\mathop{\bf D\kern 0pt}\nolimits}}
\def\K{{\mathop{\bf K\kern 0pt}\nolimits}}
\def\N{{\mathop{\mathbb N\kern 0pt}\nolimits}}
\def\PP{{\mathop{\mathbb P\kern 0pt}\nolimits}}
\def\Q{{\mathop{\mathbb Q\kern 0pt}\nolimits}}
\def\R{{\mathop{\mathbb R\kern 0pt}\nolimits}}
\def\SS{{\mathop{\mathbb S\kern 0pt}\nolimits}}
\def\ZZ{{\mathop{\mathbb Z\kern 0pt}\nolimits}}
\def\TT{{\mathop{\mathbb T\kern 0pt}\nolimits}}
\def\BB{{\mathop{\mathbb B\kern 0pt}\nolimits}}
\def\PP{{\mathop{\mathbb P\kern 0pt}\nolimits}}

\newcommand{\ds}{\displaystyle}

\newcommand{\eps}{\varepsilon}

\newcommand{\beq}{\begin{equation}}
\newcommand{\eeq}{\end{equation}}
\newcommand{\ben}{\begin{eqnarray}}
\newcommand{\een}{\end{eqnarray}}
\newcommand{\beno}{\begin{eqnarray*}}
\newcommand{\eeno}{\end{eqnarray*}}

\newtheorem{defi}{Definition}[section]
\newtheorem{thm}{Theorem}

\newtheorem{lem}[defi]{Lemma}
\newtheorem{rmk}[defi]{Remark}
\newtheorem{cor}[defi]{Corollary} 
\newtheorem{prop}[defi]{Proposition}

\begin{document}

  \title[Stability by  rescaled weak convergence for the Navier-Stokes equations]
  { Stability by  rescaled weak convergence for the Navier-Stokes equations
  }

\author[H. Bahouri]{Hajer Bahouri}
\address[H. Bahouri]%
{Laboratoire d'Analyse et de Math{\'e}matiques Appliqu{\'e}es UMR 8050 \\
Universit\'e Paris-Est Cr\'eteil \\
61, avenue du G{\'e}n{\'e}ral de Gaulle\\
94010 Cr{\'e}teil Cedex, France}
\email{hbahouri@math.cnrs.fr}
\author[J.-Y. Chemin]{Jean-Yves  Chemin}
\address[J.-Y. Chemin]%
{Laboratoire Jacques Louis Lions - UMR 7598,  Universit\'e Pierre et Marie Curie \\ Bo\^\i te courrier 187, 4 place Jussieu, 75252 Paris
Cedex 05, France}
\email{chemin@ann.jussieu.fr }
\author[I. Gallagher]{Isabelle Gallagher}
\address[I. Gallagher]%
{Institut de Math{\'e}matiques de Jussieu - Paris Rive Gauche UMR 7586 \\
      Universit{\'e} Paris-Diderot (Paris 7)  \\
B\^atiment Sophie Germain \\
Case 7012, 75205 Paris Cedex 13\\France
}
\email{gallagher@math.univ-paris-diderot.fr}
 \thanks{The third  author was partially supported by the Institut Universitaire de France and the ANR project HAB ANR-12-BS01-0013-01.}
 \begin{abstract}
We prove a weak stability result for the three-dimensional  homogeneous incompressible Navier-Stokes system. More precisely, we investigate the following problem : if a sequence $(u_{0, n})_{n\in \N}$ of initial data, bounded in some scaling invariant space,  converges weakly to an initial
data $u_0$ which generates a global regular solution, does $u_{0, n}$ generate a global regular solution ?
A positive answer in general to this question would imply global regularity for any data, through the following examples~$u_{0,n} = n \vf_0(n\cdot)$ or~$u_{0,n} = \vf_0(\cdot-x_n)$ with~$|x_n|\to \infty$. We therefore introduce a new concept of  weak convergence (rescaled weak convergence) under which we are able to give a positive answer. The proof relies on profile decompositions in anisotropic spaces and their propagation by the Navier-Stokes equations.  
  \end{abstract}
\subjclass[2010]{}

 \maketitle

%\tableofcontents

 \section{Introduction and statement of the main result}
  \label{first}
\subsection{The Navier-Stokes equations}
We are interested in the Cauchy problem for the three dimensional, homogeneous, incompressible Navier-Stokes system
$$
{\rm(NS)} \quad  \left\{
\begin{array}{l}
\partial_t u + u\cdot \nabla u -\Delta u= -\nabla p  \quad \mbox{in} \quad
\R^+ \times \R^3\\
\mbox{div} \, u= 0\\
u_{|t=0} = u_{0}\, ,
\end{array}
\right.
$$
where~$p=p(t,x)$ and~$u=(u^1,u^2,u^3)(t,x)$ are respectively the pressure and velocity of an incompressible, viscous fluid. 

\medbreak
\noindent We shall say that~$u \in L^2_{\rm{loc}}([0,T] \times \R^3)$ is a {\it weak solution} of~(NS) associated with the data~$u_0$ if
for any compactly supported, divergence free vector field~$\phi$ in~$ {\mathcal C}^\infty([0,T] \times \R^3)$ the following holds for all~$t \leq T$:
$$
 \int_{ \R^3} u  \cdot \phi (t,x)  dx =\int _{ \R^3} u_0(x) \cdot \phi (0,x)  dx   +\int_0^t \int_{ \R^3} ( u \cdot \Delta \phi  + u \otimes u : \nabla \phi + u \cdot \partial_t \phi) dxdt' \, ,
 $$
 with
 $$
 u\otimes u:\nabla \phi \eqdefa \sum_{1\leq j,k\leq 3}  u^ju^k\partial_k\phi^j \, .
 $$

\noindent As is well-known, the (NS)  system enjoys two important features. First it formally conserves the energy, in the sense that  smooth enough solutions satisfy the following equality for all times~$t \geq 0$:
\begin{equation}\label{energy=}
\frac12 \|u(t)\|_{L^2(\R^3)}^2  + \int_0^t \|\nabla u(t')\|_{L^2(\R^3)}^2 \, dt' = \frac12\|u_0\|_{L^2(\R^3)}^2 \, .
\end{equation}
Weak solutions satisfying the energy inequality
\begin{equation}\label{energyinequality}
\frac12 \|u(t)\|_{L^2(\R^3)}^2  + \int_0^t \|\nabla u(t')\|_{L^2(\R^3)}^2 \, dt' \leq \frac12\|u_0\|_{L^2(\R^3)}^2  
\end{equation}
are said to be  {\it turbulent solutions}, following the terminology of J. Leray~\cite{leray}.
The energy equality~(\ref{energy=}) can easily be obtained noticing that thanks to the divergence free condition, the nonlinear term is skew-symmetric in~$L^2$: one has indeed~$\big (u(t) \cdot \nabla u (t)+ \nabla p(t)| u(t)\big)_{L^2}  = 0$. 

\medbreak
\noindent Second, (NS)  enjoys a scaling invariance property: defining the scaling operators, for any positive real number~$\lambda $ and any point~$x_0$ of~$\R^3$,
   \begin{equation}
   \label{defscalingoperatorintro}
    \Lambda_{  {\lambda },x_0} \phi (t,x) \eqdefa
  \frac1 {\lambda}  \phi \Big( \frac t {\lambda ^2},\frac {x-x_0} {\lambda} \Big)   \andf
  \Lam_\lam \phi (t,x) \eqdefa   \frac 1 \lam\phi \Big( \frac t {\lambda ^2},\frac x\lambda \Big)\, \virgp
\end{equation}
 if~$u$ solves~(NS) with data~$u_0$, then~$   \Lambda_{  {\lambda},x_0}   u$ solves~(NS) with data~$\Lambda_{  {\lambda},x_0} u_0$. We shall say that a familly~$(X_T)_{T>0}$ of   spaces of distributions over~$[0,T] \times \R^3$ is {\it scaling invariant} if for all~$T >0$  one has
   $$
  \forall \lambda >0 \, , \forall x_0 \in \R^3 \, ,  \, u \in X_T \Longleftrightarrow    \Lambda_{  {\lambda},x_0}   u \in  X_{\lambda^{-2}T} \quad \mbox{with} \quad \|u\|_{ X_T} =  \|   \Lambda_{  {\lambda},x_0}   u\|_{ X_{\lambda^{-2}T}}  \, .
  $$
  Similarly a space~$X_0$ of distributions defined on~$\R^3$
  will be said to be scaling invariant
  if
   $$
  \forall \lambda >0 \, , \forall x_0 \in \R^3 \, ,  \, u_0 \in X_0 \Longleftrightarrow    \Lambda_{  {\lambda},x_0}   u_0 \in  X_{0} \quad \mbox{with} \quad \|u_0\|_{ X_0} =  \|   \Lambda_{  {\lambda},x_0}   u_0\|_{ X_{0}}  \, .
  $$
 This leads to the definition of a scaled solution, which will be the notion of solution we shall consider throughout this paper: high frequencies of the solution are required to belong to a scale invariant space. In the following we denote by~$ {\mathcal F}$ the Fourier transform.
   \begin{defi}  {\sl A vector field~$u$   is a (scaled) solution to~{\rm(NS)} associated with the data~$u_0$ if it is a weak solution, such that there is a compactly supported function~$ \chi \in {\mathcal C}^\infty(  \R^3)$, equal to~1 near~0, such that
    $$
 {\mathcal F}^{-1} \big( (1- \chi ){\mathcal F}u\big) \in X_T
 $$
 where~$X_T$ belongs to a family of scaling invariant spaces.}
  \end{defi}

\medskip
\noindent 
The energy conservation~(\ref{energy=}) is the main ingredient which enabled J. Leray to prove in~\cite{leray} that any initial data in~$L^2(\R^3)$  gives rise to (at least) one global turbulent solution
to (NS), belonging to the space~$L^\infty(\R^+;L^2(\R^3) )$, with~$\nabla u $ in~$L^2(\R^+;L^2(\R^3))$. Along with that fundamental result, he could also prove that if the initial data is small enough  in the sense that~$
\|u_0\|_{L^2(\R^3)}\|\nabla u_0\|_{L^2(\R^3)} $ is small enough, then there is only one such solution, and if the data belongs also to~$ H^1$ with no such smallness assumption then that uniqueness property holds at least for a short time (time at which the solution ceases to belong to~$  H^1$).

\medskip
\noindent 
 It is important to notice that the quantity~$\|u_0\|_{L^2(\R^3)}\|\nabla u_0\|_{L^2(\R^3)}$ is invariant by the scaling operator~$ \Lambda_{{\lambda },x_0}.$ Actually in dimension 2 not only does the global existence of turbulent solutions   hold, but linked to the fact that~$\|u(t)\|_{L^2(\R^2)}$ is both scale invariant and bounded globally in time thanks to the energy inequality~(\ref{energyinequality}), J. Leray proved in~\cite{leray2D} that those solutions are actually unique, for all times whatever their size. In dimension three and more, the question of the uniqueness of  Leray's solutions is still an open problem, and in relation with that problem, a number of results
have been proved concerning the existence,  global in time,  of solutions under a scaling invariant smallness assumption on the data.  Without that smallness assumption,   existence and uniqueness often holds in a scale invariant space for a short time but nothing is known beyond that time, at which some scale-invariant norm of the solution could   blow up. The question of the possible blow up in finite time of solutions to (NS) is actually one of the Millenium Prize Problems in Mathematics.

\medskip
\noindent 
We shall not recall all the results existing in the literature concerning the Cauchy problem for (NS), and refer  for instance  to~\cite{BCD},\ccite{lemarie},\ccite{meyerNSlivre} and the references therein, for recent surveys on the subject. Let us simply recall the best result known to this day on the uniqueness of solutions to (NS), which is due to H. Koch and D. Tataru in~\cite{kochtataru} : if
$$
 \|u_0\|_{\text{BMO}^{-1}(\R^3)}\eqdefa  \|u_0\|_{  B^{-1 }_{\infty, \infty}(\R^3)}  + \supetage{x\in \R^3}{R>0} \frac 1 {R^\frac 32}\Big( \int_{[0,R^2] \times B(x,R)} |(e^{t\Delta} u_0)(t,y)|^2‚àö√á¬¨¬®‚àö¬¢¬¨√Ñ¬¨‚Ä†\, dydt\Big)^\frac12  
  $$
is small enough, then there is a global, unique solution to (NS), lying  in~$\text{BMO}^{-1}\cap X$ for all times, with~$X$  another scale invariant space to be specified -- we shall not be using that space in the sequel. In the definition of~$\text{BMO}^{-1}$ above, the norm in~$ B^{-1 }_{\infty, \infty}(\R^3)$ denotes a Besov norm, which is the end-point Besov norm in which global existence and uniqueness is known to hold for small data, namely~$B^{-1+\frac3p}_{p,\infty}$ for finite~$p$  (see~\cite{Planchon}). Let us note that (NS) is illposed for initial data in~$B^{-1}_{\infty,\infty}(\R^3)$ (see\ccite{bourgainpavlovich} and\ccite{pgermainNS}).  

  \bigskip
\noindent 
We are interested here in the {\it stability} of global solutions. Let us recall that it is proved in~\cite{adt} (see~\cite{gip} for the Besov setting) that the set of initial data generating a global solution is open in~$\text{BMO}^{-1}$. More precisely, denoting by~$\text{VMO}^{-1}$ the closure of smooth fucntions in~$\text{BMO}^{-1}$, it is proved in~\cite{adt} that if~$u_0$ belongs to~$\text{VMO}^{-1}$ and generates a global, smooth solution to~(NS), then   any sequence~$(u_{0,n})_{n \in \N}$ converging to~$u_0$ in the~$\text{BMO}^{-1}$ norm  also generates a global smooth solution as soon as~$n$ is large enough.

\medskip
\noindent In this paper we would like to address the question of {\it weak stability}: 

\medbreak

{\sl If~$(u_{0,n})_{n \in \N}$, bounded in some scale invariant space~$X_0$, converges to~$u_0$ in the sense of distributions, with~$u_0$ giving rise to a global smooth solution,  is it the case for~$u_{0,n}$ when~$n$  is large enough ?}

\medbreak

 \noindent A first step in that direction was achieved in~\cite{bg}, under two additional assumptions to the weak convergence, one of which was an assumption
on the asymptotic separation of the horizontal and vertical spectral supports: we shall come back to that assumption in Section~\ref{mainresultprofilesdi}, Remark~\ref{rkoscaniso}.
 As remarked in~\cite{bg}, the first example that may come to mind of   a sequence~$(u_{0,n})_{n \in \N}$ bounded  in a scale invariant space~$X_0$ and converging weakly to~$0$  is
\begin{equation}\label{example1}
u_{0,n} = \lam_n \Phi_0(\lam_n \, \cdot) = \Lambda_{\lam_n}\, \Phi_0 \with
\lim_{n\rightarrow\infty} \Bigl(\lam_n+\frac 1 {\lam_n} \Bigr) =\infty\,.
\end{equation} 
with~$\Phi_0$ an arbitrary divergence-free vector field. If the weak stability result were true, then since the weak limit of~$(u_{0,n})_{n \in \N}$ is zero (which gives rise to the unique, global solution which is  identically zero) then for~$n$ large enough~$u_{0,n}$ would  give rise to a unique, global solution. By scale invariance then so would~$\Phi_0$, and this for any~$\Phi_0$, so that would solve the global regularity problem for (NS). 
Another   natural example is the sequence
\begin{equation}\label{example2}
u_{0,n} =\Phi_0 (\cdot - x_n)= \Lambda_{1,x_n} \Phi_0
\end{equation}
with~$(x_n)_{n \in \N}$ a sequence of~$\R^3$ going to infinity. Thus sequences built by rescaling fixed divergence free vector fields according   to the invariances  of the equations have to be excluded from our analysis, since solving (NS) for any smooth initial data seems  out of reach. This leads naturally to the following definition of {\it rescaled weak convergence}, which we shall call~{\it R-convergence}.

   \begin{defi}[R-convergence]\label{weakcvunderscaling}
{\sl We say that a sequence~$(\varphi_n)_{n \in \N}$ {\it R-converges}   to~$\varphi$   if for all sequences~$(\lambda_n)_{n \in \N}$ of positive real numbers and for all sequences~$(x_n)_{n \in \N}$ in~$\R^3$, the sequence~$    (\Lambda_{\lambda_n,x_n} (\varphi_n-\varphi))_{n \in \N}$ converges to zero in the sense of distributions, as~$n$ goes to infinity.
  }
\end{defi}
\begin{rmk}\label{counterexamples}
{\sl  Consider the sequences defined by~{\rm(\ref{example1})} and~{\rm(\ref{example2})}: if it is assumed   that they~R-converge to zero, then clearly~$\Phi_0 \equiv 0$.
  On the other hand  the sequence
\begin{equation}\label{example3}
u_{0,n} (x)= \Phi_0\big(x_1,x_2,\frac {x_3}n \big)
\end{equation}
 is easily seen to R-converge  to zero for any~$\Phi_0$ satisfying~$\Phi_0(x_1,x_2,0) \equiv 0$. }
\end{rmk}
\noindent In this paper we solve the weak stability question under the R-convergence assumption instead of   classical weak convergence.  Actually following Remark~\ref{counterexamples}, the choice of the function space in which to pick the sequence of initial data becomes crucial, as for instance contrary to the examples~(\ref{example1}) and~(\ref{example2}),  the sequence of initial data defined in~(\ref{example3}) is {\it not} bounded in~$B^{-1+\frac3p}_{p,\infty}$ for finite~$p$ (it can actually even be made arbitrarily large in~${\rm BMO}^{-1}$, see~\cite{cg3}). On the other hand it is bounded in anisotropic spaces of the type~$L^2(\R^2;L^\infty(\R))$. We are therefore led to describing sequences of initial data, bounded in anisotropic, homogeneous function spaces.  A celebrated tool to this end are profile decompositions.

\subsection{Profile decompositions and statement of the main result}
The study of the defect of compactness in Sobolev embeddings originates in the works of  P.-L. Lions (see~\cite{pl2cocomp1} and\ccite{pl2cocomp2}), L. Tartar (see~\cite{tartar}) and P. G\'erard (see~\cite{pgerard0}) and earlier decompositions of  bounded sequences into a sum  of ``profiles" can be found in the studies by   H. Br\'ezis and J.-M. Coron in~\cite{BC} and M. Struwe in   \cite{St}. Our source of inspiration here is the work~\cite{pgerard1} of P. G\'erard in which the defect of compactness of the critical Sobolev embedding~$H^s \subset L^p$ is described in terms of a sum of rescaled and translated orthogonal profiles,  up to a small term in~$ L^p$ (see Theorem~\ref{decompoh12} for a statement in the case when~$s = 1/2$). This was generalized to other Sobolev spaces by S. Jaffard in~\cite{jaffard}, to Besov spaces by G. Koch~\cite{koch}, and finally to general critical embeddings by H. Bahouri, A. Cohen and G. Koch in~\cite{bahouricohenkoch} (see also~\cite{BMM, BMM1, BP} for Sobolev embeddings in Orlicz spaces and~\cite{tintarevandco} for an abstract, functional analytic presentation of the concept in various settings). 

\medskip
\noindent In the pionneering works~\cite{bahourigerard} (for the critical 3D wave equation) and~\cite{merlevega} (for the critical 2D Schr\"odinger  equation),  this type of decomposition was introduced in the study of nonlinear partial differential equations. The ideas of~\cite{bahourigerard} were revisited in~\cite{keraani} and~\cite{gallagherNS} in the context of the Schr\"odinger  equations and Navier-Stokes equations respectively, with an aim at describing the structure of bounded sequences of solutions to those equations. These profile decomposition techniques have since then been succesfully used  in order to study the possible blow-up of solutions to   nonlinear partial differential equations,  in various contexts; we refer for instance to~\cite{gkp}, 
\cite{hmidikeraani}, \cite{jiasverak}, \cite{kk}, \cite{km}, \cite{eugenie1}, \cite{rs}.

\medbreak\noindent
Before stating our main result, let us analyze what profile decompositions can say about bounded sequences satisfying the assumptions of Definition~\ref{weakcvunderscaling}. In dimension three, the scale-invariant Sobolev space associated with (NS) is~$H^\frac12(\R^3)$, defined by
$$
\|f\|_{H^\frac12(\R^3)} \eqdefa \Big(\int_{\R^3} |\xi| \, | \widehat f (\xi)|^2 \,d \xi\Big)^\frac12 ,
$$
where~$ \widehat f $ is the Fourier transform of~$f$. The profile decomposition of   P. G\'erard~\cite{pgerard1} describing the lack of compactness of the   embedding~$H^\frac12(\R^3) \subset L^3(\R^3)$ is the following.
\begin{thm}[\cite{pgerard1}] \label{decompoh12}
{\sl Let~$(\varphi_n)_{n \in \N}$ be a  sequence  of functions, bounded in~$H^\frac12(\R^3)$ and converging weakly to some function~$\varphi^0$. Then up to extracting a subsequence (which we denote in the same way),
 there is a family of functions~$(\varphi^j)_{j \geq 1} $  in~$H^\frac12(\R^3)$, and a family~$(x_n^j)_{j \geq 1} $ of sequences of points in~$\R^3$, as well as a family of  sequences of positive real numbers~$(h_n^j)_{j \geq 1} $, 
 orthogonal in the sense that if~$j \neq k$ then
$$
\mbox{either} \quad \frac{h_n^j}{h_n^k} +  \frac{h_n^k}{h_n^j} \to \infty    \quad \mbox{as}\,  \, n \to \infty \, ,  \quad \mbox{or} \quad {h_n^j} = {h_n^k}\quad \mbox{and} \quad \frac{|x_n^k - x_n^j|}{h_n^j} \to \infty   \quad \mbox{as}\,  \,  n \to \infty
$$
such that
 for all integers~$L\geq 1$ the function
$\displaystyle 
\psi_n^L \eqdefa  \varphi_n - \varphi^0  - \sum_{j = 1}^L \Lambda_{h_n^j,x_n^j} \varphi^j
$
satisfies
$$
\limsup_{n \to \infty} \|\psi_n^L \|_{L^3(\R^3)} \to 0  \quad \mbox{as}\,  \, L \to \infty \, .
$$
Moreover one has
\beq\label{weaklimitphij}
\Lambda_{(h_n^j)^{-1},-(h_n^j)^{-1}x_n^j}\varphi_n \rightharpoonup \varphi^j  \, , \quad \mbox{as}\,  \,  n \to \infty\, .
\eeq
}
\end{thm}
\noindent If a  sequence of divergence free vector fields~$u_{0,n}$, bounded in~$H^\frac12(\R^3)$, R-converges   to some vector field~$u_0$  as defined in  Definition~\ref{weakcvunderscaling}, then applying the result~(\ref{weaklimitphij}) of Theorem~\ref{decompoh12} implies that~$\varphi^j$ is identically zero for each~$j$, which in turn implies that
 there are no non zero profiles entering in the decomposition of~$u_{0,n}$. This means that~$\psi_n^L = u_{0,n}- u_0$ and therefore the convergence of~$u_{0,n}$ to~$u_0$ is in fact {\it strong} in~$L^3(\R^3)$. The strong stability result of~\cite{gip} then implies immediately that for~$n$ large enough,~$u_{0,n}$ gives rise to a global unique solution to~(NS) if that is the case for~$u_0$. The same reasoning, using the profile decompositions of~\cite{bahouricohenkoch} and again the strong stability result~\cite{gip}, shows that if~$u_{0,n}$ is bounded in~$B^{-1+\frac3p}_{p,q}(\R^3)$ for finite~$q$, and  R-converges   to some vector field~$u_0$ then as soon as~$u_0$ generates a global smooth solution, then so does~$u_{0,n}$ for~$n$ large enough.

\bigskip 

\noindent In order to obtain a   result
which is not a direct consequence of profile decompositions and known strong stability results,
 the question of the function space in which to choose the initial data becomes a key ingredient in the analysis. 
 As explained in the previous paragraph, one expects that under the R-convergence assumption, a relevant function space in which 
 to choose the initial data should scale like~$L^2(\R^2;L^\infty(\R))$.
  To our knowledge there is no wellposedness result of any kind for (NS) in~$L^2(\R^2;L^\infty(\R))$ so we shall assume some regularity in the third direction, while keeping   the~$L^\infty$ scaling, and this leads us naturally to introducing anisotropic  Besov spaces. These spaces generalize the more usual isotropic Besov spaces, which are studied for instance in~\cite{BCD,bourdaud,RS,triebel,triebel1}.
    \begin{defi}\label{deflpanisointro}
{\sl Let~$\widehat{\chi}$ (the Fourier transform of $\chi$) be a radial
function in~$\mathcal{D}(\mathbb{R})$ such
that~$\widehat\chi(t) = 1$ for~$|t|\leq 1$
and~$\widehat\chi(t) = 0$ for~$|t|>2$.  For~$(j,k) \in \ZZ^2$,  the horizontal truncations are defined by
$$
\begin{aligned}
\widehat {S_k^{\rm h} f}  (\xi)
\eqdefa  \widehat {\chi} \bigl(2^{-k}|(\xi_{1}, \xi_2)|\bigr) \hat f(\xi)
 \andf \Delta_k^{\rm h}  \eqdefa
 S_{k+1}^{\rm h}  - S_{k}^{\rm h} \, ,
 \end{aligned}$$
and the vertical truncations by
$$
\begin{aligned}
\widehat {S_j^{\rm v} f}   \eqdefa  \widehat {\chi} (2^{-j}|\xi_3|) \hat f(\xi)
\andf  \Delta_j^{\rm v}  \eqdefa
 S_{j +1}^{\rm v} - S_{j}^{\rm v}
  \, .
 \end{aligned}$$
  For all~$p$ in~$ [1,\infty]$ and~$q$ in~$ ]0, \infty]$, and all~$(s,s')$ in~$ \R^2, $ with~$s < 2/p, s' < 1/p$ (or~$s \leq 2/p$ and~$ s' \leq 1/p$ if~$q=1$), the anisotropic homogeneous Besov space~$B^{s,s'}_{p,q}$ is
    defined as the space of tempered distributions~$f$ such that
    $$
    \|f\|_{B^{s,s' } _{p,q}}\eqdefa \Big\| 2^{ks + js'}\|\Delta_k^{\rm h} \Delta_{j}^{\rm v} f\|_{L^p} \Big\|_{\ell^q}< \infty \, .
    $$
   In all other cases of indexes~$s$ and~$s'$, the Besov space is defined   similarly, up to  taking the quotient with polynomials. }\end{defi}

\noindent {\bf Notation.}   To avoid heaviness,  we shall  in what follows denote by~$\cB^{s,s'}$ the space~$B^{s,s'}_{2,1}$,  by~$\cB^s$ the space~$\cB^{s,\frac12}$ and by~$  {\mathcal B}_{p,q}$ the space $B^{-1+\frac2p,\frac1p}_{p,q}$.
In particular~$ {\mathcal B}_{2,1} = \cB^0$.

 \bigskip
\noindent Let us  point out that   the scaling operators \eqref{defscalingoperatorintro} enjoy the following invariances:
$$
\begin{aligned}
\|  \Lambda_{  {\lambda},x_0}\varphi\|_{  {\mathcal B}_{p,q}} &= \| \varphi\|_{  {\mathcal B}_{p,q}} \quad  \mbox{and} \\
\forall r \in [1,\infty] \, , \quad \|   \Lambda_{ {\lambda},x_0}\Phi\|_{   {L^r }(\R^+;B^{-1+\frac2p+\frac2r ,\frac1p}_{p,q})} &= \|\Phi\|_{  {L^r }(\R^+; B^{-1+\frac2p+\frac2r ,\frac1p}_{p,q})}  \,,
\end{aligned}
$$
and also the following scaling property:
\begin{equation}
\label{scalingpty}\forall r \in [1,\infty] \, , \, \forall \sigma \in \R \, , \quad \|   \Lambda_{  {\lambda},x_0}\Phi\|_{   {L^r }(\R^+; B^{-1+\frac2p+\frac2r -  \sigma,\frac1p}_{p,q})}  \sim \lambda^\sigma \|\Phi\|_{  {L^r }(\R^+; B^{-1+\frac2p+\frac2r -  \sigma,\frac1p}_{p,q})} \, .
 \end{equation}

 \bigskip
\noindent The Navier-Stokes equations in anisotropic   spaces have been studied in a number of frameworks. We refer for instance, among others,  to~\cite{bg}, \cite{cheminzhang}, \cite{gz}, \cite{dragos}, \cite{paicu}. In particular in~\cite{bg} it is proved that if~$u_0$ belongs to~${\mathcal B}^0$, then there is a unique solution (global in time if the data is small enough) in~$ L^2([0,T]; \cB^1)$. That norm controls the equation, in the sense that as soon as the solution belongs to~$ L^2([0,T]; \cB^1)$, then it lies in fact in~$L^r([0,T]; \cB^{\frac2r})$ for all~$1 \leq r \leq \infty$. The space~$\cB^1$ is included in~$L^\infty$ and since the seminal work\ccite{leray} of J. Leray, it is known that the~$L^2([0,T];L^\infty)$  norm controls the propagation of regularity and also ensures weak  uniqueness among turbulent solutions. Thus the space~$\cB^0$  is   natural in this context.  
 \bigskip
\noindent Our main result is the following.
   \begin{thm}
 \label{mainresult}
 {\sl        Let~$q$ be given in~$]0,1[$ and let~$u_0$ in~$ {\mathcal B}_{1,q}$ generate a unique global solution  to~{\rm(NS)}. Let~$(u_{0,n})_{n \in \N}$ be a   sequence of divergence free vector fields bounded in~${\mathcal B}_{1,q}$, such that~$u_{0,n}$ {\rm R}-converges  to~$u_0$. Then for~$n$ large enough,~$u_{0,n}$ generates a unique, global solution to~{\rm(NS)} in the space~$L^2(\R^+; \cB^1)$.}
 \end{thm}
    
 \noindent {\bf Acknowledgments.} We want to thank very warmly Pierre Germain for suggesting  the concept of rescaled weak convergence.\medskip
   
 %%%%%%%%%%%%%%%%%%%%%%%%%%%%%%%%%%%%%%
 
 %%%%%%%%%%%%%%%%%%%%%%%%%%%%%%%%%%%%%%
 
   \section{Structure  and main ideas of the proof}
 \label{planpreuve}
 To prove Theorem~\ref{mainresult}, the first step consists in the proof of an anisotropic profile decomposition of
 the sequence of initial data. To state the result in a clear
way, let us  start by introducing some definitions and notations.
 \begin{defi}
\label{orthoseq}
{\sl We say that two  sequences of positive real numbers~$\suite {\lam^1} n \N$ and~$\suite {\lam^2} n \N$ are   \textit{{orthogonal}} if
$$
  \begin{aligned}
  \frac{\lambda_n ^{1}}{\lambda_n ^{2}} +  \frac{\lambda_n ^{2}}{\lambda_n ^{1}}
  \to \infty\,
, \quad n \to \infty\, .
\end{aligned}
$$
A family of sequences~$\bigl(\suite {\lam^j} n \N\bigr)_j$ is said to be a  \textit{{family of scales}}  if~$\lam^0_n\equiv 1$ and if~$  \lam^j_n  $ and~$  \lam^k_n  $  are orthogonal when~$j\not = k$.}
\end{defi}
  
\begin{defi}
\label{definitionspacesmu}
 {\sl  Let~$\mu$ be a positive real number less than~$1/2$, fixed from now on.  

\noindent We define~$D_\mu\eqdefa [-2+\mu,1-\mu]\times [1/2,7/2]$ and~$\wt D_\mu\eqdefa[-1+\mu,1-\mu]\times [1/2,3/2]$.  We denote by~$S_\mu$ the space of functions~$a $ belonging to~$ \displaystyle \bigcap_{(s,s')\in D_\mu}  \cB^{s,s'}$ such that
$$
\|a\|_{S_\mu} \eqdefa \sup_{(s,s')\in D_\mu}\|a\|_{\cB^{s,s'}}<\infty \, .
$$ }\end{defi}

    \medskip
  \begin{rmk}\label{rksp}
{\sl 
Everything proved in this paper would work choosing for~$D_\mu$ any set of the type~$  [-2+\mu,1-\mu]\times [1/2,A]$, with $A \geq 7/2$. For simplicity  we limit  ourselves to the case when $A = 7/2$. 
%In the fact, the space $S_\mu$ defined above plays a crucial role as Cauchy data space for the Navier-Stokes system (see Theorem \ref{slowvarsimple} below). 
}
\end{rmk} 
  \medskip
  \noindent {\bf Notation.} For all points~$x = (x_1,x_2,x_3) $ in~$\R^3$ and all   vector fields~$u= (u^1,u^2,u^3)$, we   denote    their horizontal parts by
  $$
  x_{\rm h} \eqdefa (x_1,x_2) \quad \mbox{and} \quad u^{\rm h} \eqdefa (u^1,u^2) \, .
  $$
   We shall  be considering functions which have   different types of variations in the~$x_3$ variable and the~$x_{\rm h}$ variable. The following notation will be used:
$$
 \big[f \big]_\beta (x) \eqdefa f(x_{\rm h},\beta x_3) \, .
$$
Clearly, for any function~$f$, we have the following identity which will be of constant use all along this paper:
\begin{equation}\label{scale}
\bigl\|[f]_\beta\bigr\|_{ B^{s_1,s_2}_{p,1} }\sim\beta^{s_2-\frac1p}\|f\|_{ B^{s_1,s_2}_{p,1}}\,.
\end{equation}

\medbreak

\noindent In all that follows, $\theta$  is  a given function in~$\cD(B_{\R^3}(0,1))$ which has value~$1$ near~$B_{\R^3}(0,1/2)$. For any positive real number~$\eta$, we denote
\beq
\label{defcutof}
\theta_\eta(x)\eqdefa \theta(\eta x)\andf\theta_{{\rm h},\eta} (x_{\rm h})\eqdefa  \theta_\eta(x_{\rm h}, 0) \, .
\eeq
\medbreak
\noindent In order to make  notations as light as possible, the letter~$v$ (possibly with indices) will always denote a {two-component} divergence free vector field,  which may depend on the vertical variable~$x_3$. 
   
 \medskip
 \noindent Finally we define   horizontal differentiation operators~$\nabla^{\rm h} \eqdefa (\partial_1, \partial_2)$ and~$\mbox{div}_{\rm h} \eqdefa \nabla^{\rm h} \cdot  $, as well as~$\Delta_{\rm h} \eqdefa \partial_1^2 + \partial_2^2$, and we shall
 use the following shorthand notation:~$X_{\rm h}  Y_{\rm v} :=X(\R^2;Y(\R))$ where~$X$ is a function space defined on~$\R^2$ and~$Y$  is defined on~$\R$. 
  \begin{thm}
\label{mainresultprofilesdi}
 {\sl Under the assumptions of Theorem~{\rm\ref{mainresult}} and up to the extraction of a subsequence, the following holds.  There is  a family of scales~$\big( ( \lambda^j_n )_{n \in \N}\big)_{j\in \N}$ and for all~$L \geq 1$  there is a family of sequences~$\big( ( h^j_n)_{n \in \N}\big)_{j\in \N}$ {going to zero} such that for any real number~$\alpha$ in~$] 0,1[$ {and for all~$L \geq 1$}, there are  families of sequences of divergence-free vector fields (for~$j$ ranging from~1 to~$L$),~$  (   v^{j}_{n,\alpha,L})_{n \in \N}$,~$( w^{j}_{n,\alpha,L})_{n \in \N}$,~$  ( v_{n,\alpha,L}^{0,\infty})_{n \in \N}$,~$(  w^{0,\infty}_{0,n,\alpha,L})_{n \in \N}$,~$  ( v_{0,n,\alpha,L}^{0,{\rm loc}})_{n \in \N}$ and~$(  w^{0,{\rm loc}}_{0,n,\alpha,L})_{n \in \N}$ all belonging  to~$S_\mu$,  and a smooth, compactly supported function~$u_{0,\alpha}$
 such that the sequence~$  (u_{0,n})_{n \in \N}$ can be written under the form
$$
\displaylines{
  u_{0,n}   \equiv u_{0,\alpha} +  \bigl[\bigl( v_{0,n,\alpha,L}^{0,\infty  }+h_n^0 w_{0,n,\alpha,L}^{0,\infty, {\rm h} },   w_{0,n,\alpha,L}^{0,\infty, 3 }\bigr)  \bigr]_{h_n^0}+\big[(  {v_{0,n,\alpha,L}^{0,{\rm loc}}+ h_{n }^0 w_{0,n,\alpha,L}^{0, {\rm loc},{\rm   h}}},  {w_{0,n,\alpha,L}^{0,{\rm loc},3}})\big]_{  {h_{n}^0}}\cr
  {}+\sum^L _{j = 1}  {\Lambda}_{  \lambda_n^j}  \big[(  {v_{n,\alpha,L}^{j}+ h_{n }^j w_{n,\alpha,L}^{j,{\rm   h}}},  {w_{n,\alpha,L}^{j,3}})\big]_{  {h_{n}^j}}  + { \rho_{n,\alpha,L}}
} 
$$
where~$ u_{0,\alpha}$ approximates~$u_0$ in the sense that
\begin{equation}
\label{uoalphaclosetouo}
 \lim_{\al \to 0} \|u_{0,\alpha} - u_0\|_{\cB_{1,q}} = 0 \, , 
\eeq
where the remainder term satisfies 
%\begin{equation}
%\label{restsmall}
%\lim_{L \to \infty} \lim_{\al \to 0} \limsup_{n \to \infty} \Big( \| e^{t\D}  \rho_{n,\alpha,L }\|_{\widetilde  L^2(\R^+;\cB^1)} 
%+  \|  \rho^3_{n,\alpha,L }\|_{ \cB^0} 
% \Big)= 0 \, , 
%\eeq
\begin{equation}
\label{restsmall}
\lim_{L \to \infty} \lim_{\al \to 0} \limsup_{n \to \infty} \| e^{t\D}  \rho_{n,\alpha,L }\|_{   L^2(\R^+;\cB^1)}  = 0 \, , 
\eeq
while the following uniform bounds hold:
\begin{equation}
\label{orthonorms2}
\begin{aligned}
 \qquad\cM&\eqdefa \sup_{L \geq 1} \sup_{\al \in ]0,1[}\, \sup_{n \in \N}  \Bigl( \bigl\|   (v_{0,n,\alpha,L}^{0,\infty  },w_{0,n,\alpha,L}^{0,\infty, 3 })\bigr \|_{\cB^0} +\bigl\|   (v_{0,n,\alpha,L}^{0,{\rm loc}  },w_{0,n,\alpha,L}^{0,{\rm loc}, 3 })\bigr \|_{\cB^0}\\
&\qquad\qquad \qquad \qquad\qquad \qquad\qquad+  \|u_{0,\alpha}  \|_{\cB^0} +  {\sum^L _{j = 1} }\bigl\|   (v_{n,\alpha,L}^{j},w_{n,\alpha,L}^{j,3})\bigr \|_{\cB^0}\Bigr)<\infty
\end{aligned}
  \end{equation}
and for all~$\al$ in~$ ]0,1[$, 
\begin{equation}
\label{orthonorms22}
\begin{aligned}
  \cM_\al&
  \eqdefa \sup_{L\geq1}\supetage{1\leq j \leq L}{n \in \N}\, \Bigl(\bigl\|   (v_{0,n,\alpha,L}^{0,\infty  },w_{0,n,\alpha,L}^{0,\infty, 3 })\bigr \|_{S_\mu} +\bigl\|   (v_{0,n,\alpha,L}^{0,{\rm loc}  },w_{0,n,\alpha,L}^{0,{\rm loc}, 3 })\bigr \|_{S_\mu}\\
& \qquad\qquad \qquad\qquad\qquad \qquad\qquad \qquad\qquad 
+  \|u_{0,\alpha}  \|_{S_\mu}  +  \big \| (v_{n,\alpha,L}^{j}, w_{n,\alpha,L}^{j,3})\bigr \|_{S_\mu }\Bigr) 
\end{aligned}
  \end{equation}
is finite. Finally, we have
\ben
\label{smallatzerothmdata}
    \lim_{L \to \infty} \, \lim_{\alpha \to 0} \limsup_{n \to \infty}  \big\|\bigl(v_{0,n,\alpha,L}^{0,{\rm loc}},w_{0,n,\alpha,L}^{0,{\rm loc},3}\bigr) (\cdot,0)\bigr\|_{{B}^{0}_{2,1}(\R^2)}  & =  & 0\,,\\
\label{mainresultprofilesdieq-2}
\forall \, (\al,L)\,,\exists \, \eta(\al,L)\,/\ \forall \eta \leq \eta(\al,L) \, ,  \forall n \in \N\,,\  (1-\theta_{{\rm h},\eta}) ( v_{0,n,\alpha,L}^{0,{\rm loc}},  w_{0,n,\alpha,L}^{0,{\rm loc},3} )& =  & 0 \, ,  \,\mbox{and}\\
\label{mainresultprofilesdieq-1}
\forall \,(\al,L,\eta)\,,\ \exists \,n(\al,L,\eta)\,/ \ \forall n\geq  n(\al,L,\eta)\,,\  \theta_{{\rm h},\eta}(v_{0,n,\alpha,L}^{0,\infty  },w_{0,n,\alpha,L}^{0,\infty, 3 }) & =  & 0 \, .
\een
}\end{thm}
\noindent The proof of this theorem is the purpose  of Section\refer{proofprofiledecomposition}. 

\medbreak
\noindent
Theorem \ref{mainresultprofilesdi} states that the sequence~$u_{0,n}$ is equal, up to a small remainder term, to a finite sum of  orthogonal sequences of divergence-free vector fields.  These sequences are obtained from the profile decomposition derived  in~\cite{bg} (see Proposition \ref{prop:decompositiondata0} in this paper) by grouping  together all the profiles having the same horizontal scale $ \lambda_n$, and  the  form they take depends on whether the scale $\lambda_n$ is identically equal  to one or not.  In the case when $\lambda_n$ goes to $0$ or infinity, these sequences are of the type~$ {\Lambda}_{  \lambda_n}  \big[(  {v^{\rm   h}_{n}+ h_{n } w_{n}^{\rm   h}},  {w_{n}^{3}})\big]_{  {h_{n}}} $, with~$ h_n$  a sequence going to zero.  In the case when $\lambda_n$ is identically equal  to one, we deal with three types of orthogonal sequences: the first one consists in~$u_{0,\alpha}$, an approximation of the weak limit $u_0$, the second one given by $\big[(  {v_{0,n,\alpha,L}^{{\rm loc},{\rm h}}+ h_{n }^0w_{0,n,\alpha,L}^{ {\rm loc},{\rm   h}}},  {w_{0,n,\alpha,L}^{{\rm loc},3}})\big]_{  {h_{n}^0}}$  is uniformly localized in the horizontal variable and vanishes at~$x_3 = 0$, while the horizontal support of the third one~$\big[(  {v_{0,n,\alpha,L}^{\infty,{\rm h}}+ h_{n } ^0w_{0,n,\alpha,L}^{\infty, {\rm   h}}},  {w_{0,n,\alpha,L}^{\infty,3}})\big]_{  {h_{n}^0}}$ goes to infinity. 

\medbreak
\noindent
Note that in contrast with classical profile decompositions (as stated in Theorem~\ref{decompoh12} for instance),  cores of concentration do not appear in the profile decomposition given  in Theorem~\ref{mainresultprofilesdi} since all the profiles with the same horizontal scale are grouped together, and thus the decomposition is  written in terms of scales only. The price to pay is that the profiles are no longer fixed functions, but bounded sequences.

\medbreak
\noindent
Let us point out that the {\rm R}-convergence of~$u_{0,n}$ to~$u_0$  arises in a crucial way in the proof of Theorem\refer{mainresultprofilesdi}. It excludes in the profile decomposition of $u_{0,n} $ sequences of type\refeq {example1} and\refeq{example2}. 

%Note also that the free divergence assumption on ~$u_{0,n}$ allows to include the terms of type \eqref{profile1} with $h_n $ tending to infinity in the remainder term. The fact that the profiles that must be considered are only profiles of type \eqref{profile1}  with~$ h_n$  tending to zero is essential to establish our result. 

\medbreak
\noindent
 The choice of the function space~$\cB_{p,q}$ with~$p = 1$ and~$q<1$ for the initial data is due to technical reasons.
Indeed,  the propagation of the profiles by (NS) is  efficient in~$\cB_{p,q}$   only if~$p \leq q$
(see also~\cite{gkp} in the isotropic case). Since the one-dimensional Besov space~$B^\frac1p_{p,q} (\R)$  is an algebra (and a Banach space) only if~$q \leq 1$, this forces the choice~$p =1$, and
 finally for the remainder term to be small in a space with index~$q$ equal to one, we need the original sequence to belong to a space with   index~$q  $ strictly smaller than one.

\medbreak
\noindent
 Once Theorem\refer{mainresultprofilesdi}  is proved, the main step  of the proof of Theorem~\ref{mainresult} consists in proving that each individual profile involved  in the decomposition of Theorem \ref{mainresultprofilesdi} does generate a global solution to (NS) as soon as~$n$ is large enough. This is mainly based on the following results concerning respectively profiles of the type~$ {\Lambda}_{  \lambda_n^j}  \big[(  {v_{n,\alpha,L}^{j}+ h_{n }^j w_{n,\alpha,L}^{j,{\rm   h}}},  {w_{n,\alpha,L}^{j,3}})\big]_{  {h_{n}^j}} $, with $\lambda^j_n$ going to $0$ or infinity, and the profiles of horizontal scale one,  see respectively Theorems~\ref{slowvarsimple} and ~\ref{interactionprofilescale1}. Then, an orthogonality argument leads to the fact that the sum  of the profiles also generates a global regular solution for  large enough~$n$. 
 
 \medskip
\noindent In order to  state the results, let us define the function spaces we shall be working with.

\begin{defi}
\label{definitionspaces}
 {\sl 
 
\noindent -- We define the space~$\cA^{s,s'}=   L^\infty(\R^+;\cB^{s,s'})\cap   L^2(\R^+;\cB^{s+1,s'})$ equipped with the norm
$$
\|a\|_{\cA^{s,s'} }\eqdefa \|a\|_{  L^\infty(\R^+;\cB^{s,s'})}+\|a\|_{   L^2(\R^+;\cB^{s+1,s'})} \, ,
$$
and  we denote~$\cA^s=\cA^{s,\frac12}$.

\smallbreak
\noindent -- We denote by~$\cF^{s,s'}$ any function space such that
$$
  \|
 L_0 f 
 \|_{ L^2(\R^+;\cB^{s+1,s'})} \lesssim \|f\|_{\cF^{s,s'}}
$$
where, for any non negative real number~$\tau$,~$ L_{\tau} f$ is the solution of~$\partial_t  L_{\tau} f-\D L_{\tau} f = f $ with~$L_{\tau} f_{|t=\tau}=0$. We denote~$\cF^s=\cF^{s,\frac12}$.

}\end{defi}
\medskip
 \noindent{\bf Examples.}\label{examplespage} Using the smoothing effect of the heat  flow as described by Lemma\refer{anisoheat}, it is  easy  to prove that the spaces 
 %$\wt L^2(\R^+;\cB^{s-1,s'})$,~$\wt L^2(\R^+;\cB^{s,s'-1})$,
 $  L^1(\R^+;\cB^{s,s'})$ and~$ L^1(\R^+;\cB^{s+1,s'-1})$   are continuously embedded in~$\cF^{s,s'}$. We refer to Lemma~\ref{example} for a proof, along with other examples.
\medbreak
\noindent In the following we shall designate by~$\cT_0 (A,B)$ a generic constant depending only on the quantities~$A$ and~$B$. 
 We shall denote  by~$\cT_1$ a  generic non decreasing  function from~$\R^+$ into~$\R^+$ such that 
\beq
\label{petitogeneral}
 \limsup_{r\rightarrow0}\, \frac {\cT_1(r)} r <\infty \,  ,
\eeq
and by~$\cT_2$ a generic locally bounded function from~$\R^+$ into~$\R^+$. All those functions  may vary from line to line. Let us notice that  for any  {positive sequence~$\suite a n \N$ belonging to~$\ell^1$},  we have
\beq
\label{petitogeneralseries}
\sum_n \cT_1(a_n) \leq \cT_2\Big(\sum_n a_n\Bigr) \, .
\eeq
The notation~$a\lesssim b$ means that  an absolute constant~$C$ exists such that~$a\leq Cb$.

\begin{thm}
\label{slowvarsimple}
{\sl  A  locally bounded function~$\e_1$ from~$\R^+$ into~$\R^+$ exists which satisfies the following.  For any~$(v_0, w^3_0)$ in~$S_\mu$ (see Definition~{\rm\ref{definitionspacesmu}}),  for any  positive real number~${\beta}$ such that~$\b\leq \e_1(\|(v_0, w^3_0)\|_{S_\mu})$,  the divergence free vector field 
 $$
\Phi_0  \eqdefa \big[  (  v_0-\b\nabla^{\rm h} \Delta_{\rm h}^{-1} \partial_3w^3_0, w^3_0)  \big]_\b
$$
 generates a global solution~$\Phi_\b$ to {\rm (NS)} which satisfies
 \beq
 \label{estimatesolutionun}
 \|\Phi _\b\|_{\cA^0}  \leq  \cT_1(\|(v_0, w^3_0)\|_{\cB^0}) +\b \,\cT_2(\|(v_0, w^3_0)\|_{S_\mu})\,.
 \eeq
 Moreover, for any~$(s,s') $ in~$[-1+\mu,1-\mu]\times [1/2,7/2]$, we have, for any~$r$ in~ $[1,\infty] $,
\beq
 \label{estimatesolutionunbis}
   \|\Phi_\b \|_{ L^r(\R^+; \cB^{s+\frac2r}) }  + \frac 1{\beta^{s'-\frac12}}    \|\Phi_\b \|_{ L^r(\R^+; \cB^{\frac 2r,s'}) }   \leq   \cT_2(\|(v_0, w^3_0)\|_{S_\mu})  \, .
\eeq
}\end{thm}
\noindent The proof of this theorem is the purpose of Section\refer{propagationprofiles}. Let us point out that this theorem is a result of  global existence  for the Navier-Stokes system associated to a new class of arbitrarily large initial data generalizing the example consider in \ccite{cg3}, and  where the regularity is sharply estimated, in particular in   anisotropic  norms.

\medbreak

\noindent The existence of a global regular solution  for the set of profiles associated with the  horizontal scale~$1$  is ensured by the following theorem.
\begin{thm}
\label{interactionprofilescale1}
{\sl Let us consider the initial data, with the notation of Theorem~{\rm\ref{mainresultprofilesdi}},
$$
\Phi^0_{0,n,\alpha,L} \eqdefa u_{0,\alpha} +  \bigl[\bigl( v_{0,n,\alpha,L}^{0,\infty  }+h_n^0 w_{0,n,\alpha,L}^{0,\infty, {\rm h} },   w_{0,n,\alpha,L}^{0,\infty, 3 }\bigr)  \bigr]_{h_n^0}+  \big[(  {v_{0,n,\alpha,L}^{0,\rm{loc}}+ h_{n }^0 w_{0,n,\alpha,L}^{0,\rm{loc},{\rm   h}}},  {w_{0,n,\alpha,L}^{0,\rm{loc},3}})\big]_{  {h_{n}^0}} \, .
$$
There is a constant~$\e_0$, depending only on~$u_0$ and on~$\cM_\al$, such that if~$h^0_n\leq \e_0$, then the initial data~$\Phi^0_{0,n,\alpha,L} $ generates a global smooth solution~$\Phi^0_{n,\alpha,L}$ which satisfies for all~$s$ in~$ [-1+\mu,1-\mu]$  and all~$ r$ in~$ [1,\infty] $,
\ben
\label{interactionprofilescale1eq1}
 \|\Phi^0_{n,\alpha,L} \|_{ L^r(\R^+; \cB^{s+\frac2r})} \leq \cT_0(u_0,\cM_\al) \, .\een
}\end{thm}

\medbreak

\noindent The proof of this theorem is the object of Section\refer{interactionprofiles1case}. As Theorem~\ref{slowvarsimple}, this is also a   global existence result for the Navier-Stokes system,  generalizing  Theorem~3 of\ccite{cgm} and Theorem~2 of\ccite{cgz}, where we control regularity in a very precise way.

\medbreak
\noindent
{\it Proof of Theorem{\rm\refer{mainresult}}.} Let us consider the profile decomposition given by Theorem\refer{mainresultprofilesdi}. For a given positive (and small)~$\e$,   Assertion\refeq{restsmall} allows to choose~$\al$, $L$ and~$N_0$ (depending of course on~$\e$) such that 
 \beq
 \label{mainresultdemoeq1}
 \forall n\geq N_0\,,\  \| e^{t\D}  \rho_{n,\alpha,L }\|_{  L^2(\R^+;\cB^1)}\leq \e \, .
 \eeq 
 {From now on the parameters~$\al$ and~$L$ are fixed so that~(\ref{mainresultdemoeq1}) holds.}
 Now let us consider {the} two functions~$\e_1$,~$\cT_1$ and~$\cT_2$  (resp.~$\e_0$ and~$\cT_0$) which appear in the statement of Theorem~\ref{slowvarsimple} (resp. Theorem\refer{interactionprofilescale1}).  {Since each sequence~$\suite {h^j}n \N$, for~$0 \leq j \leq L, $ goes to zero as~$n$ goes to infinity,} let us choose an integer~$N_1$ greater than or equal to~$N_0$ such that 
 \beq
 \label{mainresultdemoeq2}
 \forall n\geq N_1\,,\ \forall j\in \{0,\dots,L\}\,,\ h_n^j \leq \min\Bigl\{ \e_1(\cM_\al), \e_0 , { \frac {\e} {L \cT_2(\cM_\al) }} \Bigr\} \, \cdotp 
 \eeq
{Then} for~$1 \leq j\leq L$ (resp.~$j=0$), let us denote  by~$\Phi_{n,\e}^j$ (resp.~$\Phi_{n,\e}^0$) the global solution of~(NS) associated with the initial data
 \begin{equation*}
 \label{mainresultdemoeq2000}
\begin{split}
 &\qquad\qquad\qquad\qquad\qquad\quad \big[(  {v_{n,\alpha,L}^{j}+ h_{n }^j w_{n,\alpha,L}^{j,{\rm   h}}},  {w_{n,\alpha,L}^{j,3}})\big]_{  {h_{n}^j}}  \\
&  \Bigl({\rm resp.} \quad   u_{0,\alpha} + \bigl[\bigl( v_{0,n,\alpha,L}^{0,\infty  }+h_n^0 w_{0,n,\alpha,L}^{0,\infty, {\rm h} },   w_{0,n,\alpha,L}^{0,\infty, 3 }\bigr)  \bigr]_{h_n^0}+
  \big[(  {v_{0,n,\alpha,L}^{0,{\rm loc}}+ h_{n }^0 w_{0,n,\alpha,L}^{0,{\rm loc},{\rm   h}}},  {w_{0,n,\alpha,L}^{0,{\rm loc},3}})\big]_{  {h_{n}^0}}\Bigr) 
 \end{split}
 \end{equation*}
 given by Theorem\refer{slowvarsimple} (resp. Theorem\refer{interactionprofilescale1}).
  We  {look for} the global solution associated with~$u_{0,n} $ {under} the form
$$
u_n = u_{n,\e}^{\rm app} +R_{n,\e} \with  u_{n,\e}^{\rm app} \eqdefa \sum_{j=0}^L  \Lam_{\lam_n^j}\Phi^j_{n,\e}+e^{t\D} { \rho_{n,\alpha,L}}  \, ,
$$
{recalling that~$\lam_n^0 \equiv 1$, see Definition~\ref{orthoseq}.} As recalled in the introduction,~$  \Lam_{\lam_n^j}\Phi^j_{n,\e}$ solves (NS) with the initial data~$ \Lam_{\lam_n^j} \big[(  {v_{n,\alpha,L}^{j}+ h_{n }^j w_{n,\alpha,L}^{j,{\rm   h}}},  {w_{n,\alpha,L}^{j,3}})\big]_{  {h_{n}^j}} $   by scaling invariance of the Navier-Stokes equations. Plugging this decomposition into the Navier-Stokes equation therefore gives the following equation on~$R_{n,\eps}$:  
\beq
\label{conclusproftheo1eq1}
\begin{split}
&\partial_t R_{n,\e}- \D R_{n,\e}  +\dive \bigl(R_{n,\e}\otimes R_{n,\e}+R_{n,\e}\otimes u_{n,\e}^{\rm app}+u_{n,\e}^{\rm app}\otimes R_{n,\e}\bigr) +\nabla p_{n,\e}\\
&\qquad\qquad\qquad\qquad\qquad\qquad\qquad =F_{n,\e}  = F^{1}_{n,\e}+F^{2}_{n,\e}+F^{3}_{n,\e} \with\\
& F^{1}_{n,\e} \eqdefa \dive  \bigl(e^{t\D}\rho_{n,\al,L} \otimes e^{t\D}\rho_{n,\al,L}\bigr)\\
&F^{2}_{n,\e} \eqdefa
 \sum_{j =0}^L\dive \big(   \Lambda_{  \lambda_n^j }  \Phi^{j}_{n,\e}\otimes  e^{t\D}\rho_{n,\al,L}    +     e^{t\D}\rho_{n,\al,L} \otimes      \Lambda_{  \lambda_n^j } \Phi^j_{n,\e}\big)\andf \\
 &F^{3}_{n,\e} \eqdefa \sumetage{0\leq j,k\leq L}{j \neq k}  \mbox{div} \big(   \Lambda_{  \lambda_n^j }  \Phi^{j}_{n,\e}\otimes    \Lambda_{  \lambda_n^k }  \Phi^k_{n,\e}\big) \, ,\end{split}
\eeq
and where $\ds \big(\mbox{div} (u\otimes v)\big)^j= \sum_{k=1}^3\partial_{k}(u^jv^k)$. \\
We shall prove that there is an integer~$N \geq N_1$ such that with the notation of Definition~\ref{definitionspaces},
\beq \label{conclusproftheo1eq4}
\forall n \geq N  \, , \quad \|F_{n,\e} \|_{\cF^0} \leq C \e \, ,
\eeq
where~$C$ only depends on~$L$ and~$\cM_\al$. In the next estimates we omit the dependence of all constants on~$\al$ and~$L$, which are fixed.

\medskip
\noindent Let us start with the estimate of~$F^{1}_{n,\e}$.  
%We decompose~$F^{1}_{n,\e}$ into two parts, writing      
%$$
%\begin{aligned}
%F^{1}_{n,\e} &=  F^{1,1}_{n,\e} + F^{1,2}_{n,\e}\quad \mbox{with} \\
%F^{1,1}_{n,\e} & \eqdefa  \dive_{\rm h}  \bigl(e^{t\D}\rho^{\rm h}_{n,\al,L} \otimes e^{t\D}\rho_{n,\al,L}\bigr) \, \\
%F^{1,12}_{n,\e} & \eqdefa  \partial_3   \bigl(e^{t\D}\rho^{3}_{n,\al,L} \otimes e^{t\D}\rho_{n,\al,L}\bigr) \, .
%\end{aligned}
%$$
 Using the fact that~$\cB^1$ is an algebra,  we have
$$
\bigl\|e^{t\D}\rho^{\rm h}_{n,\al,L}  \otimes e^{t\D}\rho_{n,\al,L} \bigr\|_{L^1(\R^+;\cB^1)}
\lesssim \|e^{t\D}\rho_{n,\al,L}\bigr\|_{  L^2(\R^+;\cB^1)}^2 \, ,
$$
so
$$
\|  \dive_{\rm h}  \bigl(e^{t\D}\rho^{\rm h}_{n,\al,L} \otimes e^{t\D}\rho_{n,\al,L}\bigr) \|_{L^1(\R^+;\cB^0)} \lesssim  \|e^{t\D}\rho_{n,\al,L}\bigr\|_{  L^2(\R^+;\cB^1)}^2
$$
and
$$
\| \partial_3   \bigl(e^{t\D}\rho^{3}_{n,\al,L}   e^{t\D}\rho_{n,\al,L}\bigr)\|_{L^1(\R^+;\cB^{1,-\frac12})}\lesssim  \|e^{t\D}\rho_{n,\al,L}\bigr\|_{  L^2(\R^+;\cB^1)}^2 \, .
$$
According to the examples page~\pageref{examplespage}, we infer that 
\beq
\label{lemmeF01L1B1}
\|F^{1}_{n,\e}\|_{\cF^0} \lesssim \|e^{t\D}\rho_{n,\al,L}\bigr\|_{  L^2(\R^+;\cB^1)}^2
 \, .
\eeq
%On the other hand, product laws given in Proposition~\ref{productlawsaniso} imply that
%$$
%\begin{aligned}
%\bigl\|e^{t\D}\rho^{3}_{n,\al,L}  \otimes e^{t\D}\rho_{n,\al,L} \bigr\|_{\wt L^2(\R^+;\cB^0)}
%& \lesssim \|e^{t\D}\rho^3_{n,\al,L}\bigr\|_{ \wt L^\infty(\R^+;\cB^0)}\|e^{t\D}\rho_{n,\al,L}\bigr\|_{  L^2(\R^+;\cB^1)}\\
%& \lesssim \| \rho^3_{n,\al,L}\bigr\|_{  \cB^0}\|e^{t\D}\rho_{n,\al,L}\bigr\|_{  L^2(\R^+;\cB^1)}
% \, ,
%\end{aligned}
%$$
%so
%$$
%\| F^{1,2}_{n,\e}\|_{\wt L^2(\R^+;\cB^{0,-\frac12})}
% \lesssim  \| \rho^3_{n,\al,L}\bigr\|_{  \cB^0}\|e^{t\D}\rho_{n,\al,L}\bigr\|_{  L^2(\R^+;\cB^1)} \, .
%$$
%Again thanks to the examples page~\pageref{definitionspaces}, we find that 
%\beq
%\label{lemmeF02L1B1}
%\|F^{1,2}_{n,\e}\|_{\cF^0} \lesssim  \| \rho^3_{n,\al,L}\bigr\|_{  \cB^0}\|e^{t\D}\rho_{n,\al,L}\bigr\|_{  L^2(\R^+;\cB^1)} \, .  
%\eeq
In view of  Inequality\refeq{mainresultdemoeq1}, Estimate~(\ref{lemmeF01L1B1})  ensures  that 
\beq
\label{conclusproftheo1eq2}
\forall n \geq N_1 \, , \quad\|F^{1}_{n,\e} \|_{\cF^0} \lesssim \e^2.
\eeq
Now let us consider~$F^{2}_{n,\e}$.
By  the scaling invariance of the  operators~$  \Lam_{\lam_n^j}$ in~$L^2(\R^+;\cB^1)$ and again the fact that~$\cB^1$ is an algebra, we get
\begin{equation}
\label{prodPhirho}
\begin{aligned}
\bigl\| {    \Lambda_{  \lambda_n^j } } \Phi^{j}_{n,\e}\otimes  e^{t\D}\rho_{n,\al,L}   & +     e^{t\D}\rho_{n,\al,L} \otimes      \Lambda_{  \lambda_n^j } \Phi^j_{n,\e}\bigr\|_{L^1(\R^+;\cB^1)} 
\\
&\quad {}\lesssim\| \Phi^{j}_{n,\e}\|_{  L^2(\R^+;\cB^1)} \|  e^{t\D}\rho_{n,\al,L}\|_{  L^2(\R^+;\cB^1)} \, .
\end{aligned}
\end{equation}
{Next we  write, thanks to Estimates~(\ref{estimatesolutionun}) and \eqref{interactionprofilescale1eq1},
$$ 
\longformule{
\sum_{j=0}^L  \big\| \Phi^j_{n,\e} \big\|_{  L^2(\R^+;\cB^1)} \leq   \cT_0(u_0,\cM_\al)
}
{ {} + \sum_{j=1}^L \Big(\cT_1\bigl(\|(v_{n,\al,L}^j,w_{n,\al,L}^{j,3})\|_{\cB^0}\bigr) + h_n^j   \cT_2\bigl(\|(v_{n,\al,L}^{j},w_{n,\al,L}^{j,3})\|_{S_\mu}\bigr) \Big),
}$$ 
which can be written due to\refeq{petitogeneralseries}
$$
\sum_{j=0}^L  \bigl\| \Phi^j_{n,\e} \bigr\|_{  L^2(\R^+;\cB^1)}\leq  \cT_0(u_0,\cM_\al)+\cT_2(\cM) +\sum_{j=1}^L  h_n^j   \cT_2(\cM_\alpha) \, .
$$

\noindent Using Condition\refeq{mainresultdemoeq2} on the sequences~$\suite {h^j} n \N$ implies that 
$$
\Bigl\|\sum_{j=0}^L  \, \Phi^j_{n,\e} \Bigr\|_{  L^2(\R^+;\cB^1)} \leq
\cT_0(u_0,\cM_\al)+\cT_2(\cM)  +\e \, .
$$
 It follows  (of course up to a change of~$\cT_2$) that for small enough~$\e$ 
\beq
\label{conclusproftheo1eq20}
\Bigl\| \sum_{j=0}^L  \,  \Phi^j_{n,\e}  \Bigr\|_{  L^2(\R^+;\cB^1)} \leq
\cT_0(u_0,\cM_\al)+\cT_2(\cM)\, .
\eeq
Thanks to~(\ref{mainresultdemoeq1}) and~(\ref{prodPhirho}), this gives rise to 
\beq
\label{conclusproftheo1eq3}
\forall n \geq N_1 \, , \quad\|F^{2}_{n,\e} \|_{\cF^0} \leq \e\,  \big(\cT_0(u_0,\cM_\al)+\cT_2(\cM)  \big)
 \, .
\eeq
 Finally let us consider~$F_{n,\e}^3$.  Recalling that~$\al$ and~$L$ are fixed, it suffices to prove in view of the examples page~\pageref{examplespage} that
 there is~$N_2 \geq N_1$ such that for all~$n \geq N_2$ and for all~$  0 \leq j \neq k \leq L$,
$$
 \big \| \Lambda_{  \lambda_n^j}  \Phi^j_{n,\e}\otimes   \Lambda_{  \lambda_n^k}  \Phi^k_{n,\e}\big\|_{  {  L^1}(\R^+;  \cB^{1})} \lesssim \e 
\, .
$$
Using the fact that~$ \cB^{1}$ is an algebra along with the H\"older inequality, we infer that for a small enough~$\gamma$ in~$]0,1[$,
$$
 \big \| \Lambda_{  \lambda_n^j}  \Phi^j_{n,\e}\otimes   \Lambda_{  \lambda_n^k}  \Phi^k_{n,\e}\big\|_{ {  L^1}(\R^+;  \cB^{1})}
  \leq \| \Lambda_{  \lambda_n^j}  \Phi^j_{n,\e}\|_{   L^{\frac 2 {1+ \gamma} }(\R^+;  \cB^{1})}
   \| \Lambda_{  \lambda_n^k}  \Phi^k_{n,\e}\|_{    L^{\frac 2 {1- \gamma}}(\R^+;  \cB^{1 })} \, .
$$
The scaling invariance \eqref{scalingpty} gives
\beno
 \| \Lambda_{  \lambda_n^j}  \Phi^j_{n,\e}\|_{   L^{\frac 2 {1+ \gamma} }(\R^+;  \cB^{1 })} & \sim  & ( \lambda_n^j )^\gamma\| \Phi^j_{n,\e}\|_{   L^{\frac 2 {1+ \gamma} }(\R^+;  \cB^{1})}\andf\\
  \| \Lambda_{  \lambda_n^k}  \Phi^k_{n,\e}\|_{   L^{\frac 2 {1- \gamma} }(\R^+;  \cB^{1})} & \sim &  \frac 1 {(\lambda_n^k)^\gamma} \| \Phi^k_{n,\e}\|_{   L^{\frac 2 {1- \gamma} }(\R^+;  \cB^{1 })} \, .
 \eeno
For small enough~$\gamma$, Theorems \refer{slowvarsimple}  and \ref{interactionprofilescale1} imply that
$$
 \big \| \Lambda_{  \lambda_n^j}  \Phi^j_{n,\e}\otimes   \Lambda_{  \lambda_n^k}  \Phi^k_{n,\e}\big\|_{  {  L^1}(\R^+;\cB^1)}
 \lesssim  \Bigl( \frac{{\lambda_n^j}}{\lambda_n^k}\Bigr)^\gamma \, \cdotp
$$
We deduce that 
$$
\|F^{3}_{n,\e}\|_{\cF^0}  \lesssim \sumetage{0\leq j,k\leq L}{j \neq k}  \min  \Bigl\{  \frac{{\lambda_n^j}}{\lambda_n^k}\virgp  \frac{{\lambda_n^k}}{\lambda_n^j}\Bigr\}^{\gamma} \, .
$$
As the sequences~$\suite {\lam^j} n \N$ and~$\suite {\lam^k} n \N$ are orthogonal (see Definition \ref{orthoseq}), we have for any~$j$ and~$k$ such that~$j\not =  k$
$$
\lim_{n\to \infty}  \min  \Bigl\{  \frac{{\lambda_n^j}}{\lambda_n^k}\virgp  \frac{{\lambda_n^k}}{\lambda_n^j}\Bigr\}=0 \, .
$$
Thus an integer~$N_2$ greater than or equal to~$N_1$ exists such that
$$
\forall n \geq N_2 \, , \quad \|F^{3}_{n,\e}\|_{\cF^0} \lesssim \e \, .
$$ 
Together with\refeq{conclusproftheo1eq2} and\refeq{conclusproftheo1eq3}, this implies that 
$$
n\geq N_2\Longrightarrow  \|F_{n,\e} \|_{\cF^0} \lesssim \e \, ,
$$
which proves~(\ref{conclusproftheo1eq4}).
\bigbreak
\noindent Now, in order to conclude the proof of Theorem\refer{mainresult}, we need the following result.
\begin{prop}
\label{existencepetitcB1}
{\sl A constant~$C_0$ exists such that, if~$U$ is in~$L^2(\R^+;\cB^1)$,~$u_0$ in~$\cB^0$ and~$f$ in~$\cF^0$ such that 
$$
\|u_0\|_{\cB^0} +\|f\|_{\cF^0} \leq \frac 1{C_0} \exp \Bigl(-C_0\int_0^\infty \|U(t)\|^2_{\cB^{1}} dt\Bigr) \, ,
$$
then the problem
$$
({\rm {NS}}_U)\quad 
\quad  \left\{
\begin{array}{c}
\partial_t u + \dive (u\otimes u+u\otimes U+U\otimes u)  -\Delta u= -\nabla p +f \\
\dive u= 0\andf u_{|t=0} = u_{0}
\end{array}
\right.
$$
has a unique global solution in~$  L^2(\R^+;\cB^1)$  which satisfies
$$
\|u\|_{L^2(\R^+;\cB^1)} \lesssim \|u_0\|_{\cB^0} +\|f\|_{\cF^0}\,.
$$}
\end{prop}
\noindent The proof of this proposition can be found in Section\refer{lpanisodivers}.

\medbreak

\noindent{\it Conclusion of the proof of Theorem{\rm\refer{mainresult}}. }
By definition of $u^{\rm app}_{n,\e}$ we have
$$
\|u^{\rm app}_{n,\e}\|_{  L^2(\R^+;\cB^1)} \leq  \Bigl\|\sum_{j=0}^L  \Lam_{\lam_n^j}\Phi^j_{n,\e} \Bigr\|_{  L^2(\R^+;\cB^1)}+\|e^{t\D} \rho_{n,\al,L}\|_{  L^2(\R^+;\cB^1)} \, .
$$
Inequalities\refeq{mainresultdemoeq1} and\refeq{conclusproftheo1eq20} imply that for $n$ sufficiently large 
$$
\|u^{\rm app}_{n,\e}\|_{  L^2(\R^+;\cB^1)} \leq \cT_0(u_0,\cM_\al)+\cT_2(\cM) +C\e\,.
$$
Because of\refeq{conclusproftheo1eq4}, it is clear that, if~$\e$ is small enough, 
$$
\|F_{n,\e} \|_{\cF^0} \leq \frac 1 {C_0} \exp\Bigl (-C_0 \|u^{\rm app}_{n,\e}\|_{  L^2(\R^+;\cB^1)}^2\Bigr)
$$
which ensures that~$u_{0,n}$ generates a global regular solution and thus concludes the proof of Theorem\refer{mainresult}.
 \qed

\bigskip
\noindent The paper is structured as follows. In Section~\ref{proofprofiledecomposition} we prove Theorem\refer{mainresultprofilesdi}. Theorems~\ref{slowvarsimple} and~\ref{interactionprofilescale1} are proved in Sections~\ref{propagationprofiles} and~\ref{interactionprofiles1case} respectively. Section~\ref{lpanisodivers} is devoted to the recollection of some material on anisotropic Besov spaces. We also prove  Proposition~\ref{existencepetitcB1} and  an anisotropic propagation of regularity result for the Navier-Stokes system (Proposition~\ref{propaganisoNS3D}).

%%%%%%%%%%%%%%%%%%%%%%%%%%%%%%%%%%%%%%%%%%%%%%%%%%%%%%%%%%

\section{Profile decomposition of the sequence of initial data:  proof of Theorem\refer{mainresultprofilesdi}}
 \label{proofprofiledecomposition}
 The proof of Theorem\refer{mainresultprofilesdi} is structured as follows. First, in Section~\ref{divfreeprofile} we write down the profile decomposition of any bounded sequence of divergence free vector fields R-converging to zero, following the results of~\cite{bg}. Next we reorganize the profile decomposition by  grouping together all profiles having the same horizontal scale and we check that all the conclusions of Theorem~\ref{mainresultprofilesdi} hold: that is performed in Section~\ref{reorganizationprofiles}.

  \subsection{Profile decomposition of divergence free vector fields, R-converging to zero}
 \label{divfreeprofile}
 \noindent
 In this section we start by recalling the result of~\cite{bg}, where an
anisotropic profile decomposition of   sequences of~$\cB_{1,q}$ is introduced. Then we use the assumption of R-convergence  (see Definition~\ref{weakcvunderscaling}) to eliminate from the profile decomposition all isotropic profiles. Finally we study the particular case of divergence free vector fields. Under this assumption, we are able  to restrict our attention to (rescaled) vector fields with slow vertical variations. 

\subsubsection{The case of bounded sequences}
 Before stating the result proved in~\cite{bg}, let us give the definition of anisotropic scaling operators: for any two sequences of positive real numbers~$(\e_n)_{n \in \N}$ and~$(\gamma_n)_{n \in \N}$, and for any
 sequence~$(x_n)_{n \in \N}$ of points in~$\R^3$, we denote
   $$\begin{aligned}
\Lambda_{  {\e_n}, \gamma_n,x_n} \phi (x)& \eqdefa
 \frac1 {\varepsilon_n}  \phi\left( \frac{x_{\rm h}-x_{n,{\rm h}} }{\varepsilon_n}\virgp \frac{x_3-x_{n,3}} {\gamma_n} \right)\cdotp
  \end{aligned}
$$
Observe that the operator~$ \Lambda_{  {\e_n}, \gamma_n,x_n}$   is an isometry in the space~$\cB_{p,q}$ for any~$1 \leq p \leq \infty$ and any~$0<q \leq \infty$.
Notice also that when the sequences~$\e_n$ and~$\gamma_n$ are equal, then the operator~$\Lambda_{  {\e_n}, \gamma_n,x_n}$ reduces to the isotropic scaling operator~$\Lambda_{  {\e_n},x_n}$ defined in~(\ref{defscalingoperatorintro}), and   such isotropic profiles will be the ones to disappear in the profile decomposition thanks to the assumption of R-convergence.
We also have a definition of orthogonal triplets of sequences, analogous to Definition~\ref{orthoseq}.
\begin{defi}\label{orthoseqaniso}
{\sl We say that two triplets of sequences~$( {\varepsilon_n^\ell} , {\gamma_n^\ell}, {x_n^\ell}  )_{n \in \N}$ with~$\ell$ belonging to~$ \{1,2\}$, where~$( {\varepsilon_n^\ell} , {\gamma_n^\ell})_{n \in \N}$ are two sequences  of positive real numbers and ${x_n^\ell} $  are  sequences  in~$ \R^3$,   are   \textit{{orthogonal}} if, when~$n$ tends to infinity,
$$
\displaylines{
 \mbox{either} \quad \frac{\varepsilon_n^{1}}{\varepsilon_n^{2}} +  \frac{\varepsilon_n^{2}}{\varepsilon_n^{1}}
  +  \frac{\gamma_n^{1}}{\gamma_n^{2}} +  \frac{\gamma_n^{2}}{\gamma_n^{1}}
  \to \infty\,
  \cr
  \mbox{or}\quad  ( {\varepsilon_n^1} , {\gamma_n^1})\equiv ( {\varepsilon_n^2} , {\gamma_n^2})\quad \mbox{and} \quad
  | (x_n^1)^{\varepsilon_n^1,\gamma_n^1} - (x_n^2)^{\varepsilon_n^1,\gamma_n^1} |   \to \infty\,
, 
}
  $$
  where we have denoted
    $\displaystyle
 (  {x_n^\ell} )^{ {\varepsilon_n^k} , {\gamma_n^k}} \eqdefa \Big(
 \frac{  {x_{n,\rm h}^{\ell}}}{ {\varepsilon_n^k} } \virgp
  \frac{  {x_{n,3}^{\ell}}}{ {\gamma_n^k} }
 \Big) \cdotp
  $     
   } 
   \end{defi}
\noindent We recall without proof the following result. 
 \begin{prop}[\cite{bg}]\label{prop:decompositiondata0}
{\sl  Let~$(\varphi_{n})_{n \in \N}$ be a sequence  belonging  to~$\cB_{1,q}$ for some~$0<q \leq 1$, with~$\varphi_{n}$ converging weakly to~$\phi^0$ in~$ \cB_{1,q}$   as~$n$ goes to infinity. 
For all integers~$\ell \geq 1$ there is a triplet of   orthogonal sequences in the sense of Definition~{\rm \ref{orthoseqaniso}}, denoted by~$( {\varepsilon_n ^\ell} , {\gamma_n^\ell}, {x_n^\ell})_{n \in \N}$  and    functions~$    \phi^{\ell} $  in~$\cB_{1,q}$ such that up to extracting a subsequence,   one can write the sequence~$  (\varphi_{n} )_{n \in \N}$ under the following form, for each~$L \geq 1$:
\begin{equation}\label{decompositiondecompositiondata0}
\varphi_{n} = \phi^{0}  +  \sum_{
   \ell = 1}^{L }  \Lambda_{ {\varepsilon_n^{\ell}} , {\gamma_n^{\ell}}, {x_n^{\ell}} }   \phi^{ \ell}   + {\psi_n^{L} }   \, , 
 \end{equation}
 where~$\psi_n^L$  satisfies
  \begin{equation}
  \label{smallremainderpsi}
\limsup_{n \to \infty}   \| \psi_n^{L}\|_{\cB^0 } \to 0\, , \quad L \to \infty \, .
\end{equation}
 Moreover
 the following stability result
 holds:
    \begin{equation}\label{ortnorms}
   \sum_{\ell \geq 1}    \|   \phi^{\ell}  \|_{{\mathcal B}_{1,q}}   \lesssim \sup_n \|\varphi_{n} \|_{{\mathcal B}_{1,q}} + \|\varphi_{0} \|_{{\mathcal B}_{1,q}} \, .
  \end{equation}
 }\end{prop}
\begin{rmk}\label{equalscales}
{\sl As pointed out in~\cite[Section~2]{bg}, if two scales appearing in the above decomposition are not orthogonal, then they can be chosen to be equal. We shall therefore assume from now on   that is the case: two sequences of scales are either orthogonal, or equal.
}\end{rmk}
\begin{rmk}\label{testfunctionrk}
{\sl  By density of smooth, compactly supported functions in~$\cB_{1,q}$, one can write
$$
\displaystyle  \phi^{\ell}  =  \phi_\alpha^{\ell} + r_\alpha^{\ell} 
 \quad \mbox{with} \quad
\| r_\alpha^{\ell} \|_{{\mathcal B}_{1,q}} \leq \alpha
$$
 where~$   \phi_\alpha^{\ell} $ are arbitrarily smooth and compactly supported, and  moreover
 \begin{equation}\label{orthonormsbis}
   \sum_{\ell \geq 1} \big(   \|   \phi_\alpha^{\ell}  \|_{{\mathcal B}_{1,q}} +   \|   r_\alpha^{\ell}  \|_{{\mathcal B}_{1,q}}  \big) \lesssim \sup_n \|\varphi_{n} \|_{{\mathcal B}_{1,q}} + \|\varphi_{0} \|_{{\mathcal B}_{1,q}} 
  \, .\end{equation}
}  \end{rmk}
\noindent Next we consider the particular case when~$\varphi_{n}$ R-converges to~$\phi^0$, in the sense of Definition~\ref{weakcvunderscaling}. Let us prove the following result.
 \begin{prop}\label{prop:decompositiondata}
{\sl Let~$\varphi_{n}$ and~$\varphi_{0}$
belong to~$\cB_{1,q}$ for some~$0<q \leq \infty$, with~$\varphi_{n}$ R-converging to~$\phi^0$   as~$n$ goes to infinity.  Then with the notation of Proposition~{\rm\ref{prop:decompositiondata0}}, the following result holds:
\begin{equation}\label{anisoscales*}
\forall \ell \geq 1\, , \quad
 \displaystyle \lim_{n \to \infty}  \: (\gamma_n^\ell)^{-1} \varepsilon_n^{\ell}   \in \{0 , \infty\}\,  .
  \end{equation}
} \end{prop}

\begin{rmk}\label{rkoscaniso}{\sl This proposition shows that if one assumes that the  weak convergenceis actually an R-convergence, then the  only profiles remaining in the decomposition are those with truly anisotropic horizontal and vertical scales. This eliminates profiles of the type~$n \phi (nx)
 $ and~$ \vf(\cdot-x_n)$ with~$|x_n|\to \infty$, for which clearly the conclusion of Theorem~{\rm\ref{mainresult}} is unknown in general. This also 
 shows  that the assumption of R-convergence   is equivalent to the one of  anisotropic oscillations introduced in \cite{bg} and defined as follows:
 a sequence~$(f_n)_{n \in \N}$, bounded in~$ {\mathcal B}_{1,q}$, is said to be 
        {{anisotropically oscillating}}  if  for all sequences~$ (k_n,j_n) $ in~$   \ZZ^\N \times  \ZZ^\N$,
 \begin{equation}\label{frequencyextraction}
   \liminf_{n \to \infty}  \, 2^{k_n+ j_n } \|\Delta_{k_n}^{\rm h} \Delta_{j_n}^{\rm v} f_n \|_{L^1(\R^3)}   = C >0 \quad
   \Longrightarrow  \:  \lim_{n \to \infty}  \:|j_n - k_n| = \infty \, .
\end{equation}
} \end{rmk} 

\medskip
\begin{proof}[Proof of Proposition~{\rm\ref{prop:decompositiondata}}]
To prove~(\ref{anisoscales*}) we consider the decomposition provided in Proposition~\ref{prop:decompositiondata0} and we  assume that there is~$k \in \N$ such that~$(\gamma_n^k)^{-1} \varepsilon_n^{k} $ goes to 1 as~$n$ goes to infinity. We rescale the decomposition~(\ref{decompositiondecompositiondata0}) to find,
 choosing~$L \geq k$,
$$
\begin{aligned}
 \varepsilon_n^k (\varphi_{n}-\varphi_0)(\varepsilon_n^k \cdot + x_n^k) &= 
 \sum_{
   \ell = 1}^{L }  \Lambda_{\frac {\varepsilon_n^{\ell}} {\varepsilon_n^{k}} ,\frac {\gamma_n^{\ell}}{\varepsilon_n^{k}}, {x_n^{\ell,k}} }   \phi^{\ell} + \Lambda_{ \frac1{\varepsilon_n^{k}},  \frac1{\varepsilon_n^{k}},- \frac{x_n^k}  {\varepsilon_n^{k}}} {\psi_n^{L} }    
   \end{aligned}
$$
where
$$
x_n^{\ell,k} \eqdefa \frac{ x_n^{\ell} - x_n^{k} }{\eps_n^k} \, \cdotp
$$
 Now let us take the weak limit of both sides of the equality as~$n$ goes to infinity. By Definition~\ref{weakcvunderscaling} we know that the left-hand side goes weakly to zero.  Concerning the right-hand side, we start by noticing that    $$
 \frac{ \eps_n^\ell}{\eps_n^k}\to 0 \, \, \mbox{or} \, \,  \frac{ \eps_n^\ell}{\eps_n^k} \to \infty  \, \Longrightarrow  \,   \Lambda_{\frac {\varepsilon_n^{\ell}} {\varepsilon_n^{k}} ,\frac {\gamma_n^{\ell}}{\varepsilon_n^{k}}, {x_n^{\ell,k}} } \phi^{\ell}  \convf 0 \, , 
$$
as $n$ tends to infinity, for any value of the sequences~$ {\gamma_n^{\ell}}, {  x_n^{\ell}},$ and~$ x_n^{k}$. So we can restrict the sum on the right-hand side to the case when~$  \eps_n^\ell /\eps_n^k \to 1$. Then we write similarly
$$
  \frac{ \eps_n^\ell}{\gamma_n^\ell }\to \infty   \, \Longrightarrow  \,  \Lambda_{1,\frac {\gamma_n^{\ell}}{\varepsilon_n^{\ell}}, {x_n^{\ell,k}} } \phi^{\ell}   \convf 0 \, ,
$$ 
so there only remain indexes~$\ell$ such that~$ \eps_n^\ell /\gamma_n^\ell \to 0 $ or~$1 $.
Finally we use the fact that
if~$ \eps_n^\ell /\gamma_n^\ell \to 1 $, then the weak limit of~$ \Lambda_{1, {x_n^{\ell,k}} }  \phi^{\ell}$ can be other than zero only if~$ x_n^{\ell,k} \to a^{\ell,k} \in \R^3$, and similarly if~$ \eps_n^\ell /\gamma_n^\ell \to 0 $, then the weak limit of~$ \Lambda_{1,\frac {\gamma_n^{\ell}}{\varepsilon_n^{\ell}}, {x_n^{\ell,k}} }  \phi^{\ell}$ can be other than zero only if~$ x_{n,{\rm h}}^{\ell,k} \to a^{\ell,k}_{\rm h} \in \R^2$, and~$(x_{n,3}^{\ell} - x_{n,3}^{k}) / \gamma_n^\ell  \to a^{\ell,k}_3 \in \R$. 
So let us define
$$
\begin{aligned}
S^{1,L}(k) & \eqdefa \left\{
1 \leq \ell \leq L   \, / \, \eps_n^\ell = \eps_n^k \, , \,x_n^{\ell,k} \to a^{\ell,k} \in \R^3 \, , \,
\frac {\eps_n^\ell}{\gamma_n^{\ell}} %\eps_n^\ell(\gamma_n^{\ell}) ^{-1} 
\to 1 
\right\} \quad \mbox{and}\\
S^{0,L}(k) &\eqdefa \left\{
1 \leq \ell \leq L  \, / \, \eps_n^\ell = \eps_n^k \, , \,x_{n,{\rm h}}^{\ell,k} \to a^{\ell,k}_{\rm h} \in \R^2 \, ,\, \frac{x_{n,3}^{\ell} - x_{n,3}^{k}}{  \gamma_n^\ell } \to a^{\ell,k}_3 \in \R \, , \,\frac {\eps_n^\ell}{\gamma_n^{\ell}} \to 0 \right\}\, .
\end{aligned}
$$
Actually by orthogonality the set~$S^{1,L} (k)$ only contains one element, which is~$k$. So
 for each~$L \geq 1$, as~$n$ goes to infinity  we have finally
$$
- \Lambda_{ \frac1{\varepsilon_n^{k}},  \frac1{\varepsilon_n^{k}}, -\frac{x_n^k}  {\varepsilon_n^{k}}}  {\psi_n^{L} }   
  \convf
  \phi^k+
  \sum_{\ell \in S^{0,L}(k)} \phi^\ell (\cdot_{\rm h} - a_{\rm h}^{\ell,k}, - a_3^{\ell,k}) \, .
$$
Since the left-hand side tends to $0$ in~$\cB^0$ as~$L$ tends to infinity, uniformly in~$n \in \N$, we deduce that~$\phi^k$ must be independent of~$x_3$. That means that there is no vertical scale~$\gamma_n^k$, which proves the result.

  \medbreak

  \subsubsection{The case of divergence free vector fields}
Putting together Propositions~\ref{prop:decompositiondata0} and~\ref{prop:decompositiondata} along with Remark~\ref{testfunctionrk} and the fact that~$u_{0,n}$ is divergence free we obtain the following result.
   \begin{prop}\label{prop:decompositiondatabis}
 {\sl Under the assumptions of Theorem~{\rm\ref{mainresult}}, the following holds.
For all integers~$\ell \geq 1$ there is a triplet of   orthogonal sequences in the sense of Definition~{\rm \ref{orthoseq}}, denoted by~$( {\varepsilon_n ^\ell} , {\gamma_n^\ell}, {x_n^\ell})_{n \in \N}$  and for all~$\alpha$ in~$ ]0,1[$
 there are  arbitrarily smooth divergence free vector fields~$( \widetilde  \phi_\alpha^{h,\ell},0) $ and~$(- \nabla^{\rm h}\Delta_{\rm h}^{-1} \partial_3  \phi_\alpha^{\ell}  ,\phi_\alpha^{\ell})  $ with~$\widetilde  \phi_\alpha^{h,\ell}$ and~$ \phi_\alpha^{\ell}$  compactly supported, and such that  up to extracting a subsequence,   one can write the sequence~$  (u_{0,n})_{n \in \N}$ under the following form, for each~$L \geq 1$:
\begin{equation}\label{decu0n3.7}
\begin{aligned}
  {u_{0,n}  }  =  u_0
 & +  \sum_{
   \ell = 1}^{L }  \Lambda_{ {\varepsilon_n^{\ell}} , {\gamma_n^{\ell}}, {x_n^{\ell}} } \Big( \widetilde  \phi_\alpha^{{\rm h},\ell} +\widetilde  r_\alpha^{{\rm h},\ell}
- {\frac{ \varepsilon_n^{\ell} }{\gamma_n^{\ell} }}
 \nabla^{\rm h} \Delta_{\rm h}^{-1} \partial_3 ( \phi_\alpha^{\ell} + r_\alpha^{\ell}   )  ,  \phi_\alpha^{\ell}+ r_\alpha^{\ell}  \Big)  \\
 & \quad  +\big(  {\widetilde \psi_n^{{\rm h},  L} }-\nabla^{\rm h} \Delta_{\rm h}^{-1} \partial_3  {\psi_n^{L} }, {\psi_n^{L} }\big)    \, , 
   \end{aligned}
\end{equation}
 where~$\widetilde \psi_n^{{\rm h}, L}$ and~$\psi_n^L$ are independent of~$\alpha$  and satisfy
 \begin{equation}
 \label{smallremainderpsibis}
\limsup_{n \to \infty} \Big(   \|\widetilde \psi_n^{{\rm h},  L}\|_{\cB^0  } + \| \psi_n^{L}\|_{\cB^0 } \Big)\to 0\, , \quad L \to \infty \, ,
\end{equation}

\bigbreak\noindent
  while~$\widetilde  r_\alpha^{{\rm h},\ell}$ and~$  r_\alpha^{\ell}$ are independent of~$n$ and~$L$ and satisfy for each~$\ell \in \N$
  \begin{equation}\label{smallralpha}
   \| \widetilde  r_\alpha^{{\rm h},\ell}\|_{ {\mathcal B}_{1,q}} +\|    r_\alpha^{\ell}\|_{ {\mathcal B}_{1,q}} \leq \alpha \, .
\end{equation}

\medskip
\noindent
Moreover the following properties hold:
\begin{equation}\label{anisoscales}
\begin{aligned}
\forall \ell \geq 1\, , \quad
 \displaystyle \lim_{n \to \infty}  \: (\gamma_n^\ell)^{-1} \varepsilon_n^{\ell}   \in \{0 , \infty\}\,  ,
 \\
  \mbox{and} \quad
    \displaystyle \lim_{n \to \infty}  \: (\gamma_n^\ell)^{-1} \varepsilon_n^{\ell}  =\infty\,  \Longrightarrow\,
   \phi_\alpha^{\ell}\equiv r_\alpha^{\ell} \equiv 0 \, ,
   \end{aligned}
    \end{equation}
as well as the following stability result, which is uniform in~$\alpha$:
    \begin{equation}\label{orthonorms}
   \sum_{\ell \geq 1} \big(\|\widetilde  \phi_\alpha^{{\rm h},\ell} \|_{{\mathcal B}_{1,q}} +\|\widetilde  r_\alpha^{{\rm h},\ell} \|_{{\mathcal B}_{1,q}} +   \|   \phi_\alpha^{\ell}  \|_{{\mathcal B}_{1,q}} +   \|   r_\alpha^{\ell}  \|_{{\mathcal B}_{1,q}}  \big) \lesssim \sup_n \|u_{0,n}\|_{{\mathcal B}_{1,q}} + \|u_0\|_{{\mathcal B}_{1,q}} \, .
  \end{equation}
 }\end{prop}
\begin{proof}[Proof of Proposition~{\rm\ref{prop:decompositiondatabis}}] Note that due to Proposition~\ref{prop:decompositiondata} which asserts  that  the hypothesis of R-convergence  is equivalent to the one of anisotropic oscillations required in \cite{bg} (see Remark~\ref{rkoscaniso}), Proposition~\ref{prop:decompositiondatabis} is nothing else than Proposition~2.4 in \cite{bg}. Let us recall the argument. First  we decompose the third component $u^3_{0,n} $ according to  Proposition~\ref{prop:decompositiondata0} and Remark~\ref{testfunctionrk}: with the above notation, this  gives rise to 
   \begin{equation}  \label{decvert}{u^{3}_{0,n}  }  =  u^{3}_0
  +  \sum_{
   \ell = 1}^{L }  \Lambda_{ {\varepsilon_n^{\ell}} , {\gamma_n^{\ell}}, {x_n^{\ell}} } \big(  \phi_\alpha^{\ell}+ r_\alpha^{\ell}  \big) + \psi_n^{L}\,,
    \end{equation} 
    with~$\displaystyle \limsup_{n \to \infty} \| \psi_n^L\|_{\cB^0 }\stackrel{L\to\infty}\longrightarrow 0 $. Moreover thanks to Proposition \ref{prop:decompositiondata}, we know that  for all $\ell \geq 1$,  we have~$ 
 \displaystyle \lim_{n \to \infty}  \: (\gamma_n^\ell)^{-1} \varepsilon_n^{\ell}$  belongs to~$\{0 , \infty\}$.

  \smallbreak
 \noindent Next thanks to the divergence-free assumption we recover the profile decomposition for~$u^{\rm h}_{0,n} $. Indeed
 there is a two-component, divergence-free vector field~$ {\nabla^{\rm h}  }^\perp  C_{0,n}$
such that
$$
u_{0,n}^{\rm h}   = {\nabla^{\rm h}  }^\perp    C_{0,n}  - \nabla^{\rm h}  \Delta_{\rm h}  ^{-1} \partial_3  u_{0,n}^3 \, ,
$$
where~$ {\nabla^{\rm h}  }^\perp= (-\partial_1, \partial_2)$, and some function~$\varphi$ such that
$$
u_{0}^{\rm h}  =  {\nabla^{\rm h}  }^\perp\varphi- \nabla^{\rm h} \Delta_{\rm h} ^{-1} \partial_3  u_{0}^3 \, .
$$
Now since~$\ds \partial_3  u_{0,n}^3= - \mbox{div}_{\rm h} \:  u_{0,n}^{\rm h} $
and ~$ u_{0,n}^{\rm h} $ is bounded in~${\mathcal B}_{1,q}$,  we deduce that~$  {\nabla^{\rm h}  }^\perp C_{0,n} $ is  a bounded  sequence in~${\mathcal B}_{1,q}$ and  similarly for~$ {\nabla^{\rm h}  }^\perp \varphi$. Thus, applying again the profile decomposition of Proposition~\ref{prop:decompositiondata0} and Remark~\ref{testfunctionrk}, 
 we get
  \begin{equation}  \label{decomphoriz}
 {\nabla^{\rm h}  }^\perp C_{0,n}   -  \nabla^\perp_h \varphi  =
 \sum_{\ell  = 1}^L   \Lambda_{ {\tilde\varepsilon_n^{\ell}} , {\tilde\gamma_n^{\ell}}, {\tilde x_n^{\ell}} } \big( \widetilde  \phi_\alpha^{{\rm h},\ell} +\widetilde  r_\alpha^{{\rm h},\ell}\big)
 + \tilde \psi_n^{{\rm h},L}
\end{equation} 
  with~$\displaystyle \limsup_{n \to \infty} \| \tilde\psi_n^{{\rm h},L}\|_{\cB^0 }\stackrel{L\to\infty}\longrightarrow 0 $.    
Moreover   Proposition \ref{prop:decompositiondata}  ensures that  for all $\ell \geq 1$, we have~$ 
 \displaystyle \lim_{n \to \infty}  \: (\tilde\gamma_n^\ell)^{-1} \tilde\varepsilon_n^{\ell}   \in \{0 , \infty\}$. \\
 Finally,  by the divergence free assumption,~$    u_{0,n}^3$ is bounded in ~$
  B^{0,2}_{1,q}$ which implies that necessarily~$ \phi_\alpha^{\ell}\equiv r_\alpha^{\ell} \equiv 0$ in the case when~$  \displaystyle \lim_{n \to \infty}  \: (\gamma_n^\ell)^{-1} \varepsilon_n^{\ell}  = \infty $ (see Lemma 5.3 in \cite{bg}). \\
  Up to relabelling the various sequences appearing in~(\ref{decvert}) and~(\ref{decomphoriz}), Proposition \ref{prop:decompositiondatabis} follows.
 \end{proof} 
 
\medbreak
 
 \subsection{Regrouping of profiles according to horizontal scales}\label{reorganizationprofiles}
 With the notation of Proposition~\ref{prop:decompositiondatabis}, let us define the following scales:~$ {\eps_n^0} \equiv  {\gamma_n^0} \equiv   { 1} $, and~$ {  x_n^0} \equiv    0$, so that one has~$
u_0 \equiv \Lambda_{ {\eps_n^{0}} , {\gamma_n^{0}}, {  x_n^{0}} } u_0
$.

\medskip
\noindent
 In order to proceed with the re-organization of the profile decomposition provided in Proposition~\ref{prop:decompositiondatabis}, we    introduce some more definitions,  keeping the notation of
Proposition~\ref{prop:decompositiondatabis}.
For a given~$L \geq 1$ 
we define recursively an increasing  (finite) sequence of indexes~$\ell_k \in  \{  1,\dots,L\}$ by
\begin{equation}\label{deflitilde}
\begin{aligned}
 \ell_0 \eqdefa 0  \, , \quad
 \ell_{k+1}   \eqdefa \min \Big \{ \ell \in \{  \ell_k+1,\dots,L\} \, / \,\, \frac {\eps_n^\ell} {\gamma_n^{\ell}}\to 0
\quad \mbox{and} \quad
\ell \notin \displaystyle  \bigcup_{k' = 0}^{k}   \Gamma^L({{ {\eps_n^{ \ell_{k'}}}}})
\Big\}  \, ,
\end{aligned}
\end{equation}
where for~$0 \leq \ell \leq L$,
we define (recalling that by Remark~\ref{equalscales} if two scales are not orthogonal, then they are equal),
\begin{equation}\label{defLambdatilde}
  \Gamma^L({{ {\eps_n^\ell}}})\eqdefa \Big \{
\ell' \in \{1,\dots,L\} \, / \, { {\eps_n^{\ell'}} } \equiv { {\eps_n^{\ell}} }  \quad \mbox{and} \quad
\eps_n^{\ell'} (\gamma_n^{\ell'}) ^{-1} \to 0 \, , \, n \to \infty
\Big\}  \, .
\end{equation}
We call~$ \cL(L)$ the largest index of the sequence~$(\ell_k)$ and we may then introduce the following partition:
\begin{equation}\label{defpartition}
     \big \{ \ell \in \{1,\dots,L\} \, / \, \eps_n^\ell ( \gamma_n^\ell)^{-1} \to 0 \big\}   = \bigcup_{k=0}^{ \cL(L)}   \Gamma^L({{ {\eps_n^{  \ell_k}}}})  \, .
\end{equation}
  We shall now regroup profiles in the decomposition~(\ref{decu0n3.7}) of~$u_{0,n}$  according to the value of their horizontal scale.
  We fix from now on an integer~$L \geq 1$.

 \subsubsection{Construction of the profiles for~$\ell = 0$}\label{j0profile}
 Before going into the technical details of the construction, let us discuss an example explaining the computations of this paragraph. Consider the particular case when~$u_{0,n}  $ is given by
 $$
 u_{0,n} (x)= u_0(x)+\big(v_0^{0} (x_{\rm h},2^{-n}x_3) +  w_0^{0,\rm h} (x_{\rm h},2^{-2n}x_3) ,0 \big) + \big(v_0^{0} (x_{1}+ n, x_2, 2^{-n}x_3) ,0 \big) \, , 
 $$ with~$v_0^{0}$ and~$w_0^{0,\rm h}$ smooth (say in~$  B^{s,s'}_{1,q}$ for all~$s,s'$ in~$\R$) and compactly supported. Let us assume that~$(u_{0,n})_{n\in \N}$  R-converges to $u_0$, as $n$ tends to infinity. Then we can write
 $$
 u_{0,n} (x)=u_0(x)+ \big( v^{0,\rm {loc}}_{0,n}(x_{\rm h},2^{-n}x_3),0 \big)+\big(v^{0,\infty}_{0,n}(x_{\rm h},2^{-n}x_3),0 \big) \, ,
 $$
 with~$v^{0,\rm {loc}}_{0,n}(y):=v_0^{0} (y) +  w_0^{0,\rm h} (y_h,2^{-n}y_3)$ and~$v^{0,\infty}_{0,n} (y) = v_0^{0} (y_{1}+ 2^n, y_2, y_3)$. We notice that~$v^{0,\rm {loc}}_{0,n}$ and~$v^{0,\infty}_{0,n}$ are uniformly bounded in~${\mathcal B}_{1,q}$, but also in~$  B^{s,s'}_{1,q}$ for any~$s$ in~$ \R$ and~$s' \geq 1/2$.

 \noindent Moreover since~$u_{0,n} \rightharpoonup u_0$ as~$n$ goes to infinity, we have that~$v_0^{0} (x_{\rm h},0) +  w_0^{\rm h} (x_{\rm h},0) \equiv 0 $, hence~$v^{0,\rm {loc}}_{0,n}(x_{\rm h},0) = 0$. 
  The initial data~$u_{0,n}$ has therefore been re-written as
 $$
 u_{0,n} (x)=u_0(x)+ \big(v^{0,\rm {loc}}_{0,n}(x_{\rm h},2^{-n}x_3),0 \big)+\big(v^{0,\infty}_{0,n}(x_{\rm h},2^{-n}x_3),0 \big) \,  \quad \mbox{with} \quad v^{0,\rm {loc}}_{0,n}(x_{\rm h},0) = 0 
 $$
 and where the support in~$x_{\rm h}$ of~$v^{0,\rm {loc}}_{0,n}(x_{\rm h},2^{-n}x_3) $ is in a fixed compact set whereas the support  in~$x_{\rm h}$ of~$v^{0,\infty}_{0,n}(x_{\rm h},2^{-n}x_3) $ escapes to infinity.
 This is of the same form   as in the statement of Theorem~\ref{mainresultprofilesdi}.

\smallskip

 \noindent  When considering all the profiles having the same horizontal scale (1 here), the point is therefore to choose the smallest vertical scale ($2^n$ here) and to write the decomposition in terms of that scale only. Of course that implies that contrary to usual profile decompositions,  the profiles are no longer fixed functions in~${\mathcal B}_{1,q}$, but  sequences of functions, bounded in~${\mathcal B}_{1,q}$.

\bigskip
\noindent In view of the above example,  let~$\ell_0^-$ be an integer such that~$ {\gamma_n^{\ell_0^-}} $ is the smallest vertical scale going to infinity, associated with profiles for~$1  \leq \ell \leq L$, having~$1$ for  horizontal scale. More precisely we ask that
$$
 {\gamma_n^{\ell_0^-}} = \min_{\ell  \in   \Gamma^L(  1)}    {\gamma_n^\ell }   \, ,
$$
where according to~(\ref{defLambdatilde}),
$$
 \Gamma^L({ {1}}) = \Big \{
\ell' \in \{1,\dots,L\} \, / \, { {\eps_n^{\ell'}} } \equiv { {1} }  \quad \mbox{and} \quad
   \gamma_n^{\ell'} \to \infty \, , \, n \to \infty
\Big\} \, .
$$
 Notice that  the minimum of   the sequences~$\gamma_n^\ell$ is well  defined in our context thanks to the fact that due to Remark~\ref{equalscales}, either two sequences are orthogonal in the sense of Definition~\ref{orthoseq}, or they are equal.
Remark also that~$\ell_0^-$ is by no means unique, as several profiles may have the same horizontal scale as well as the same vertical scale (in which case the concentration cores must be orthogonal).
\medskip
\noindent Now we denote
  \begin{equation}\label{defh0}
  { h_n^0}\eqdefa  { (\gamma_n^{\ell_0^-} )^{-1} } \, ,
\end{equation}
and we notice that~$h_n^0$ goes to zero as~$n$ goes to infinity for each~$L$. Note also that~$h_n^0$    depends on~$L$   through the choice of~$\ell_0^-$, since if~$L$ increases then~$\ell_0^-$ may also increase; this dependence is omitted in the notation for simplicity. Let us define (up to a subsequence extraction)
 \begin{equation}\label{defaell}
   a^\ell \eqdefa \lim_{n \to \infty} \Big(  x_{n,{\rm h}}^{\ell},  \frac{x_{n,3}^{\ell}}{\gamma_n^\ell}\Big) \, \cdotp
\end{equation}

\medskip
\noindent   We then define the divergence-free vector fields
    \begin{equation}\label{defv0}
v^{0,{\rm loc,  h}}_{0,n,\alpha,L} (y)  \eqdefa \sumetage{ \ell \in   \Gamma^L({ {1}})} {a_{\rm  h}^\ell \in \R^2 }\widetilde \phi_\alpha^{h,\ell} \Big(y_{\rm h}-  x_{n,{\rm   h}}^\ell \, , \frac{y_3}{h_n^0 \gamma_n^{\ell}} -  \frac{   x_{n,3}^\ell}{ \gamma_n^{\ell}}  \Big)
\end{equation}
and
  \begin{equation}\label{defw0}
\begin{aligned}
    w_{0,n,\alpha,L}^{0,\rm loc} (y)  \eqdefa \sumetage{ \ell \in   \Gamma^L({ {1}})} {a_{\rm  h}^\ell \in \R^2 } \Big(
 - \frac{1}{h_n^0\gamma_n^\ell} \nabla^{\rm h} \Delta_{\rm h}^{-1} \partial_3  \phi_\alpha^{ \ell} , \phi_\alpha^{ \ell}
\Big) \Big(y_{\rm h}-  x_{n,{\rm   h}}^\ell \, , \frac{y_3}{h_n^0 \gamma_n^{\ell}}  -  \frac{   x_{n,3}^\ell}{ \gamma_n^{\ell}}  \Big)
\, .
 \end{aligned}
\end{equation}
By construction we have
$$
    w_{0,n,\alpha,L}^{0,{\rm  loc, h}} = - \nabla^{\rm h} \Delta_{\rm h}^{-1} \partial_3     w_{0,n,\alpha,L}^{0,{\rm  loc},3} \, .
$$
Similarly we define
   \begin{equation}\label{defv0infty}
   v^{0,\infty, {\rm   h}}_{0,n,\alpha,L} (y)  \eqdefa \sumetage{ \ell \in   \Gamma^L({ {1}})} {|a_{\rm  h}^\ell| = \infty}\widetilde \phi_\alpha^{h,\ell} \Big(y_{\rm h}-  x_{n,{\rm   h}}^\ell \, , \frac{y_3}{h_n^0 \gamma_n^{\ell}} -  \frac{   x_{n,3}^\ell}{ \gamma_n^{\ell}}  \Big)
\end{equation}
and
  \begin{equation}\label{defw0infty}
\begin{aligned}
    w_{0,n,\alpha,L}^{0,\infty} (y)  \eqdefa \sumetage{ \ell \in   \Gamma^L({ {1}})} {|a_{\rm  h}^\ell| = \infty } \Big(
 - \frac{1}{h_n^0\gamma_n^\ell} \nabla^{\rm h} \Delta_{\rm h}^{-1} \partial_3  \phi_\alpha^{ \ell} , \phi_\alpha^{ \ell}
\Big) \Big(y_{\rm h}-  x_{n,{\rm   h}}^\ell \, , \frac{y_3}{h_n^0 \gamma_n^{\ell}}  -  \frac{   x_{n,3}^\ell}{ \gamma_n^{\ell}}  \Big)
\, .
 \end{aligned}
\end{equation}
By construction we have again
$$
    w_{0,n,\alpha,L}^{0,\infty,{\rm   h}} = - \nabla^{\rm h} \Delta_{\rm h}^{-1} \partial_3     w_{0,n,\alpha,L}^{0,\infty,3} \, .
$$
Moreover  recalling the notation
$$
 \big[f  ]_{ h_n^0}  (x) \eqdefa f(x_{\rm h},{    h}_{n}^0 x_3)
$$
and
   $$
  \Lambda_{ {\varepsilon_n} , {\gamma_n}, {x_n} } \phi (x) \eqdefa
  \frac1 {\varepsilon_n}  \phi\left( \frac{x_{\rm h}-x_{n,{\rm   h}} }{\varepsilon_n}, \frac{x_3-x_{n,3}} {\gamma_n} \right) \, ,
  $$
one can compute that
 \begin{equation}\label{formula0}
   \sumetage{ \ell \in   \Gamma^L({ {1}})} {a_{\rm  h}^\ell \in \R^2} \  \Lambda_{ 1, {\gamma_n^{\ell}}, {  x_n^{\ell}} }  \left(
\widetilde \phi_\alpha^{h,\ell} - \frac1{\gamma_n^\ell} \nabla^{\rm h} \Delta_{\rm h}^{-1} \partial_3  \phi_\alpha^{ \ell} , \phi_\alpha^{ \ell}
\right)   = \big[(   v_{0,n,\alpha,L}^{0,{\rm  loc, h}} + h_n^0   w_{0,n,\alpha,L}^{0,{\rm loc,h}}  ,   w_{0,n,\alpha,L}^{0,{\rm loc},3} )\big]_{ h_n^0}  
\end{equation}
and
\begin{equation}\label{formula0infty}
   \sumetage{ \ell \in   \Gamma^L({ {1}})} {|a_{\rm  h}^\ell | = \infty  } \  \Lambda_{ 1, {\gamma_n^{\ell}}, {  x_n^{\ell}} }  \left(
\widetilde \phi_\alpha^{h,\ell} - \frac1{\gamma_n^\ell} \nabla^{\rm h} \Delta_{\rm h}^{-1} \partial_3  \phi_\alpha^{ \ell} , \phi_\alpha^{ \ell}
\right)   = \big[(   v_{0,n,\alpha,L}^{0,\infty, {\rm   h}} + h_n^0    w_{0,n,\alpha,L}^{0,\infty,{\rm h}}  ,    w_{0,n,\alpha,L}^{0,\infty,3} )\big]_{ h_n^0}  
\, .\end{equation}
Let us now check that~$  v_{0,n,\alpha,L}^{0,{\rm loc, h}}$,~$  w_{0,n,\alpha,L}^{0,{\rm loc}}$,~$   v_{0,n,\alpha,L}^{0,\infty,{\rm h}}$ and~$w_{0,n,\alpha,L}^{0,\infty}$ satisfy the   bounds given in the statement of Theorem~\ref{mainresultprofilesdi}. We shall only study~$  v_{0,n,\alpha,L}^{0,{\rm loc,h}}$ and~$  w_{0,n,\alpha,L}^{0,{\rm loc}}$ as the other study is very similar.
On the one hand, by translation and scale invariance of~${\mathcal B}_{1,q}$ and using definitions~(\ref{defv0}) and~(\ref{defw0}), we get
\begin{equation}\label{stab1}
\|v_{0,n,\alpha,L}^{0,{\rm loc, h}}\|_{{\mathcal B}_{1,q}} \leq \sum_{ \ell \geq 1} \|\widetilde  \phi_\alpha^{h,\ell}\|_{{\mathcal B}_{1,q}} \quad \mbox{and} \quad
\|  w_{0,n,\alpha,L}^{0,{\rm loc}, 3}\|_{{\mathcal B}_{1,q}} \leq \sum_{ \ell \geq 1} \| \phi_\alpha^{\ell}\|_{{\mathcal B}_{1,q}} \, .
\end{equation}
By~(\ref{orthonorms}), we infer that
\begin{equation}\label{boundv0w0-unif}
\| v_{0,n,\alpha,L}^{0,{\rm loc, h}}\|_{{\mathcal B}_{1,q}} + \| w_{0,n,\alpha,L}^{0,{\rm loc}, 3}\|_{{\mathcal B}_{1,q}}  \leq C \quad  \mbox{uniformly in} \, \, \alpha \, , L \, , n \, .
\end{equation}
Moreover for each given~$\alpha$, the profiles are as smooth as needed, and since in the above sums by construction~$\gamma_{n,L}^{\ell_0^-} \leq \gamma_n^{\ell} $,   one gets also after an easy computation
\begin{equation}\label{boundv0w0-notunif}
\forall s \in \R \, , \forall s' \geq 1/2\, , \quad   \|v_{0,n,\alpha,L}^{0,{\rm loc, h}}\|_{  B^{s,s'}_{1,q}  } + \|v_{0,n,\alpha,L}^{0,{\rm loc}} \|_{  B^{s,s'}_{1,q} }   \leq C(\alpha) \quad  \mbox{uniformly in} \, \, n\, , L \,  .
   \end{equation}
Estimates~(\ref{boundv0w0-unif}) and~(\ref{boundv0w0-notunif}) give easily~(\ref{orthonorms2}) and~(\ref{orthonorms22}). 

   \medskip
   \noindent
   Finally let us estimate~$ v_{0,n,\alpha,L}^{0,{\rm loc, h}} (\cdot , 0)$ and~$w_{0,n,\alpha,L}^{0,{\rm loc}, 3}(\cdot , 0)$ in~$   B^{0}_{2,1} (\R^2)$ and prove~(\ref{smallatzerothmdata}).
  On the one hand by assumption we know that~$u_{0,n}  \convf u_0$ in the sense of distributions. On the other hand we can take weak limits in the decomposition of~$u_{0,n}  $ provided by Proposition~\ref{prop:decompositiondatabis}.
  We recall   that by~(\ref{anisoscales}), if~$\eps_n^\ell/\gamma_n^\ell \to \infty$ then~$\phi_\alpha^\ell \equiv r_\alpha^\ell \equiv 0$. Then we notice that clearly
$$
\eps_n^\ell \to 0 \, \, \mbox{or} \, \,  \eps_n^\ell \to \infty  \, \Longrightarrow  \, \Lambda_{ {\eps_n^{\ell}} , {\gamma_n^{\ell}}, {  x_n^{\ell}} } f \convf 0
$$
for any value of the sequences~$ {\gamma_n^{\ell}}, {  x_n^{\ell}}$ and any function~$f$. Moreover
$$
\gamma_n^\ell \to 0   \,\Longrightarrow   \,\Lambda_{ {1} , {\gamma_n^{\ell}}, {  x_n^{\ell}} } f \convf  0
$$
for any  sequence of cores~$  {  x_n^{\ell}}$ and any function~$f$, so we are left  with  the study of  profiles such that~$\eps_n^\ell \equiv 1$ and~$\gamma_n^\ell \to \infty$.
Then we also notice that if~$\gamma_n^\ell \to \infty$, then with Notation~(\ref{defaell}),
\begin{equation}\label{limxtoinftygamma}
| a_{\rm h}^\ell| = \infty
 \,\Longrightarrow   \,\Lambda_{ {1} , {\gamma_n^{\ell}}, {  x_n^{\ell}} } f \convf  0 \,.
\end{equation}
Consequently for each~$L \geq 1$ and each~$\alpha$ in~$]0,1[$, we have in view of \eqref{decvert} and \eqref{decomphoriz}, as~$n$ goes to infinity
\begin{eqnarray*}
u^3_{0,n}-\psi_n^{L} -\sum_{ \ell \in   \Gamma^L({ {1}})}     r_\alpha^\ell(\cdot -  x_{n,{\rm h}}^{\ell}, \frac{ \cdot -   x_{n,3}^{\ell}}{\gamma_n^\ell}) & \convf & u^3_0 +    \sumetage{ \ell \in   \Gamma^L({ {1}})}{ s.t. \, a^\ell_{\rm h} \in \R^2 }      \phi_\alpha^\ell(\cdot -  a_{h}^\ell,  0)\\ 
\nabla^\perp_h C_{0,n}-\tilde \psi_n^{h,L}-\sum_{ \ell \in   \Gamma^L({ {1}})}     \widetilde r_\alpha^{h,\ell}(\cdot -  x_{n,{\rm h}}^{\ell}, \frac{ \cdot -   x_{n,3}^{\ell}}{\gamma_n^\ell}) &\convf& \nabla^\perp_h \varphi +\sumetage{ \ell \in   \Gamma^L({ {1}})}{ s.t. \, a^\ell_{\rm h} \in \R^2}       \phi_\alpha^{{\rm h},\ell}(\cdot -  a_{h}^\ell,  0)\,.
\end{eqnarray*}
By hypothesis the sequence $(u^3_{0,n})_{n \in \N}$ converges weakly to  $ u^3_0$ and the sequence $(\nabla^\perp_h C_{0,n})_{n \in \N}$ converges weakly to  $\nabla^\perp_h \varphi$, so for each~$L \geq 1$ and all~$\alpha$ in~$] 0,1[$, we have as~$n$ goes to infinity
\begin{equation}\label{wlimGamma}
\begin{aligned}
   - {\psi_n^{L} }  
-   \sum_{ \ell \in   \Gamma^L({ {1}})}      r_\alpha^\ell(\cdot -  x_{n,{\rm h}}^{\ell},    \frac{  \cdot -   x_{n,3}^{\ell}}{\gamma_n^\ell}) & \convf \sumetage{ \ell \in   \Gamma^L({ {1}})}{ s.t. \, a^\ell_{\rm h} \in \R^2 }      \phi_\alpha^\ell(\cdot -  a_{h}^\ell, 0)
\\ -\tilde \psi_n^{h,L}-\sum_{ \ell \in   \Gamma^L({ {1}})}     \widetilde r_\alpha^{h,\ell}(\cdot -  x_{n,{\rm h}}^{\ell}, \frac{ \cdot -   x_{n,3}^{\ell}}{\gamma_n^\ell}) & \convf     \sumetage{ \ell \in   \Gamma^L({ {1}})}{ s.t. \, a^\ell_{\rm h} \in \R^2 }       \widetilde \phi_\alpha^{h,\ell}(\cdot -  a_{h}^\ell, 0)\, .
\end{aligned}
\end{equation}
Now let~$\eta>0$ be given. Then thanks to~(\ref{smallremainderpsibis}) and~(\ref{smallralpha}), 
 there is~$L_0 \geq 1$ such that for all~$L \geq L_0$  there is~$\alpha_0 \leq 1$  (depending on~$L$) such that for all~$ L \geq L_0 $ and~$\alpha \leq \alpha_0 $, uniformly in $n \in \N$
$$
  \Big \| \big(  {\widetilde \psi_n^{{\rm h},  L} }, {\psi_n^{L} }\big)\Big \|_{\cB^0}
+
\Big \|   \sum_{ \ell \in   \Gamma^L({ {1}})}    ( \widetilde r_\alpha^{h,\ell},   r_\alpha^\ell)(\cdot -  x_{n,{\rm h}}^{\ell}, \frac{ \cdot -  x_{n,3}^{\ell}}{\gamma_n^\ell})\Big \|_{\cB^0}
\leq \eta \, .
$$
Using the fact that~$\cB^0$
  is embedded in~$L^\infty(\R ;  B^{0}_{2,1}(\R^2))$, we infer from~(\ref{wlimGamma}) that 
for~$L \geq L_0$
and~$\alpha \leq \alpha_0$
 \begin{equation}\label{lessthanetatildephi}
 \Big \| \sumetage{ \ell \in   \Gamma^L({ {1}})}{ s.t. \, a^\ell_{\rm h} \in \R^2 }     \widetilde \phi_\alpha^{h,\ell} (\cdot -  a_{h}^\ell,  0)
\Big \|_{  B^{0}_{2,1}(\R^2)}
  \leq \eta 
  \end{equation}
  and 
 \begin{equation}\label{lessthanetapastildephi}
  \Big \|  \sumetage{ \ell \in   \Gamma^L({ {1}})}{ s.t. \, a^\ell_{\rm h} \in \R^2 }     \phi_\alpha^\ell(\cdot -  a_{h}^\ell, 0)
\Big \|_{  B^{0}_{2,1}(\R^2)}
  \leq \eta\, .
 \end{equation}
But by~(\ref{defv0}), we have
$$
v_{0,n,\alpha,L}^{0,{\rm loc,h}}(\cdot , 0)  = \sumetage{ \ell \in   \Gamma^L({ {1}})} {a_{\rm  h}^\ell \in \R^2 } \widetilde \phi_\alpha^{h,\ell} \Big(\cdot-  x_{n,{\rm h}}^\ell, -  \frac{   x_{n,3}^\ell}{ \gamma_n^{\ell}}  \Big)
$$
and by~(\ref{defw0}) we have also
$$
w_{0,n,\alpha,L}^{0,{\rm loc}, 3} (\cdot , 0)  =
\sumetage{ \ell \in   \Gamma^L({ {1}})} {a_{\rm  h}^\ell \in \R^2 }  \phi_\alpha^{ \ell}
  \Big(\cdot-  x_{n,{\rm h}}^\ell,   -  \frac{   x_{n,3}^\ell}{ \gamma_n^{\ell}}  \Big)
\, .
$$
It follows that we can write  for all~$ L \geq L_0 $ and~$\alpha \leq \alpha_0 $,
  $$
\begin{aligned}
 \limsup_{n \to \infty}\|v_{0,n,\alpha,L}^{0,{\rm loc,h}}(\cdot,0)\|_{{  B}^{0}_{2,1}(\R^2)} &  \leq  \big \|  \sumetage{ \ell \in   \Gamma^L({ {1}})}{  a^\ell_{\rm h} \in \R^2}   \widetilde \phi_\alpha^{h,\ell} (\cdot-  a_{h}^\ell, 0 )
\big \|_{{  B}^{0}_{2,1}(\R^2)}  \\
 &  \leq  \eta 
 \end{aligned}
$$
thanks to~(\ref{lessthanetatildephi}).
A similar estimate for~$w_{0,n,\alpha,L}^{0,{\rm loc}, 3} (\cdot,0)$ using~(\ref{lessthanetapastildephi})
gives finally
\begin{equation}\label{boundv0w0-atzero}
\lim_{L \to \infty}\lim_{\alpha \to 0}  \limsup_{n \to \infty}  \Big(\|v_{0,n,\alpha,L}^{0,{\rm loc,h}}(\cdot,0)\|_{{  B}^{0}_{2,1}(\R^2)} + \|w_{0,n,\alpha,L}^{0,{\rm loc}, 3} (\cdot,0)\|_{{  B}^{0}_{2,1}(\R^2)} \Big) = 0     \, .
\end{equation}
The results~(\ref{mainresultprofilesdieq-2}) and~(\ref{mainresultprofilesdieq-1}) involving  the cut-off function~$\theta$ are simply due to the fact that the profiles are compactly supported.
 
 \subsubsection{Construction of the profiles for~$\ell \geq 1$}\label{jprofile}
  The construction is very similar to the previous one.
 We start by considering a fixed integer~$ j  \in \{1, \dots, {\mathcal {L}}(L)\}$.
\medskip
\noindent Then we define an integer~$ \ell_j^-$ so that, up to a sequence extraction,
$$
   { \gamma_n^{ \ell_j^-}}   = \minetage{\ell  \in   \Gamma^{L}({   {\eps_n^{ \ell_j}}})}
{}
   {\gamma_n^{\ell}}\,,
$$
where as in \eqref{defLambdatilde}
$$
  \Gamma^L({{ {\eps_n^\ell}}})\eqdefa \Big \{
\ell' \in \{1,\dots,L\} \, / \, { {\eps_n^{\ell'}} } \equiv { {\eps_n^{\ell}} }  \quad \mbox{and} \quad
\eps_n^{\ell'} (\gamma_n^{\ell'})^{-1} \to 0 \, , \, n \to \infty
\Big\}  \, .
$$
Notice that necessarily~$\eps^{\ell_j^-} \not \equiv 1$.
Finally
we define
$$
 { h^j _{n} }\eqdefa   {\eps_n^{  \ell_j}}( {\gamma_n^{\ell_j^-}})^{-1}
\, .
$$
By construction we have that~$ h_{n}^j \to 0$ as~$n \to \infty$ (recall that~$ {\eps_n^{\ell_j}}‚àö√á¬¨¬®‚àö¬¢¬¨√Ñ¬¨‚Ä†\equiv  {\eps_n^{\ell_j^-}}$).
Then we define for~$j \leq   {\mathcal L}(L) $
\begin{equation}\label{notequaltozerotoolarge}
  v_{n,\alpha,L}^{j,{\rm h}} (y) \eqdefa \sum_{  \ell \in    \Gamma^{L}({   {\eps_n^{  \ell_{  j }}}})}  \widetilde  \phi_\alpha^{h,\ell}  \Big(y_h - \frac{  x_{n,{\rm h}}^\ell}{\eps_n^{  \ell_{  j}}} , \frac{ \eps_{n}^  {\ell_{  j } }}{ h_{n}^j \gamma_n^{\ell}} y_3 -\frac{  x_{n,3}^\ell}{\gamma_n^{\ell}} \Big)
\end{equation}
and
$$
\begin{aligned}
  w_{n,\alpha,L}^{j} (y) \eqdefa \sum_{  \ell \in    \Gamma^{L}({   {\eps_n^{  \ell_{  j }}}}) } \Big( -\frac{ \eps_{n}^  {\ell_{  j } }}{ h_{n}^j \gamma_n^{\ell}}   \nabla^{\rm h} \Delta_{\rm h}^{-1} \partial_3  \phi_\alpha^{ \ell} , \phi_\alpha^ {\ell} \Big) \Big(y_h - \frac{  x_{n,{\rm h}}^\ell}{\eps_n^{  \ell_{  j}}} , \frac{ \eps_{n}^  {\ell_{  j } }}{ h_{n}^j \gamma_n^{\ell}} y_3 -\frac{  x_{n,3}^\ell}{\gamma_n^{\ell}} \Big)\\
\end{aligned}
$$
and we choose
\begin{equation}\label{equaltozerotoolarge}
   {\mathcal L}(L)<j  \leq L \quad \Rightarrow \quad     v_{n,\alpha,L}^{j,{\rm h}} \equiv 0 \quad \mbox{and} \quad
    w_{n,\alpha,L}^{j} \equiv 0  \,.
\end{equation}
We notice that
$$
  w_{n,\alpha,L}^{j,{\rm h}}  = - \nabla^{\rm h} \Delta_{\rm h}^{-1} \partial_3   w_{n,\alpha,L}^{j,3}\,.
$$
Defining
$$
 { \lambda_n^j }\eqdefa {\eps_n^{  \ell_{  j}}}\, ,
$$
a computation, similar to that giving~(\ref{formula0}) implies directly that
\begin{equation}\label{formulajtilde}
\begin{aligned}
 & \sum_{  \ell \in   \Gamma^{L}({   {\eps_n^{  \ell_{  j }}}})} \Lambda_{{   {\eps_n^{  \ell_{  j }}}}, {\gamma_n^{\ell}}, {  x_n^{\ell}} }
\Big( \widetilde \phi_\alpha^{h,\ell}- \frac{ \lambda_n^j }{\gamma_n^\ell} \nabla^{\rm h} \Delta_{\rm h}^{-1} \partial_3  \phi_\alpha^{ \ell} , \phi_\alpha^{ \ell}\Big) \\
 & \qquad \qquad\qquad\qquad =  \Lambda_{{  \lambda_n^j}} \big [(  v_{n,\alpha,L}^{j,{\rm h}}+   h_{n}^j   w_{n,\alpha,L}^{j,{\rm h}},   w_{n,\alpha,L}^{j,3}) \big]_{  h_{n}^j} \, .
\end{aligned}
\end{equation}
 Notice that since~$  {\eps_n^{  \ell_{  j}}}\not\equiv   1$ as recalled above, we have that~$ { \lambda_n^j }\to 0$ or~$\infty$ as~$n \to \infty$.

 \smallskip
 \noindent The a priori bounds for the profiles~$(    {v_{n,\alpha,L}^{j,{\rm h}}} ,    {w_{n,\alpha,L}^{j,3}})_{1 \leq j \leq L}$  are obtained exactly as in the previous paragraph: let us prove that
   \begin{equation}\label{boundvjwj-unif}
     \begin{aligned}
  \sum_{j \geq 1} \big( \| v_{n,\alpha,L}^{j,{\rm h}}\|_{{\mathcal B}_{1,q}} + \| w_{n,\alpha,L}^{j,3} \|_{{\mathcal B}_{1,q}}\big) \leq C   \, , \quad \mbox{and} \\
   \forall s \in \R \, , \quad \forall s'\geq 1/2 \, , \quad    \sum_{j \geq 1} \big( \| v_{n,\alpha,L}^{j,{\rm h}}\|_{  B^{s,s'}_{1,q}} + \| w_{n,\alpha,L}^{j,3} \|_{  B^{s,s'}_{1,q}}  \big) \leq C(
       \alpha) \, .
    \end{aligned}
  \end{equation}
  We shall detail the argument for the first inequality only, and in the case of~$v_{n,\alpha,L}^{j,{\rm h}}$ as the study of~$ w_{n,\alpha,L}^{j,3} $ is similar.
    We write, using the definition of~$  v_{n,\alpha,L}^{j,{\rm h}}$ in~(\ref{notequaltozerotoolarge}),
   $$
    \begin{aligned}
       \sum_{j = 1}^{L}   \| v_{n,\alpha,L}^{j,{\rm h}}\|_{{\mathcal B}_{1,q}}       & =    \sum_{j=1}^{  {\mathcal L}(L)}
 \Big\|   \sum_{  \ell \in    \Gamma^{L}({   {\eps_n^{  \ell_{  j }}}})}  \widetilde  \phi_\alpha^{h,\ell}  \Big(y_h - \frac{  x_{n,{\rm h}}^\ell}{\eps_n^{  \ell_{  j}}} , \frac{ \eps_{n}^  {\ell_{  j } }}{ h_{n}^j \gamma_n^{\ell}} y_3 -\frac{  x_{n,3}^\ell}{\gamma_n^{\ell}} \Big)
\Big \|_{{\mathcal B}_{1,q}}   \, ,
      \end{aligned}
  $$
so by definition of the partition~(\ref{defpartition}) and by scale and translation invariance of~${\mathcal B}_{1,q}$ we find thanks to~(\ref{orthonorms}), that there is a constant~$C $ independent of~$L$ such that
$$
 \sum_{j = 1}^{L}   \| v_{n,\alpha,L}^{j,{\rm h}}\|_{{\mathcal B}_{1,q}} \leq \sum_{\ell=1}^{L}
    \|\widetilde  \phi_\alpha^{h, \ell} \|_{{\mathcal B}_{1,q}}  \leq C \, .
$$
The result is proved.

%  \medskip
%  \noindent
%  Finally let us prove that for each~$\eta>0$ there is~$L_0$ such that
%\begin{equation}\label{smallforlargeL0}
%   \sum_{j \geq L_0} \big( \| v_{n,\alpha,L}^{j,{\rm h}}\|_{{\mathcal B}_{1,q}} + \| w_{n,\alpha,L}^{j,3} \|_{{\mathcal B}_{1,q}}\big) \leq \eta \, ,
%  \end{equation}
%uniformly in~$n,\alpha$ and~$L$. Again we only consider~$v_{n,\alpha,L}^{j,{\rm h}}$ and we just
%copy the above computations, starting from~$j$ equal to some~$L_0$ to be determined. We get
%$$
% \begin{aligned}
%       \sum_{j = L_0}^{L}   \| v_{n,\alpha,L}^{j,{\rm h}}\|_{{\mathcal B}_{1,q}}          & \leq \sum_{j=L_0}^{  {\mathcal L}(L)}
%   \sum_{  \ell \in  \Gamma^{L}({   {\eps_n^{  \ell_{  j }}}})}  \|\widetilde  \phi_\alpha^{h,\ell} \|_{{\mathcal B}_{1,q}}   \, ,      \end{aligned}
%$$
%and since by construction~$  \ell_{  j } \geq j$ due to~(\ref{deflitilde}) and by definition if $\ell \in  \Gamma^{L}({   {\eps_n^{  \ell_{  j }}}})$ then $\ell \geq  \ell_{  j }$, we get
%$$
%     \sum_{j = L_0}^{L}   \| v_{n,\alpha,L}^{j,{\rm h}}\|_{{\mathcal B}_{1,q}}   \leq \sum_{\ell \geq L_0}     \|\widetilde  \phi_\alpha^{h, \ell} \|_{{\mathcal B}_{1,q}}   .
%$$
%If~$L_0$ is large enough we know from~(\ref{orthonormsL0}) that
%$$
% \sum_{\ell \geq L_0}     \|\widetilde  \phi_\alpha^{h, \ell} \|_{{\mathcal B}_{1,q}}  \leq \eta
% $$
%whence the result~(\ref{smallforlargeL0}).
%

\subsubsection{Construction  of the remainder term}\label{remainderconstruction}
With the notation of Proposition~\ref{prop:decompositiondatabis}, let us first define
  the remainder terms
   \begin{equation}\label{defrn1}
\begin{aligned}
{ \tilde \rho^{(1),{\rm h}}_{n,\alpha,L}}\eqdefa  -   \sum_{
   \ell = 1  }^{L }  { \frac{ \varepsilon_n^{\ell} }{\gamma_n^{\ell} }}
 \Lambda_{ {\varepsilon_n^{\ell}} , {\gamma_n^{\ell}}, {x_n^{\ell}} }  \nabla^{\rm h}\Delta_{\rm h}^{-1} \partial_3  r_\alpha^{\ell}   -\nabla^{\rm h} \Delta_{\rm h}^{-1} \partial_3  {   \psi_n^{L} }
\end{aligned}
\end{equation}
and
   \begin{equation}\label{defrn2}
\begin{aligned}
 { \rho^{(2)}_{n,\alpha,L}}\eqdefa
   \sum_{\ell  = 1} ^{   L}
  \Lambda_{ {\eps_n^{\ell}} , {\gamma_n^{\ell}}, { x_n^{\ell}} } \big( \tilde  r_\alpha^{h,\ell}
,0 \big)  &+       \sum_{
   \ell = 1  }^{L }  \Lambda_{ {\varepsilon_n^{\ell}} , {\gamma_n^{\ell}}, {x_n^{\ell}} }  (0    ,   r_\alpha^{\ell}   )    +  \big( { \widetilde \psi_n^{h,  L} } ,  {\psi_n^{L} }\big)   \, .
\end{aligned}
\end{equation}
Observe that by construction,  thanks to~(\ref{smallremainderpsi}) and~(\ref{smallralpha}) and to the fact that if~$r_\alpha^\ell \not \equiv 0$, then~$\varepsilon^{\ell}_n/\gamma^{\ell}_n$ goes to zero as~$n$ goes to infinity, we have
\begin{equation}\label{limitrhon12}
\begin{aligned}
\lim_{L \to \infty } \lim_{\alpha \to 0}   \limsup_{n\to \infty} \|\tilde \rho^{(1),{\rm h}}_{\alpha,n,L}\|_{\cB^{1,-\frac12}} = 0 \, ,  \\
  \mbox{and} \quad
\lim_{L \to \infty } \lim_{\alpha \to 0}   \limsup_{n\to \infty} \|\rho^{(2)}_{\alpha,n,L}\|_{\cB^0} = 0 \,   .
\end{aligned}
\end{equation}
Then  we   notice that for each~$\ell \in \N$ and each~$\alpha \in ]0,1[$,
we have by a direct computation
   $$
\left\| \Lambda_{ {\eps_n^{\ell}} , {\gamma_n^{\ell}}, { x_n^{\ell}} } (\widetilde\phi_\alpha^{h, \ell},0)\right\|_{\cB^{1,-\frac12}} \sim    \left( \frac{\gamma_n^\ell
}{\eps_n^\ell} \right)^\frac12 \left\| \widetilde\phi_\alpha^{h, \ell}\right\|_{\cB^{1,-\frac12}}\,.
$$
We deduce that if~$\eps_n^\ell / \gamma_n^\ell \to \infty$, then ~$\Lambda_{ {\eps_n^{\ell}} , {\gamma_n^{\ell}}, {\tilde x_n^{\ell}} } (\widetilde\phi_\alpha^{h, \ell},0)$ goes to zero in~${\cB^{1,-\frac12}}$ as~$n$ goes to infinity, hence so does the sum over~$\ell \in \{1 , \dots , L\}$. It follows that  for each given~$\alpha$ in~$ ]0,1[$ and~$L \geq 1$ we may define
$$
\rho^{(1)}_{n,\alpha,L} \eqdefa  { \tilde \rho^{(1),{\rm h}}_{n,\alpha,L}}+ \sumetage{\ell = 1}{\eps_n^\ell / \gamma_n^\ell \to \infty}^L \Lambda_{ {\eps_n^{\ell}} , {\gamma_n^{\ell}}, {  x_n^{\ell}} } (\widetilde\phi_\alpha^{h, \ell},0)
$$
and we have
\begin{equation}\label{limitrhon3}
 \lim_{L \to \infty }   \lim_{\alpha \to 0}   \limsup_{n\to \infty} \| \rho^{(1)}_{n,\alpha,L}\|_{ \cB^{1,-\frac12}} = 0   \, .
\end{equation}
Finally, as~$\cD(\R^3)$ is dense in~$\cB_{1,q}$, let us choose a family~$(u_{0,\al})_\al$ of functions in~$\cD(\R^3)$ 
such that~$\|u_0-u_{0,\al}\|_{\cB_{1,q}} \leq\al$ and let us define
\begin{equation}
\label{defrn3}
 \rho_{n,\alpha,L}\eqdefa  \rho^{(1)}_{\alpha,n,L}+\rho^{(2)}_{n,\alpha,L}+u_0-u_{0,\al} \, .
\end{equation}
Inequalities~(\ref{limitrhon12}) and~(\ref{limitrhon3}) give
\begin{equation}\label{limitrhon333}
 \lim_{L \to \infty }   \lim_{\alpha \to 0}   \limsup_{n\to \infty}   \|e^{t\Delta} \rho_{n,\alpha,L}\|_{L^2(\R^+; \cB^{1})}  = 0   \, .
\end{equation}
%\begin{equation}\label{limitrhon333}
% \lim_{L \to \infty }   \lim_{\alpha \to 0}   \limsup_{n\to \infty} \big(  \|e^{t\Delta} \rho_{n,\alpha,L}\|_{L^2(\R^+; \cB^{1})} 
% +  \|  \rho^3_{n,\alpha,L}\|_{ \cB^{0}} 
% \big)= 0   \, .
%\end{equation}

\subsubsection{End of the proof of   Theorem~{\rm\ref{mainresultprofilesdi}}}
Let us return to the decomposition given in Proposition~\ref{prop:decompositiondatabis},  and use   definitions~(\ref{defrn1}),~(\ref{defrn2}) and~(\ref{defrn3})  which imply that
$$
\begin{aligned}
u_{0,n}  & =  u_{0,\al}   +
    \sumetage{
   \ell = 1}{\eps_n^\ell/\gamma_n^\ell \to 0}^{L } \Lambda_{ {\varepsilon_n^{\ell}} , {\gamma_n^{\ell}}, {x_n^{\ell}} } \Big( \widetilde  \phi_\alpha^{h,\ell} -\frac{ \varepsilon_n^{\ell} }{\gamma_n^{\ell} }
 \nabla^{\rm h}\Delta_{\rm h}^{-1} \partial_3  \phi_\alpha^{\ell}   ,  \phi_\alpha^{\ell} \Big)  +   \rho_{n,\alpha,L} \, .\nonumber
   \end{aligned}
   $$
   We recall that for all $ \ell$ in~$ \N$, we have
 $ \lim_{n \to \infty}  \: (\gamma_n^\ell)^{-1} \varepsilon_n^{\ell}   \in \{0 , \infty\}$ and in the case where  the ratio~$\eps_n^\ell/\gamma_n^\ell$ goes to infinity then~$ \phi_\alpha^{\ell} \equiv 0$.
   Next we separate the case when the horizontal scale is one, from the others: with the notation~(\ref{defLambdatilde}) we write
   $$
\begin{aligned}
u_{0,n}  =  u_{0,\al}  & +       \sum_{  \ell \in   \Gamma^L({ {1}})}  \Lambda_{ 1, {\gamma_n^{\ell}}, {  x_n^{\ell}} }  \left(\widetilde \phi_\alpha^{h,\ell}
 - \frac1{\gamma_n^\ell} \nabla^{\rm h} \Delta_{\rm h}^{-1} \partial_3  \phi_\alpha^{ \ell} , \phi_\alpha^{ \ell}
\right) \\
&+
  \sumetagetr{
   \ell = 1}{ {\varepsilon_n^{\ell}} \not \equiv 1}{\eps_n^\ell/\gamma_n^\ell \to 0} ^{L } \Lambda_{ {\varepsilon_n^{\ell}} , {\gamma_n^{\ell}}, {x_n^{\ell}} } \Big( \widetilde  \phi_\alpha^{h,\ell}    -\frac{ \varepsilon_n^{\ell} }{\gamma_n^{\ell} }
 \nabla^{\rm h}\Delta_{\rm h}^{-1} \partial_3  \phi_\alpha^{\ell}   ,  \phi_\alpha^{\ell} \Big)  +   \rho_{n,\alpha,L} \, .\nonumber
   \end{aligned}
   $$
With~(\ref{formula0})
%and the notation~(\ref{defv0}),~(\ref{defw0}) and~(\ref{defh0}),
this can be written
 $$
\begin{aligned}
u_{0,n}  =  u_{0,\al}  & +   \big[( v_{0,n,\alpha,L}^{0,{\rm loc,h}}+ h_n^0 w_{0,n,\alpha,L}^{0,{\rm loc,h}}  , w_{0,n,\alpha,L}^{0,{\rm loc},3} )\big]_{ h_n^0}  + \big[(  {v_{0,n,\alpha,L}^{0,\infty,{\rm   h}}+ h_n^0 w_{0,n,\alpha,L}^{0,\infty,{\rm   h}}},  {w_{0,n,\alpha,L}^{0,\infty,3}})\big]_{  {h_n^0}} \\
&+
  \sumetagetr{
   \ell = 1}{ {\varepsilon_n^{\ell}} \not \equiv 1}{\eps_n^\ell/\gamma_n^\ell \to 0} \Lambda_{ {\varepsilon_n^{\ell}} , {\gamma_n^{\ell}}, {x_n^{\ell}} } \Big( \widetilde  \phi_\alpha^{h,\ell}    -\frac{ \varepsilon_n^{\ell} }{\gamma_n^{\ell} }
 \nabla^{\rm h}\Delta_{\rm h}^{-1} \partial_3  \phi_\alpha^{\ell}   ,  \phi_\alpha^{\ell} \Big)  +   \rho_{n,\alpha,L} \, .\nonumber
   \end{aligned}
   $$
 Next we use the partition~(\ref{defpartition}), so that with notation~\refeq{deflitilde} and\refeq{defLambdatilde},
 $$
 \begin{aligned}
u_{0,n}  =  u_{0,\al}  & +  \big[( v_{0,n,\alpha,L}^{0,{\rm loc,h}}+ h_n^0 w_{0,n,\alpha,L}^{0,{\rm loc,h}}  , w_{0,n,\alpha,L}^{0,{\rm loc},3} )\big]_{ h_n^0}  + \big[(  {v_{0,n,\alpha,L}^{0,\infty,{\rm   h}}+ h_n^0 w_{0,n,\alpha,L}^{0,\infty,{\rm   h}}},  {w_{0,n,\alpha,L}^{0,\infty,3}})\big]_{  {h_n^0}}\\
& +\sum_{j = 1}^{  {\mathcal L}(L)}  \sumetage{
  \ell  \in \Gamma^L({ {\eps_n^{  \ell_j}}})}  { {\eps_n^{  \ell_j}} \not \equiv   1}  \Lambda_{ {\varepsilon_n^{\ell_j}} , {\gamma_n^{\ell}}, {x_n^{\ell}} } \Big(\widetilde  \phi_\alpha^{h,\ell}       -\frac{ \varepsilon_n^{\ell_j} }{\gamma_n^{\ell} }
 \nabla^{\rm h}\Delta_{\rm h}^{-1} \partial_3  \phi_\alpha^{\ell}   ,  \phi_\alpha^{\ell} \Big)  +   \rho_{n,\alpha,L} \, .\nonumber
   \end{aligned}
 $$
Then we finally use the identity~(\ref{formulajtilde}) which gives
 $$
\begin{aligned}
     u_{0,n}  = u_{0,\al} &+ \big[( v_{0,n,\alpha,L}^{0,{\rm loc,h}}+ h_n^0 w_{0,n,\alpha,L}^{0,{\rm loc,h}}  , w_{0,n,\alpha,L}^{0,{\rm loc},3} )\big]_{ h_n^0}  + \big[(  {v_{0,n,\alpha,L}^{0,\infty,{\rm   h}}+ h_n^0 w_{0,n,\alpha,L}^{0,\infty,{\rm   h}}},  {w_{0,n,\alpha,L}^{0,\infty,3}})\big]_{  {h_n^0}}\\
   & \quad   {} + \sum_{j = 1}^L \Lambda_{  \lambda_n^j} [(v_{n,\alpha,L}^{j,{\rm h}}+ h_{n}^j w_{n,\alpha,L}^{j,{\rm h}},w_{n,\alpha,L}^{j,3})]_{h_{n}^j} + \rho_{n,\alpha,L}\, .
    \end{aligned} $$
 The end of the proof follows from the estimates~(\ref{boundv0w0-unif}),~(\ref{boundv0w0-notunif}),~(\ref{boundv0w0-atzero}),~(\ref{boundvjwj-unif}),  along with~(\ref{limitrhon333}). 
  Theorem~\ref{mainresultprofilesdi} is proved.
 \end{proof}

\medbreak

 %%%%%%%%%%%%%%%%%%%%%%%%%%%%%%%%%%%%%%%%%%

 \section{Propagation of  profiles: proof of Theorem\refer{slowvarsimple}}
\label{propagationprofiles}
 The goal of this section is the proof of Theorem\refer{slowvarsimple}. Let us consider~$(v_0,w^3_0)$ satisfying the assumptions of that theorem. In order to prove that the initial data defined by 
 $$
 \Phi_{0} \eqdefa \big[  (  v_0 -\b\nabla^{\rm h} \Delta_{\rm h}^{-1} \partial_3w^3_0, w^3_0 )  \big]_{\b}
 $$
generates a global smooth solution for small enough~$\b$, let us  look for the solution under the form
\begin{equation}
\label{defuapppropagaprofile}
\Phi_\beta=\Phi^{\rm app}+\psi \with 
\Phi^{\rm app}\eqdefa \big[(   v + \beta  w^{\rm h} , w^3  ) \big]_{\b }
\end{equation}
 where~$v$  solves the two-dimensional Navier-Stokes equations
$$
{\rm(NS2D)}_{x_3} \quad  \left\{
\begin{array}{l}
\partial_t  v +  v\cdot \nabla^{\rm h} v -\Delta_{\rm h}  v= -\nabla^{\rm h} p  \quad \mbox{in} \, \, \, 
\R^+ \times \R^2\\
\mbox{div} _{\rm h}  v= 0\\
v_{|t=0} = v_0(\cdot ,x_3) \, ,
\end{array}
\right.
$$
while~$w^3 $ solves the transport-diffusion equation\label{defTbetapage}
$$
( T_\b ) \quad \left\{
\begin{array}{l}
\partial_t w^3  +  v\cdot \nabla^{\rm h}w^3 -\Delta_{\rm h} w^3 - \b^2 \partial_3^2 w^3 = 0 \quad \mbox{in}\, \, \, 
\R^+ \times \R^3\\
w^3_{|t=0} = w_0^3 \,
\end{array}
\right.
$$
and~$w ^{\rm h} $ is determined by the divergence free condition on~$w$ which gives~$w ^{\rm h}\eqdefa -\nabla^{\rm h} \Delta_{\rm h}^{-1} \partial_3w^3 $.

\medskip

\noindent
\textcolor{black}{In Section~\ref{twodimparameter} (resp. \ref{propagtwodimparameter}), we prove a priori estimates on~$v$ (resp. $w$), and Section~\ref{conclusionproofpropagprofile} is devoted to the conclusion of the proof of Theorem\refer{slowvarsimple}, studying the perturbed Navier-Stokes equation satisfied by~$\psi$.}

\medbreak

\noindent Before starting the proof we  recall the following definitions of space-time norms, first introduced by J.-Y. Chemin and N. Lerner in~\cite{cheminlerner},  and which are    very useful in the context of the Navier-Stokes equations:
\begin{equation}
\label{deflrtildenorm}
 \|f\|_{{\widetilde L^r}([0,T]; B^{s,s'}_{p,q} )} \eqdefa \big\|2^{ks + js'}\|\Delta_k^{\rm h} \Delta_{j}^{\rm v} f\|_{L^r([0,T];L^p)}\big\|_{\ell^q} \, .
\end{equation}
 \noindent Notice that of course~$ {\widetilde L^r}([0,T]; B^{s,s'}_{p,r} ) =  {L^r}([0,T]; B^{s,s'}_{p,r} )$, and by Minkowski's inequality, we have the embedding~$ {\widetilde L^r}([0,T]; B^{s,s'}_{p,q} ) \subset  {L^r}([0,T]; B^{s,s'}_{p,q} )$ if~$r \geq q$.

\subsection{Two dimensional flows with parameter}\label{twodimparameter}
Let us prove the following result on~$v$, the solution of~${\rm(NS2D)}_{x_3}$.

\begin{prop}
\label{regulNS2D}
{\sl
 Let~$ v_0$ be a  {two-component} divergence free vector field depending on the vertical variable~$x_3$, and belonging  to~$S_\mu$. Then the unique, global solution~$v$ to~${\rm(NS2D)}_{x_3}$ belongs to~$\cA^0$ and satisfies the following estimate:
 \begin{equation}
 \label{eqdim21}
 \|  v \|_{\cA^0} \leq \cT_1(\| v_0\|_{\cB^0}) \,. 
 \end{equation}
 Moreover, for all~$(s,s')$ in~$D_\mu$,  we have
 \beq
\label{eqdim22}
 \forall r \in [1,\infty]\,,\  \|v \|_{\wt L^r(\R^+;\cB^{s+\frac 2 r,s'})}
 \leq \cT_2(\|v_0\|_{S_\mu}).
 \eeq
} \end{prop}
\begin{proof}
This proposition
 is  a result about the regularity of the solution of ${\rm(NS2D)}$ when the initial data depends
on a real parameter~$x_3$,  measured in terms of Besov spaces with respect to the variable~$x_3$. Its proof is structured as follows. First, we deduce from the classical energy estimate for the two dimensional Navier-Stokes system, a stability result  in the spaces~$ {L}^r(\R^+;  H^{s+\frac2r}(\R^2))$ with~$r$ in~$[2,\infty]$ and $s$ in $]-1,1[$.  This is the purpose of Lemma\refer{regulNS2DdemoLemme1}, the proof of  which uses essentially energy estimates together with   paraproduct laws.

\noindent Then we have to  translate the  stability result  of Lemma\refer{regulNS2DdemoLemme1} in terms of Besov spaces with respect to the third variable (seen before simply as a parameter), namely by propagating the vertical regularity. First of all,  this requires to deduce from the stability in the spaces~${L}^r(\R^+;  H^{s+\frac2r}(\R^2))$  with~$r$ in~$[2,\infty]$, the fact  that the   vector field~$v $, now seen as a function of three variables, belongs to~${L^r}(\R^+;L^\infty_{\rm v}(  H^{s+\frac2r}(\R^2))$ again for~$r$ in~$[2,\infty]$.
This is the purpose of Lemma\refer{lemmedemoeq},  the proof of which relies on the equivalence of   two definitions of Besov spaces with regularity index in~$]0,1[$:  the first one involving the dyadic decomposition of the   frequency space, and the other one consisting in estimating integrals in   physical space.   

\noindent Finally for $s$ in $]-\frac 12,\frac 12[$ and $s' > 0$ a Gronwall type lemma  enables us  to propagate  the regularities.   When $\ds s' \geq \frac 12$    product  laws enable us  to gain horizontal regularity up to $\ds ]-2,1[$   and to conclude the proof of Proposition~\ref{regulNS2D}. 
\medbreak
\noindent Let us state the first lemma in this proof.
\begin{lem}
\label{regulNS2DdemoLemme1}
{\sl
For any compact set~$I$ included in~$]-1,1[$, a constant $C$ exists such that, for any $r$ in~$[2,\infty]$  and any $s$ in  ~$I$,~we have  for any two solutions~$v_1$ and~$v_2$ of the two-dimensional Navier-Stokes equations
\begin{equation}
\label{eqdim2}
\|v_1-v_2\|_{ {L}^r(\R^+;  H^{s+\frac2r}(\R^2))}\lesssim
\|v_1(0)-v_2(0)\|_{  H^s(\R^2)} \, E_{12}(0) \,  ,
\end{equation}
where we define
$$
E_{12}(0)\eqdefa \exp C \bigl(\|v_1(0)\|_{L^2}^2+\|v_2(0)\|_{L^2}^2\bigr) \, .
$$
}
\end{lem}
\begin{proof}
In the proof of this lemma, all the functional spaces are over~$\R^2$ and we no longer mention this fact in notations. Moreover, the constant which appears in the definition of~$E_{12}(0)$ can change along the proof.  Defining~$v_{12}(t)\eqdefa v_1(t)-v_2(t)$, we get
\beq
\label{regulNS2DdemoLemme1demoeq1}
\begin{split}
\partial_tv_{12}+v_2\cdot\nabla^{\rm h}{}v_{12} - \D_{\rm h} v_{12} &=-v_{12}\cdot\nabla^{\rm h}{}v_1 -\nabla^{\rm h} p \,.
\end{split}
\eeq
In order to establish \eqref{eqdim2}, we shall resort to  an energy estimate making use of  product laws and of the following estimate proved in  
\cite[Lemma 1.1]{chemin10}: 
 \beq
\label{eschem}
\big(v\cdot\nabla^{\rm h}a|a\big)_{  H^s}  \lesssim   \|\nabla^{\rm h} v\|_{L^2} \|a\|_{  H^s} \|\nabla^{\rm h} a\|_{  H^s},
\eeq
available uniformly  for any $s$ in~$[-2+\mu,1-\mu]$. 

\noindent Let us notice that thanks to the divergence free condition, taking the $  H^s$ scalar product  with $v_{12}$ in Equation \eqref{regulNS2DdemoLemme1demoeq1} implies that 
$$ \frac 1 2 \frac d { dt} \|v_{12}(t)\|_{  H^s}^2 +  \|\nabla^{\rm h} v_{12}(t)\|_{  H^s}^2= -\big(v_2(t)\cdot\nabla^{\rm h}v_{12}(t)|v_{12}(t)\big)_{  H^s} -\big(v_{12}(t)\cdot\nabla^{\rm h}v_{1}(t)|v_{12}(t)\big)_{  H^s} \, .
$$  
Whence, by time integration we get 
$$
\begin{aligned} \|v_{12}(t)\|_{  H^s}^2 + 2 \int_0^t \|\nabla^{\rm h} v_{12}(t')\|_{  H^s}^2 dt'  =   \|v_{12}(0)\|_{ H^s}^2 - 2 \int_0^t \big(v_2(t')\cdot\nabla^{\rm h}v_{12}(t')|v_{12}(t')\big)_{ H^s} \, dt' \\     - 2 \int_0^t \big(v_{12}(t')\cdot\nabla^{\rm h}v_{1}(t')|v_{12}(t')\big)_{ H^s}  \, dt' \, . 
\end{aligned}
$$
Now using Estimate  \eqref{eschem}, we deduce that there is a positive constant $C$ such that for any $s$ in~$I$, we have 
\begin{equation}\label{term1}
\begin{aligned}
 2  \, \Big|\int_0^t \big(v_2(t') & \cdot\nabla^{\rm h}v_{12}(t')|v_{12}(t')\big)_{ H^s} dt' \Big| \\
& \leq C \int_0^t \|v_{12}(t')\|_{ H^s} \|\nabla^{\rm h} v_2(t')\|_{L^2}  \|\nabla^{\rm h} v_{12}(t')\|_{ H^s}  \,  dt'
\\
&\leq\frac 1 2 \int_0^t \|\nabla^{\rm h} v_{12}(t')\|_{ H^s}^2  \, dt'  + \frac {C^2} 2 \int_0^t \|v_{12}(t')\|^2_{ H^s} \|\nabla^{\rm h} v_2(t')\|^2_{L^2}  \,    dt' \, .
\end{aligned}
\end{equation}
Noticing that 
$$\int_0^t \big(v_{12}(t')\cdot\nabla^{\rm h}v_{1}(t')|v_{12}(t')\big)_{ H^s} dt' \leq \int_0^t  \|\nabla^{\rm h} v_{12}(t')\|_{ H^s} \|v_{12}(t')\cdot\nabla^{\rm h}v_{1}(t')\|_{ H^{s-1}}   \, dt' \,,$$
we deduce  by Cauchy-Schwarz inequality and product laws in Sobolev spaces on~$\R^2$ that as long as~$s$ is in~$]0,1[$,  
\begin{equation}\label{term2}
\begin{aligned}
 2  \, \Big|\int_0^t \big(v_{12}(t')&\cdot\nabla^{\rm h}v_{1}(t')|v_{12}(t')\big)_{ H^s} dt' \Big| 
 \\
&\leq  C \int_0^t \|\nabla^{\rm h} v_{12}(t')\|_{ H^s}  \|v_{12}(t')\|_{ H^s} \|\nabla^{\rm h} v_1(t')\|_{L^2} \,     dt'  \\
&\leq \frac 1 2 \int_0^t \|\nabla^{\rm h} v_{12}(t')\|_{ H^s}^2 \,  dt'  + \frac {C^2} 2 \int_0^t \|v_{12}(t')\|^2_{ H^s} \|\nabla^{\rm h} v_1(t')\|^2_{L^2}   \,   dt'\, .
\end{aligned}
\end{equation}
 {\color{black}When~$s=0$ we simply write, by product laws and interpolation,
\begin{equation}\label{term22}
\begin{aligned}
 2  \, \Big|\int_0^t \big(v_{12}(t')&\cdot\nabla^{\rm h}v_{1}(t')|v_{12}(t')\big)_{L^2} dt' \Big| 
 \\
&\leq  C \int_0^t \| v_{12}(t')\|_{ H^\frac12}  \|v_{12}(t')\cdot \nabla^{\rm h} v_1(t')\|_{ H^{-\frac12}}    dt'  \\
&\leq \frac 1 2 \int_0^t \|\nabla^{\rm h} v_{12}(t')\|_{L^2}^2  \, dt'  + \frac {C^2} 2 \int_0^t \|v_{12}(t')\|^2_{L^2} \|\nabla^{\rm h} v_1(t')\|^2_{L^2}    \,  dt'\, .
\end{aligned}
\end{equation}
Finally in the case when~¬¨‚Ä†$s$ belongs to~$]-1,0[$, we have
\begin{equation}\label{term23}
\begin{aligned}
 2  \, \Big|\int_0^t \big(v_{12}(t')&\cdot\nabla^{\rm h}v_{1}(t')|v_{12}(t')\big)_{ H^s} dt' \Big| 
 \\
&\leq  C \int_0^t \| v_{12}(t')\|_{ H^s}  \|v_{12}(t')\cdot \nabla^{\rm h} v_1(t')\|_{ H^{s}}    dt'  \\
&\leq \frac 1 2 \int_0^t \|\nabla^{\rm h} v_{12}(t')\|_{ H^s}^2  \, dt'  + \frac {C^2} 2 \int_0^t \|v_{12}(t')\|^2_{ H^s} \|\nabla^{\rm h} v_1(t')\|^2_{L^2}   \,   dt'\, .
\end{aligned}
\end{equation}}
Combining \eqref{term1} and \eqref{term2}-\eqref{term23}, we infer that  for~$s$ in~$]-1,1[$,
$$
\begin{aligned}
\|v_{12}(t)\|_{ H^s}^2 +  \int_0^t \|\nabla^{\rm h} v_{12}(t')\|_{ H^s}^2 dt'   &  \lesssim    \|v_{12}(0)\|_{ H^s}^2 \\
 &\quad +  \int_0^t \|v_{12}(t')\|^2_{ H^s} \big(\|\nabla^{\rm h} v_1(t')\|^2_{L^2}  + \|\nabla^{\rm h} v_2(t')\|^2_{L^2}\big)   \,  dt'  \, . 
\end{aligned}
$$
Gronwall's lemma implies that there exists a positive constant $C$ such that 
$$ \|v_{12}(t)\|_{ H^s}^2 +  \int_0^t \|\nabla^{\rm h} v_{12}(t')\|_{ H^s}^2 dt' \lesssim  \|v_{12}(0)\|_{ H^s}^2 \exp C \int_0^t  \big(\|\nabla^{\rm h} v_1(t')\|^2_{L^2}  + \|\nabla^{\rm h} v_2(t')\|^2_{L^2}\big)   dt'\, . $$ 
But for any $i$ in~‚àö¬¢¬¨√†¬¨√∂‚àö√â‚Äö√á¬®$ \{1,2\}$, we have by the classical~$L^2$ energy estimate 
\begin{equation}\label{standen}
\int_0^t \|\nabla^{\rm h} v_{i}(t')\|^2_{L^2} dt'  \leq \frac 1 2 \| v_i (0)\|^2_{L^2}\,.
\end{equation}
Consequently   for~$s$ in~$]-1,1[$,
$$ \|v_{12}(t)\|_{ H^s}^2 +  \int_0^t \|\nabla^{\rm h} v_{12}(t')\|_{ H^s}^2 dt' \lesssim  \|v_{12}(0)\|_{ H^s}^2 \, E_{12}(0) \,  ,$$ 
which leads to the result by interpolation. 
\end{proof}

\medbreak

\noindent{\it Continuation of the proof of Proposition{\rm\refer{regulNS2D}}. }
\noindent Using    Lemma \ref{regulNS2DdemoLemme1}, we are going to establish  the following result, which will be of great help to control all   norms of~$v$ of the type~$\wt{ L} ^r(\R^+;\cB^{ \frac 2 r})$ for~$r$ in~$[4,\infty]$
thanks to  a Gronwall type argument. \begin{lem}
\label{lemmedemoeq}  {\sl For any compact set~$I$ included in~$]-1,1[$, a constant $C$ exists such that, for any~$r$ in~$[2,\infty]$  and   any $s$ in~$I$, we have for any solution~$v$ to~${\rm(NS2D)}_{x_3}$, 
$$
\|v\|_{L^r(\R^+; L^\infty_{\rm v}( H_{\rm h}^{s+\frac2r}))} \lesssim  \|v_0\|_{\cB^s}   E (0)
\with E (0)\eqdefa  \exp \bigl(C \|v (0)\|_{L^\infty_{\rm v} L^2_{\rm h}}^2\bigr).
$$
}\end{lem}
\begin{proof}
We shall use the characterization of Besov spaces via differences in physical space:  as is well-known (see for instance Theorem~2.36 of\ccite{BCD}), for any Banach space~$X$ of distributions   one has
\beq
\label{equivsobtranslat}
\bigl\|\bigl(2^{\frac j2} \|\D^{\rm v}_j u\|_{L^2_{\rm v} (X)}\bigr)_j\bigr\|_{\ell^1(\ZZ)}
\sim
\int_{\R}\frac { \|u-(\tau_{-z} u) \|_{
L^2_{\rm v}( X)}} {|z|^{\frac12}} \frac {dz }{|z|}
\eeq
where  the translation operator~$\tau_{-z}$ is defined by
$$
(\tau_{-z} f) (t,x_{\rm h},x_3) \eqdefa f (t,x_{\rm h},x_3+ z) \, .
$$
The above Lemma\refer{regulNS2DdemoLemme1}  implies in particular that, for any~$r$ in~$[2,\infty]$, any~$s$  in~$I$ and any couple~$(x_3,z)$ in~$ \R^2$, if~$v$ solves~${\rm(NS2D)}_{x_3}$ then
$$
\|v -\tau_{-z}v\|_{Y^{s}_r}\lesssim  \|v_0 -\tau_{-z}v_0\|_{ H^s_{\rm h}}  E (0) \with Y^s_r \eqdefa L^r(\R^+; H_{\rm h}^{s+\frac2r} )\, .
$$
Taking the~$L^2$ norm of the above inequality with respect to the~$x_3$ variable and then the~$L^1$ norm with respect to the measure~$|z|^{-\frac32}dz$ gives
\begin{equation}\label{estimatev-tauz}
\int_\R  \frac {\|v -\tau_{-z}v \|_{L^2_{\rm v}(Y^s_r)} }{|z|^{\frac12}} \frac {dz}{|z|}
 \lesssim \int_{\R}\frac { \|v_0 -\tau_{-z}v_0\|_{L^2_{\rm v}(H^s_{\rm h})}}{|z|^{\frac12}}
\frac{dz}{|z|}\, E(0)\, .
\end{equation}
Returning to the characterization~(\ref{equivsobtranslat}) with ~$X=Y^s_r$, we find
that
$$\int_\R  \frac {\|v -\tau_{-z}v \|_{L^2_{\rm v}(Y^s_r)} }{|z|^{\frac12}} \frac {dz}{|z|}
 \sim\sum_{j \in \ZZ}2^{\frac j 2} \Bigl\|\bigl\|\bigl(2^{k(s+\frac2r)} \D_j^{\rm v}\D^{\rm h}_k v (t,\cdot,z)\big)_k \bigl\|_{L^r(\R^+;\ell^2(\ZZ;L^2_{\rm h}))}\Bigr\|_{L^2_{\rm v}} \, .
$$
Similarly we have
$$
 \int_{\R}\frac { \|v_0 -\tau_{-z}v_0\|_{L^2_{\rm v}(H^s_{\rm h})}}{|z|^{\frac12}}
\frac{dz}{|z|} \sim \sum_{j \in \ZZ} 2^\frac{j}2\bigl\|\bigl(2^{ks} \|\D_j^{\rm v}\D^{\rm h}_k v_0 \|_{  L^2_{\rm h}}\bigr)_k\bigr\|_{\ell ^2(\ZZ;L^2_{\rm v})}\, ,
$$
so by   the embedding from~$\ell^1(\ZZ)$ to~$\ell^2(\ZZ)$, we get
$$
  \int_{\R}\frac { \|v_0 -\tau_{-z}v_0\|_{L^2_{\rm v}(H^s_{\rm h})}}{|z|^{\frac12}}
\frac{dz}{|z|}  \lesssim  \sum_{(j,k) \in \ZZ^2}2^\frac{j}2 2^{ks} \|\D_j^{\rm v}\D^{\rm h}_k v_0 \|_{ L^2(\R^3)}\, .
$$
Therefore, we deduce from Estimate~(\ref{estimatev-tauz}) that 
$$
\sum_{j \in \ZZ} 2^{\frac j 2} \Bigl\|\bigl\|\bigl(2^{k(s+\frac2r)} \D_j^{\rm v}\D^{\rm h}_k v (t,\cdot,z)\big)_k\bigr\|_{L^r(\R^+;\ell ^2(\ZZ;L^2_{\rm h}))}\Bigr\|_{L^2_{\rm v}}
\lesssim  \|v_0\|_{\cB^s} \, E (0) \, .
$$
As~$r\geq 2$,
Minkowski's inequality implies that
$$
\sum_{j \in \ZZ} 2^{\frac j 2} \Bigl\|\bigr\|\bigl(2^{k(s+\frac2r)} \D_j^{\rm v}\D^{\rm h}_k v (t,\cdot)\big)_k\bigr\|_{\ell ^2(\ZZ;L^2 )}\Bigr\|_{L^r(\R^+)}
\lesssim  \|v_0\|_{\cB^s} \, E (0) \, .
$$
Bernstein's inequality as stated in Lemma\refer{Bernsteinaniso} implies that
$$
 \|\D_j^{\rm v}\D^{\rm h}_k v(t,\cdot) \|_{L^\infty_{\rm v}(L^2_{\rm h})}
 \lesssim
 2^{\frac j2} \|\D_j^{\rm v}\D^{\rm h}_k v(t,\cdot) \|_{L^2 } \, ,
$$
thus we infer  that
$$
\Bigl\| \bigl\| \bigl(2^{k(s+\frac2r)} \|\D^{\rm h}_k v \|_{L^\infty_{\rm v}(L^2_{\rm h})}\bigr)_k\bigr\|_{\ell^2(\ZZ)}\Bigr\|_{L^r(\R^+)}\lesssim  \|v_0\|_{\cB^s}\, E (0)\, .
$$
Permuting the~$\ell^2$ norm and the~$L^\infty_{\rm v}$ norm thanks to Minkowski's inequality again, concludes the proof of the lemma.
\end{proof}

\medbreak

\begin{rmk}
\label{inclusionXrp}
{\sl Let us remark that thanks to the Sobolev embedding of~$ H^{\frac12}(\R^2)$ into~$L^4(\R^2)$, we have, choosing~$s=0$ and~$r = 4$ or~$r=2$,
$$
\|v\|_{L^4(\R^+; L^\infty_{\rm v}(L^4_{\rm h}))} + \|v\|_{L^2(\R^+;L^\infty_{\rm v}(  H^1_{\rm h}))} \lesssim  \|v_0\|_{\cB^0} \,E (0) \, .
$$
}
\end{rmk}
\noindent{\it Continuation of the proof of Proposition{\rm\refer{regulNS2D}}. }
Now our purpose is the proof of the following  inequality: for any~$v$ solving~${\rm(NS2D)}_{x_3}$, for any~$r$ in~$ [4,\infty]$ and any~$s$ in~$ \Bigl]-\frac 12,\frac 12\Bigr[$ and any positive~$s' $,
\beq
\label{regulNS2Ddemoeq20} \|v\|_{\wt L^r(\R^+;\cB^{s+\frac2r,s'})} \lesssim
\|v_0\|_{ \cB^{s,s'}} \exp \Bigl(\int_0^\infty C\bigl( \|v(t)\|_{L^\infty_{\rm v}(L^4_{\rm h}))}^4
+\|v(t)\|^2_{L^\infty_{\rm v}( H^1_{\rm h})}\bigr)dt\Bigr) \, .\eeq
The case when~$r$ is in~$ [2,4]$ will be dealt with later.
We are going to use a Gronwall-type argument. Let us introduce, for any nonnegative~$\lam$, the following notation: for any function~$F$ we define
$$
F_\lam (t) \eqdefa F(t)\exp \Bigl(-\lam \int_0^t\phi(t') dt' \Bigr)\with \phi(t) \eqdefa \|v(t)\|_{L^\infty_{\rm v}(L^4_{\rm h})}^4+\|v(t)\|^2_{L^\infty_{\rm v}( H^1_{\rm h})} \, .
$$
Notice that thanks to Remark~\ref{inclusionXrp}, we know that
\beq
\label{inclusionXrpeq}
\int_0^t \phi(t') \, dt' \lesssim E(0) (\|v_0\|_{\cB^0}^2+\|v_0\|_{\cB^0}^4 ) \, .
\eeq
Then we write, using the Duhamel formula and the action of the heat flow described in Lemma\refer{anisoheat},  that
\beq
\label{regulNS2Ddemoeq21}
\begin{split}
&\|\D_j^{\rm v}\D_k^{\rm h}v_\lam(t)\|_{L^2} \leq C e^{-c2^{2k}t} \|\D_j^{\rm v}\D_k^{\rm h}v_0\|_{L^2}\\
&\qquad\qquad{}
+C 2^k \int_0 ^t \exp \Bigl(-c(t-t')2^{2k} -\lam \int_{t'}^t\phi(t'')dt'' \Bigr)
\|\D_j^{\rm v}\D_k^{\rm h}(v \otimes v)_\lam(t')\|_{L^2} dt' \, .
\end{split}
\eeq
Notice that~$(v\otimes v)_\lam = v\otimes v_\lam$.
In order to study the term~$\|\D_j^{\rm v}\D_k^{\rm h}(v \otimes v)_\lam(t')\|_{L^2} $, we need an anisotropic version of Bony's paraproduct decomposition. Let us write that 
\ben
\nonumber ab & = & \sum_{\ell=1}^4 T^\ell(a,b)\with\\
\nonumber T^1(a,b) & = &\sum_{j,k} S^{\rm v}_jS^{\rm h}_k a \D^{\rm v}_j\D^{\rm h}_k b\,,\\
\label{Bonydecompaniso}T^2(a,b) & = &\sum_{j,k} S^{\rm v}_j\D^{\rm h}_k a \D^{\rm v}_jS^{\rm h}_{k+1} b\,,\\
\nonumber T^3(a,b) & = &\sum_{j,k} \D^{\rm v}_jS^{\rm h}_{k} a S^{\rm v}_{j+1}\D^{\rm h}_k b\,,\\
\nonumber T^4(a,b) & = &\sum_{j,k} \D^{\rm v}_j\D^{\rm h}_k a S^{\rm v}_{j+1}S^{\rm h}_{k+1} b \, .
\een

%%%%%%%%%
%%%%%%%%%
\noindent We shall only estimate~$T^1$ and~$T^2$, the other two terms being strictly analogous. By definition of~$T^1$, using the definition of  horizontal and vertical truncations together with the fact that the support of the Fourier  transform of the product of two functions is included in the sum of the two supports,  and Bernstein's  and H\"older's inequalities, there is some fixed nonzero integer  $N_0$ such that 
\beno
\|\D_j^{\rm v}\D_k^{\rm h}T^1(v(t) ,v_\lam(t))\|_{L^2}  &  \lesssim  &  2^{ \frac k 2}\|\D_j^{\rm v}\D_k^{\rm h}T^1(v(t) ,v_\lam(t))\|_{L^2_{\rm v}(L_{\rm h}^{\frac43})} \\
&  \lesssim & 2^{ \frac k 2} \sumetage{j'\geq j-N_0}{k'\geq k-N_0}
\|S_{j'}^{\rm v}S_{k'}^{\rm h} v(t) \|_{L^\infty_{\rm v}(L_{\rm h}^4)} \|\D_{j'}^{\rm v}\D_{k'}^{\rm h} v_\lam(t)\|_{L^2}
\\
& \lesssim & 2^{ \frac k 2} \|v(t) \|_{L^\infty_{\rm v}(L^4_{\rm h})}
\sumetage{j'\geq j-N_0}{k'\geq k-N_0}   \|\D_{j'}^{\rm v}\D_{k'}^{\rm h} v_\lam(t)\|_{L^2} \, .
\eeno
By definition of~$\wt {L}^4(\R^+; \cB^{s+\frac12, s'})$ we get
$$
\|\D_j^{\rm v}\D_k^{\rm h}T^1(v(t) ,v_\lam(t))\|_{L^2}  \lesssim  2^{ \frac k 2} \|v_\lam\|_{\wt {L}^4(\R^+; \cB^{s+\frac12, s'})}\|v(t) \|_{L^\infty_{\rm v}(L^4_{\rm h})} \sumetage{j'\geq j-N_0}{k'\geq k-N_0}  2^{-k'(s+\frac12)}2^{-j's'}
\wt f_{j',k'}(t)
$$
where~$\wt f_{j',k'}(t)$, defined by
$$
\wt f_{j',k'}(t) \eqdefa \|v_\lam\|_{\wt {L}^4(\R^+; \cB^{s+\frac12, s'})}^{-1}  2^{k'(s+\frac12 )} 2^{j's'}
 \|\D_{j'}^{\rm v}\D_{k'}^{\rm h} v_\lam(t)\|_{L^2}  \, ,
$$
 is on the sphere of~$\ell^1(\ZZ^2;L^4(\R^+))$. This implies that
$$
\longformule
{
2^{js'}2^{ks}\|\D_j^{\rm v}\D_k^{\rm h}T^1(v(t) ,v_\lam(t))\|_{L^2}
}
{{}\lesssim
 \|v_\lam\|_{\wt {L}^4(\R^+; \cB^{s+\frac12, s'})}  \|v(t) \|_{L^\infty_{\rm v}(L^4_{\rm h})}
\sumetage{j'\geq j-N_0} {k'\geq k-N_0} 2^{-(j'-j)s'} 2^{-(k'-k)(s+\frac12) }  \wt f_{j',k'}(t) \,  .}
$$
%Defining
%\beq
%\label{regulNS2Ddemoeq22}
%\begin{split}
%&d_{j,k} := C\sumetage{j'\geq j-N_0} {k'\geq k-N_0} 2^{-(j'-j)s'} 2^{-2\frac { k'-k} p} \wt d_{j',k'}\andf\\
%&f_{j,k}(t)  := \frac 1 {C d_{j,k} }
%\sumetage{j'\geq j-N_0} {k'\geq k-N_0} 2^{-(j'-j)s'} 2^{-2\frac { k'-k} p} \wt d_{j',k'} \wt f_{j',k'}(t) \,   ,
%\end{split}
%\eeq
%we get that
Since $\ds s > - \frac12$ and $s'> 0$,  it follows  by Young's inequality on series, that
$$
2^{js'}2^{k s}\|\D_j^{\rm v}\D_k^{\rm h}T^1(v(t) ,v_\lam(t))\|_{L^2} \lesssim
 \|v_\lam\|_{\wt {L}^4(\R^+; \cB^{s+\frac12, s'})}  \|v(t) \|_{L^\infty_{\rm v}(L^4_{\rm h})}   f_{j,k}(t)
$$
where~$f_{j,k}(t)$ is on the sphere of~$\ell^1(\ZZ^2;L^4(\R^+))$.
As~$\phi(t)$ is greater than~$\|v(t)\|^4_{L^\infty_{\rm v}(L^4_{\rm h})}$, we infer that
\beq
\label{regulNS2Ddemoeq23-1}
\begin{split}
\cT_{j,k,\lam}^1(t) & \eqdefa
2^k 2^{js'}2^{k s} \int_0 ^t \exp \Bigl(-c(t-t')2^{2k} -\lam \int_{t'}^t\phi(t'')dt'' \Bigr)\\
&\qquad\quad\qquad\qquad\quad\qquad \qquad\quad\qquad{}\times
\|\D_j^{\rm v}\D_k^{\rm h}T^1(v(t'), v_\lam(t'))\|_{L^2} dt'\\
& \lesssim   \|v_\lam\|_{\wt {L}^4(\R^+; \cB^{s+\frac12, s'})} \\
& \qquad\qquad  {}\times
2^{k}
\int_0 ^t   \exp \Bigl(-c(t-t')2^{2k} -\lam \int_{t'}^t\phi(t'')dt'' \Bigr)
\phi^{\frac 1 4} (t') f_{j,k}(t')dt' \,  .
\end{split}
\eeq
Using H\"older's  inequality, we deduce  that
$$
\longformule{
\cT_{j,k,\lam}^1(t)  \lesssim   \|v_\lam\|_{\wt {L}^4(\R^+; \cB^{s+\frac12, s'})} \biggl(\int_0 ^t   e^{-c(t-t')2^{2k} } f^4_{j,k}(t')dt'\biggr)^{\frac 1 4}
}
{
{}\times 2^{k} \biggl(\int_0 ^t   \exp \Bigl(-c(t-t')2^{2k} -\lam \int_{t'}^t\phi(t'')dt'' \Bigr)
\phi(t')^{\frac 1 3}dt'  \biggr)^{\frac 34 } \,  .
}
$$
Then H\"older's inequality  in the last term of the above inequality ensures that
\beq
\label{regulNS2Ddemoeq23}
\cT_{j,k,\lam}^1(t)  \lesssim   \frac 1 {\lam^{\frac 1 4}}\biggl(\int_0 ^t   e^{-c(t-t')2^{2k} } f^4_{j,k}(t')dt'\biggr)^{\frac 1 4}  \|v_\lam\|_{\wt {L}^4(\R^+; \cB^{s+\frac12, s'})} \,  .
\eeq
Now let us study the term with~$T^2$. Using again that  the support of the Fourier  transform of the product of two functions is included in the sum of the two supports, let us write that
$$
\|\D_j^{\rm v}\D_k^{\rm h}T^2(v(t) ,v_\lam(t))\|_{L^2}   \lesssim  \sumetage{j'\geq j-N_0}{k'\geq k-N_0}
\|S_{j'}^{\rm v}\D_{k'}^{\rm h} v(t) \|_{L^\infty_{\rm v}(L^2_{\rm h})} \|\D_{j'}^{\rm v}S_{k'+1}^{\rm h} v_\lam(t)\|_{L^2_{\rm v}(L^\infty_{\rm h})}\, .
$$
Combining  Lemma\refer{Bernsteinaniso} with  the definition of the function~$\phi$, we get \beq
\label{regulNS2Ddemoeq24}
\|S_{j'}^{\rm v}\D_{k'}^{\rm h} v(t) \|_{L^\infty_{\rm v}(L^2_{\rm h})}\lesssim2^{-k'} \|v(t)\|_{L^\infty_{\rm v}( H^1_{\rm h})} \lesssim 2^{-k'} \phi^{\frac 1 2}(t) \,  .
\eeq
Now let us observe that using  again the Bernstein inequality, we have
\beno
 \|\D_{j'}^{\rm v}S_{k'+1}^{\rm h} v_\lam(t)\|_{L^2_{\rm v}(L^\infty_{\rm h})} & \lesssim & \sum_{k''\leq k'} \|\D_{j'}^{\rm v}\D_{k''}^{\rm h} v_\lam(t)\|_{L^2_{\rm v}(L^\infty_{\rm h})}\\
 &\lesssim &\sum_{k''\leq k'} 2^{k''}\|\D_{j'}^{\rm v}\D_{k''}^{\rm h} v_\lam(t)\|_{L^2} \, .
\eeno
By definition of the~$\wt {L}^4(\R^+; \cB^{s+\frac12, s'})$ norm, we  have
$$
2^{j's'} 2^{k'(s-\frac {1}2)} \, \|\D_{j'}^{\rm v}S_{k'+1}^{\rm h} v_\lam(t)\|_{L^2_{\rm v}(L^\infty_{\rm h})} \lesssim
\|v_\lam\|_{\wt {L}^4(\R^+; \cB^{s+\frac12, s'})}
\sum_{k''\leq k'} 2^{(k'-k'')(s-\frac {1}2) }\underline f_{j',k''}(t)
$$
where~$\underline  f_{j',k''}(t)$,  on the sphere of~$\ell^1(\ZZ^2 ;L^4(\R^+))$,  is defined by
$$
\underline  f_{j',k''}(t) \eqdefa \|v_\lam\|_{\wt {L}^4(\R^+; \cB^{s+\frac12, s'})}
^{-1} 2^{j's'} 2^{k''(s +\frac12)} \|\D_{j'}^{\rm v}\D_{k''}^{\rm h} v_\lam(t)\|_{L^2} \, .
$$
%Defining
%\[
%\begin{split}
%&\wt d_{j',k'} := C\sum_{k''\leq k-1}  2^{-2\frac { k'-k} p} \underline d_{j',k''}\andf\\
%&\wt f_{j',k'}(t)  := \frac 1 {C d_{j',k''} }
%C\sum_{k''\leq k-1}  2^{2\frac { k''-k} p}  \underline d_{j',k'}  \underline f_{j',k'}(t) ,
%\end{split}
%\]
%we get that
Since $s< \frac12$,    this ensures by Young's inequality that 
$$
 \|\D_{j'}^{\rm v}S_{k'+1}^{\rm h} v_\lam(t)\|_{L^2_{\rm v}(L^\infty_{\rm h})} \lesssim  2^{-j's'} 2^{-k'(s-\frac {1}2)} \,
\|v_\lam\|_{\wt {L}^4(\R^+; \cB^{s+\frac12, s'})} \wt f_{j',k'}(t)
$$
where~$\wt f_{j',k'}(t)$ is on the sphere of~$\ell^1(\ZZ^2;L^4(\R^+))$.  Together with Inequality\refeq{regulNS2Ddemoeq24}, this gives
$$
2^{js'} 2^{k(s+\frac {1}2)} \, \|\D_j^{\rm v}\D_k^{\rm h}T^2(v(t) ,v_\lam(t))\|_{L^2} \lesssim \phi(t)^{\frac 1 2}
\|v_\lam\|_{\wt {L}^4(\R^+; \cB^{s+\frac12, s'})} f_{j,k}(t) \, ,
$$
where~$f_{j,k}(t)$ is on the sphere of~$\ell^1(\ZZ^2;L^4(\R^+))$.  We deduce  that
\beq
\label{regulNS2Ddemoeq23pm1}
\begin{aligned}
\cT_{j,k,\lam}^2(t) & \eqdefa 2^{k}
2^{js'} 2^{ks} \, \int_0 ^t \exp \Bigl(-c(t-t')2^{2k} -\lam \int_{t'}^t\phi(t'')dt'' \Bigr) \\
& \qquad \qquad \qquad\qquad \qquad \qquad\qquad   \times{}
\|\D_j^{\rm v}\D_k^{\rm h}T^2(v(t'), v_\lam(t'))\|_{L^2} \,  dt'\\
& \lesssim   \|v_\lam\|_{\wt {L}^4(\R^+; \cB^{s+\frac12, s'})} \\
& \qquad\qquad  \times{}
2^{\frac {k} 2}\int_0 ^t   \exp \Bigl(-c(t-t')2^{2k} -\lam \int_{t'}^t\phi(t'')dt'' \Bigr)
\phi(t')^{\frac 1 2}  f_{j,k}(t')dt' \,  .
\end{aligned}
\eeq
Using H\"older's inequality twice, we get
\ben
\nonumber
\cT_{j,k,\lam}^2(t) & \lesssim &  \|v_\lam\|_{\wt {L}^4(\R^+; \cB^{s+\frac12, s'})} \biggl(\int_0 ^t   e^{-c(t-t')2^{2k} } f^4_{j,k}(t')dt'\biggr)^{\frac 1 4}\\
\nonumber& &\qquad\qquad\qquad{}\times 2^{\frac k2} \biggl(\int_0 ^t   \exp \Bigl(-c(t-t')2^{2k} -\lam \int_{t'}^t\phi(t'')dt'' \Bigr)
\phi(t')^{\frac 2 3}dt'  \biggr)^{\frac 3 4}\\
\label{regulNS2Ddemoeq23bis}
&\lesssim& \frac 1 {\lam^{\frac 1 2}}\|v_\lam\|_{\wt {L}^4(\R^+; \cB^{s+\frac12, s'})} \biggl(\int_0 ^t   e^{-c(t-t')2^{2k} } f^4_{j,k}(t')dt'\biggr)^{\frac 1 4} \,  .
\een
As~$T^3$ is estimated like~$T^1$ and~$T^4$ is estimated like~$T^2$, this implies finally that 
$$
\longformule{
2^{js'} 2^{ks} \|\D_j^{\rm v}\D_k^{\rm h}v_\lam(t)\|_{L^2} \lesssim
2^{js'} 2^{ks} e^{-c2^{2k}t} \|\D_j^{\rm v}\D_k^{\rm h}v_0\|_{L^2}
}
{{}+
\Bigl(\int_0^t e^{-c(t-t')2^{2k}} f_{j,k}^4(t')dt'\Bigr)^{\frac14}\Bigl(\frac 1 {\lam^{\frac14}}+\frac 1{\lam^{\frac12}}\Bigr)\|v_\lam\|_{\wt {L}^4(\R^+;\cB^{s+\frac12, s'})} \,  .
}
$$
As we have
$$
\begin{aligned}
\left( \int_0^\infty\Bigl(\int_0^t e^{-c(t-t')2^{2k}} f^4_{j,k}(t')dt'\Bigr)^{\frac14\times 4 } dt\right)^\frac14&= c^{-1}d_{j,k}2^{-\frac k2} \\
\mbox{and} \,
\sup_{t\in \R^+} \Bigl(\int_0^t e^{-c(t-t')2^{2k}} f^4_{j,k}(t')dt'\Bigr)^\frac14 &=d_{j,k} \,  , \quad \mbox{with} \,  d_{j,k}  \in \ell^1( \ZZ^2)  \,  ,
\end{aligned}
$$
we infer that
$$
\begin{aligned}
2^{js'} 2^{ks} \bigl( \|\D_j^{\rm v}\D_k^{\rm h}v_\lam\|_{L^\infty(\R^+;L^2)}
&+2^{\frac k 2} \|\D_j^{\rm v}\D_k^{\rm h}v_\lam\|_{L^4(\R^+;L^2)}\bigr)\\
&\quad  \lesssim
2^{js'} 2^{ks} \|\D_j^{\rm v}\D_k^{\rm h}v_0\|_{L^2}
+ d_{j,k}  \Bigl(\frac 1 {\lam^{\frac14}}+\frac 1{\lam^{\frac12}}\Bigr)\|v_\lam\|_{\wt {L}^4(\R^+;\cB^{s+\frac12, s'})} \,   .
\end{aligned}
$$
Taking the sum over~$j$ and~$k$ and choosing~$\lam$ large enough, we have proved\refeq{regulNS2Ddemoeq20}.

%%%%%%%%%%
%%%%%%%%%%
\medbreak
\noindent
Let us gain~$L^2$-integrability in~$t$. Using\refeq{regulNS2Ddemoeq23-1} and\refeq{regulNS2Ddemoeq23pm1} with~$\lam=0$, we find that
$$
\begin{aligned}
2^{js'} 2^{k(s+1)} \|\D_j^{\rm v}\D_k^{\rm h}v(t)\|_{L^2} &\lesssim
2^{js'} 2^{k(s+1)} e^{-c2^{2k}t} \|\D_j^{\rm v}\D_k^{\rm h}v_0\|_{L^2}\\
& \quad +  2^{2 k} \, \|v\|_{\wt {L}^4(\R^+; \cB^{s+\frac12, s'})}
\int_0^t e^{-c(t-t')2^{2k}} \bigl((  g_{j,k}(t') + 2^{-\frac k 2} h_{j,k}(t') \bigr)dt'
\,,
\end{aligned}
$$
where~$g_{j,k}$ (resp. $h_{j,k}$) are in~$\ell^1(\ZZ^2 ; L^2(\R^+))$ (resp.~$\ell^1(\ZZ^2 ; L^{\frac43}(\R^+))$), with
$$
\sum_{(j,k) \in \ZZ^2} \|g_{j,k}\|_{L^2(\R^+)} \lesssim \|\phi\|_{L^1}^\frac14 \quad \mbox{and} \quad  \sum_{(j,k) \in \ZZ^2} \|h_{j,k}\|_{L^\frac43(\R^+)} \lesssim \|\phi\|_{L^1}^\frac12 \, .
$$
Laws of convolution in the time variable, summation over~$j$ and~$k$ and~(\ref{regulNS2Ddemoeq20}) imply that
$$
\| v\|_{\wt L^2(\R^+;  \cB^{s+1,s'})}\lesssim
\|v_0\|_{ \cB^{s,s'}}  \exp \Bigl(C\int_0^\infty\phi(t)dt\Bigr)\, .
$$
This implies by interpolation in view of    \eqref{regulNS2Ddemoeq20}  that  for all $r$ in $[2,\infty]$,  all $s$ in $]-\frac 12,\frac 12[$ and all positive $s'$ 
\beq
\label{regulNS2Ddemoeq20L2} 
 \| v\|_{\wt L^r(\R^+; \cB^{s+\frac 2r,s'})}\lesssim
\|v_0\|_{ \cB^{s,s'}}  \exp \Bigl(C\int_0^\infty\phi(t)dt\Bigr)\, ,
\eeq
which in view of\refeq{inclusionXrpeq} ensures Inequality \eqref{eqdim21} and achieves the proof of  Estimate \eqref{eqdim22} in the case when $s$ belongs to $]-\frac 12,\frac 12[$. 

%%%%%%%%%%
\medbreak
\noindent Now we are going to double the interval, namely prove that for any~$s$ in~$ ]-1,1[$,  \textcolor{black}{any~$s'\geq 1/2$} and any~$r$ in~$[2,\infty]$ we have
\beq
\label{eqdim21demoeq11}
\|v\|_{\wt L^r(\R^+;\cB^{s+\frac2 r, s'})} \lesssim \|v_0\|_{\cB^{s,s'}}
+\|v_0\|_{\cB^{\frac s 2 ,s'}}\|v_0\|_{\cB^{\frac s 2  }}
\exp (C\|v_0\|_{\cB^0} E_0)\, .
\eeq
Proposition\refer{productlawsaniso} implies that for any~$s$ in~$ ]-1,1[$   \textcolor{black}{and any~$s'\geq 1/2$},  we have
$$
\|v(t)\otimes v(t)\|_{\cB^{s,s'}} \lesssim \|v(t)\|_{\cB^{\frac {s+1} 2}} 
\|v(t)\|_{\cB^{\frac  {s+1}  2, s'}} \, .
$$
The smoothing effect of the horizontal heat flow described in Lemma\refer{anisoheat} implies therefore  that, for any ~$s$ belonging to~$ ]-1,1[$, \textcolor{black}{any~$s'\geq 1/2$} and any~$r$ in~$[2,\infty]$,
\beno
\|v\|_{\wt L^r( \R^+;\cB^{s+\frac 2 r, s'})} & \lesssim &  \|v_0\|_{\cB^{s,s'}} +\|v\otimes v\|_{\wt L^2(\R^+;\cB^{s,s'})} \\
& \lesssim &  \|v_0\|_{\cB^{s,s'}} +\|v\|_{\wt L^4(\R^+;\cB^{\frac  {s+1}  2})}
\|v\|_{\wt L^4(\R^+;\cB^{\frac  {s+1}  2,s'})} \, .
\eeno
Finally Inequality\refeq{regulNS2Ddemoeq20} ensures that  for any~$s$ in~$ ]-1,1[$,  \textcolor{black}{any~$s'\geq 1/2$}  and any~$r$ in~$[2,\infty]$,
\beq
\label{eqdim21demoeq111}
\|v\|_{\wt L^r( \R^+;\cB^{s+\frac 2 r, s'})} \lesssim \|v_0\|_{\cB^{s,s'}} +\|v_0\|_{\cB^{\frac s 2}}
\|v_0\|_{\cB^{\frac s 2,s'}} \exp (C\|v_0\|_{\cB^0}E(0)) \, .
\eeq
This concludes the proof of Inequality\refeq{eqdim21demoeq11}.

\medbreak
\noindent Now let us conclude the proof of Estimate\refeq{eqdim22}. Again Proposition\refer{productlawsaniso}  implies that, for any~$s$ in~$ ]-2,0]$  \textcolor{black}{and any~$s'\geq 1/2$}, we have
$$
\|v(t)\otimes v(t)\|_{\cB^{s+1,s'}} \lesssim \|v(t)\|_{\cB^{\frac s 2+1} }
\|v(t)\|_{\cB^{\frac s 2+1, s'}}  \, .
$$
This gives rise to 
$$
\|v\otimes v\|_{L^1(\R^+;\cB^{s+1,s'})} \lesssim
 \|v\|_{L^2(\R^+;\cB^{\frac s 2+1})} 
\|v\|_{L^2(\R^+;\cB^{\frac s 2+1, s'})} \, .
$$
The smoothing effect of the heat flow gives, for  any~$r$ in~$[1,\infty]$ and any~$s$ in~$]-2,0]$, 
$$
\|v\|_{\wt L^r(\R^+;\cB^{s+\frac2 r, s'})} \lesssim  \|v_0\|_{\cB^{s,s'}} +
 \|v\|_{L^2(\R^+;\cB^{\frac s 2+1})} 
\|v\|_{L^2(\R^+;\cB^{\frac s 2+1, s'})} \,  .
$$
%Inequality\refeq{eqdim21demoeq11} implies that, for  any~$r$ in~$[2,\infty]$ and any~$s$ in~$]-2,0]$ \textcolor{red}{~$s'\geq 1/2$}  , 
%$$
%\|v\|_{\wt  L^r(\R^+;\cB^{s+\frac2 r, s'})} \lesssim  \|v_0\|_{\cB^{s,s'}} +
% \|v\|_{L^2(\R^+;\cB^{\frac s 2+1, \frac 12})} 
%\|v\|_{L^2(\R^+;\cB^{\frac s 2+1, s'})}  \, .
%$$
Inequality\refeq{eqdim21demoeq111} implies that, for  any~$r$ in~$[1,\infty]$ and any~$s$ in~$]-2,0]$ and~~$s'\geq 1/2$ , 
\beq
\label{eqdim21demoeq1111}
\|v\|_{\wt L^r(\R^+;\cB^{s+\frac2 r, s'})} \lesssim  \|v_0\|_{\cB^{s,s'}}
+\|v_0\|_{\cB^{\frac s 4}}^3
\|v_0\|_{\cB^{\frac s 4,s'}} \exp (C\|v_0\|_{\cB^0}E_0) \, .
\eeq
This proves the estimate\refeq{eqdim22} and thus Proposition\refer{regulNS2D}.
  \end{proof}
  
  \medbreak
  
  \subsection{Propagation of regularity by a 2D flow with parameter}\label{propagtwodimparameter}

Now let us estimate the norm of the function~$w^3$ defined as the solution of~$(T_\b)$ defined page~\pageref{defTbetapage}.  This is described in the following proposition. 
 \begin{prop}
\label{regiultransportdiff2D}
{\sl Let~$v_0$ and~$v$ be as in Proposition~{\rm\refer{regulNS2D}}.  For any non negative real number~$\b$, let us consider~$w^3$ the solution of
$$
(T_\b) \quad 
\partial_t w^3 + v\cdot \nabla^{\rm h}w^3-\Delta_{\rm h} w^3- \b^2 \partial_3^2 w^3= 0 \andf
w^3_{|t=0} = w_{0}^3\, .
$$
Then~$w^3$ satisfies the following estimates where all the constants are independent of~$\b$:
 \begin{equation}
 \label{eqdim21transportiff}
 \|  w^3 \|_{\cA^0} \lesssim \| w^3_0\|_{\cB^0}\exp \bigl(\cT_1( \| v_0 \|_{\cB^0})\bigr) \,  ,
 \end{equation}
 and for  any~$s$ in~$ [-2+\mu,0]$ and any~$s' \geq 1/2$,
we have
 \beq
\label{eqdim22transportiff}
 \ \|w^3\|_{\cA^{s,s'}} \lesssim
\bigl(\|w_0^3\|_{ \cB^{s, s'}} + \|w_0^3\|_{\cB^0} \cT_2(\|v_0\|_{S_\mu})\bigr)
\exp \bigl(\cT_1( \| v_0 \|_{\cB^0})\bigr) \, .
 \eeq
}\end{prop} \begin{proof}
 %%%%%%%%%%%
 %%%%%%%%%%%
This is a question of   propagating   anisotropic regularity by a transport-diffusion equation. This propagation is described by the following lemma, which will easily lead to Proposition~\ref{regiultransportdiff2D}.
 \begin{lem}
 \label{propagationlemmaheathorizontal}
{\sl  Let us consider~$(s,s')$ a couple of real numbers, and $\cQ$  a bilinear operator which maps continuously~$\cB^{1} \times \cB^{s+1,s'}$ into~$ \cB^{s,s'} $.   A constant~$C$ exists such that for any two-component vector field~$v$ in~$L^2(\R^+;\cB^{1})$, any~$f$ in~$L^1(\R^+;\cB^{s,s'})$, any $a_0$ in $\cB^{s,s'}$ and for any non negative~$\b$, if~$\D_\b\eqdefa \D_{\rm h} +\b^2\partial_z^2$ and~$a$ is the solution of
 $$
 \partial_t a -\D_\b a +\cQ(v,a) = f \andf a_{|t=0}=a_0 \, ,
 $$
 then $a$ satisfies
 $$
 \forall r\in [1,\infty]\,,\ \|a\|_{\wt L^r(\R^+; \cB^{s+\frac2 r,s'})} \leq C
 \bigl(\|a_0\|_{\cB^{s,s'}}+\|f\|_{L^1(\R^+;\cB^{s,s'})} \bigr)\exp\Bigl(C\int_0^\infty\|v(t)\|_{\cB^{1}}^2 dt\Bigr) \, .
 $$
}  \end{lem}
 \begin{proof}
 This is a Gronwall type estimate. However
the fact that the third index of the Besov spaces   is one, induces some technical difficulties which  lead us to work  first on subintervals~$I$ of~$\R^+$  on which~$\|v\|_{L^2(I;\cB^{1})}$ is small.

\medskip
\noindent Let us first consider any subinterval~$I=[\tau_0,\tau_1]$ of~$\R^+$. The Duhamel formula and the smoothing effect of the heat flow described in Lemma\refer{anisoheat} imply that 
$$
\longformule{
 \|\D_k^{\rm h}\D_j^{\rm v} a(t)\|_{L^2}  \leq e^{-c2^{2k}(t-\tau_0)} \|\D_k^{\rm h}\D_j^{\rm v} a(\tau_0)\|_{L^2} 
 }
 { {}
+C \int_{\tau_0}^t e^{-c2^{2k}(t-t')} \bigl\|\D_k^{\rm h}\D_j^{\rm v} \bigl(\cQ(v(t'),a(t'))+f(t')\bigr)\bigr\|_{L^2} dt'   \, .
}
$$
After multiplication by~$2^{ks+js'}$ and using Young's inequality in the time integral, we deduce that 
$$
\longformule{
2^{ks+js'}  \bigl(\|\D_k^{\rm h}\D_j^{\rm v} a\|_{L^\infty(I;L^2)}  
+ 2^{2k}\|\D_k^{\rm h}\D_j^{\rm v} a\|_{L^1(I;L^2)}  \bigr) \leq C 2^{ks+js'}  \|\D_k^{\rm h}\D_j^{\rm v} a(\tau_0)\|_{L^2} 
 }
 { {}
+C \int_{\tau_0}^t d_{k,j}(t') \bigl(\|v(t')\|_{\cB^1} \|a(t')\|_{\cB^{s+1,s'}}+\|f(t')\|_{\cB^{s,s'}}\bigr) dt'
}
$$
where for any~$t$,~$d_{k,j}(t)$ is an element of the sphere of~$\ell^1(\ZZ^2)$. By summation over~$(k,j)$ and using the Cauchy-Schwarz inequality, we infer that 
\beq
\label{propagationlemmaheathorizontaldeoeq1}
\begin{split}
\|a\|_{\wt L^\infty(I;\cB^{s,s'})} +\|a\|_{L^1(I;\cB^{s+2,s'})}  &
\leq
 C \|a(\tau_0)\|_{\cB^{s,s'}}  +C \|f \|_{L^1(I;\cB^{s,s'})}\\
&\qquad\qquad\qquad{}+
C\|v\|_{L^2(I;\cB^1)} \|a\|_{L^2(I;\cB^{s+1,s'})} \, .
\end{split}
\eeq
Let us define the   increasing sequence~$(T_m)_{0\leq m\leq M+1}$ by induction  such that~$T_0=0$,~$T_{ M+1}=\infty$ and
$$
\forall m <M\,,\ \int_{T_m}^{T_{m+1}} \|v(t)\|_{\cB^1}^2dt = c_0 \andf
\int_{T_{M}}^{\infty} \|v(t)\|_{\cB^1}^2dt \leq c_0 \, ,
$$
for some given~$c_0$ which will be chosen later on.  Obviously, we have
\beq
\label{propagationlemmaheathorizontaldeoeq2}
\int_0^\infty   \|v(t)\|^2_{\cB^1} dt  \geq
\int_0^{T_{M}}  \|v(t)\|^2_{\cB^1} dt =Mc_0 \, .
\eeq
 Thus the number~$M$ of~$T_m's$ such that~$T_{m}$ is finite is less than~$c_0^{-1}\|v\|_{L^2(\R^+;\cB^{1})}^2$.  
 Applying Estimate \refeq{propagationlemmaheathorizontaldeoeq1} to the interval~$[T_{m},T_{m+1}]$, we get 
 $$
 \begin{aligned}
  \|a\|_{L^\infty([T_{m},T_{m+1}];\cB^{s,s'})} +\|a\|_{L^1([T_{m},T_{m+1}];\cB^{s+2,s'})} & \leq 
 \|a\|_{L^2([T_{m},T_{m+1}];\cB^{s+1,s'})}
\\ & \quad  +C \bigl( \|a(T_{m})\|_{\cB^{s,s'}}  +C \|f\|_{L^1([T_{m},T_{m+1}];\cB^{s,s'})}\bigr)
 \end{aligned} $$
 if~$c_0$ is chosen such that ~$C \sqrt{c_0} \leq 1$. As 
 $$
\|a\|_{L^2([T_{m},T_{m+1}];\cB^{s+1,s'})} \leq \|a\|_{L^\infty([T_{m},T_{m+1}];\cB^{s,s'})} ^{\frac12}\|a\|_{L^1([T_{m},T_{m+1}];\cB^{s+2,s'})}^{\frac12}\, ,
 $$
 we infer that 
\beq
\label{propagationlemmaheathorizontaldeoeq3}
 \begin{aligned}\|a\|_{L^\infty([T_{m},T_{m+1}];\cB^{s,s'})}& +\|a\|_{L^1([T_{m},T_{m+1}];\cB^{s+2,s'})} \\
 &\quad \leq   2C\bigl(\|a(T_{m})\|_{\cB^{s,s'}}  + \|f\|_{L^1([T_{m},T_{m+1}];\cB^{s,s'})}\bigr) \, .
\end{aligned}
\eeq
Now let us us prove by induction that
$$
\ \|a\|_{L^\infty([0,T_{m}];\cB^{s,s'})}  \leq
(2C)^{m}  \bigl(\|a_0\|_{\cB^{s,s'}}+\|f\|_{L^1([0,T_{m}],\cB^{s,s'})} \bigr).
$$
Using~(\ref{propagationlemmaheathorizontaldeoeq3}) and the induction hypothesis we get
\beno
\|a\|_{L^\infty([ {T_{m}} ,T_{m+1}];\cB^{s,s'})}  
& \leq &  2C\bigl( \|a\|_{L^\infty([0,T_{m}];\cB^{s,s'})}  + \|f\|_{L^1([T_{m},T_{m+1}];\cB^{s,s'})}\bigr)\\
& \leq &  (2C)^{m+1} \bigl(\|a_0\|_{\cB^{s,s'}}+\|f\|_{L^1([0,T_{m+1}],\cB^{s,s'})} \bigr ) \, ,
\eeno
provided that $2C \geq 1$. This proves in view of \eqref{propagationlemmaheathorizontaldeoeq2} that
$$
\|a\|_{L^\infty(\R^+;\cB^{s,s'})} \leq C  \bigl(\|a_0\|_{\cB^{s,s'}}+\|f\|_{L^1(\R^+;\cB^{s,s'})} \bigr)\exp\Bigl(C\int_0^\infty\|v(t)\|_{\cB^{1}}^2 dt\Bigr) \, .
$$
We deduce from\refeq{propagationlemmaheathorizontaldeoeq3} that 
$$
\longformule{
\|a\|_{L^1([T_{m},T_{m+1}];\cB^{s+2,s'})} \leq   C\bigl(\|a_0\|_{\cB^{s,s'}}+\|f\|_{L^1(\R^+;\cB^{s,s'})} \bigr)\exp\Bigl(C\int_0^\infty\|v(t)\|_{\cB^{1}}^2 dt\Bigr) 
}
{
{} + C\|f\|_{L^1([T_{m},T_{m+1}];\cB^{s,s'})}\, .}
$$
Once noticed that~$xe^{Cx^2}\leq e^{C'x^2}$, the result comes by summation over~$ m$ and the fact that the total number of~$m$'s is less than or equal to~$c_0^{-1} \|v\|^2_{L^2(\R^+;\cB^{1})}$. The lemma is proved.
 \end{proof}

\medbreak

\noindent {\it Conclusion of the proof of Proposition~{\rm\refer{regiultransportdiff2D}}.}
  We  apply Lemma\refer{propagationlemmaheathorizontal} with $ {\color{black}{\cQ}}(v,a) =\dive_{\rm h} (av)$, $ {f=0}$, $a=w^3$, and~$(s,s')=(0,1/2)$.
  Indeed since~$\cB^1$ is an algebra we have
  $$
  \| {\cQ}(v,a)\|_{\cB^0}  \lesssim \|av\|_{\cB^1} \lesssim \|a\|_{\cB^1}  \|v\|_{\cB^1}.
  $$
So  Lemma\refer{propagationlemmaheathorizontal}  gives
  $$
 \|  w^3 \|_{\cA^0} \lesssim \| w^3_0\|_{\cB^0}\exp \Bigl( C \int_0^\infty
 \|v(t)\|_{\cB^1}^2 dt\Bigr) \, .
$$
Thanks to Estimate\refeq{eqdim21} of Proposition\refer{regulNS2D}
we deduce\refeq{eqdim21transportiff}. 

\medbreak
\noindent Now for~$s$ belonging to~$[-2+\mu,0]$, we apply Lemma\refer{propagationlemmaheathorizontal} with~$a=w^3$,~${\color{black}{\cQ}}(v,a) =  \dive_{\rm h} (T^{\rm v} _v a)$, and~$f= \dive_{\rm h} (\wt T^{\rm v} _a v)$, where with the notations of Definition \ref{deflpanisointro} 
\beq 
\label{bonyvertical}
 T^{\rm v} _v a 
  \eqdefa \sum_jS^{\rm v}_{j-1} v \D^{\rm v}_ja\,,\   R^{\rm v}(a,v) \eqdefa
\!\!\!  \sumetage{j}{ -1\leq \ell\leq 1} \D^{\rm v}_{j-\ell} a\D^{\rm v}_j v\andf 
\wt T^{\rm v} _a v 
 \eqdefa T^{\rm v}_a v+R^{\rm v} (a,v)\,. 
\eeq
  Lemma\refer{productlawsanisodemolem} implies that for  any~$s$ in~$ [-2+\mu,0]$ and any~$s' \geq 1/2$,
$$
\| T^{\rm v} _v w^3 \|_{\cB^{s+1,s'}} \lesssim \| v \|_{\cB^1} \|  w^3 \|_{\cB^{s+1,s'}}\,. 
$$ 
We infer from Lemma\refer{propagationlemmaheathorizontal}  that, for any~$r$ in~$[1,\infty]$, 
\beq
\label{useestwithr1}
\|w^3\|_{\wt L^r(\R^+; \cB^{s+\frac2 r, s'})} \lesssim  \bigl(\|w^3_0\|_{\cB^{s,s'}}
+\|\dive_{\rm h} (\wt T^{\rm v} _a v)\|_{L^1(\R^+;\cB^{s,s'})} \bigr)\exp \bigl(\cT_1( \| v_0 \|_{\cB^0})\bigr) \, .
\eeq
But we have, using laws of anisotropic  paraproduct given in Lemma~\ref{productlawsanisodemolem}, 
$$
\begin{aligned}
\|\dive_{\rm h} (\wt T^{\rm v} _{w^3} v)\|_{L^1(\R^+;\cB^{s,s'})} & \lesssim \|\wt T^{\rm v} _{w^3}  v\|_{L^1(\R^+;\cB^{s+1,s'})} \\
&\lesssim \|w^3\|_{L^2(\R^+;\cB^{1})}
\|v\|_{L^2(\R^+;\cB^{s+1,s'})} \, .
\end{aligned}
$$
Applying\refeq{eqdim21transportiff} and\refeq{eqdim22} gives\refeq{eqdim22transportiff}. Proposition\refer{regiultransportdiff2D} is   proved.
\end{proof}

\noindent As~$w^{\rm h}$ is defined by~$w^{\rm h} =- \nabla_{\rm h}\D_{\rm h}^{-1} \partial_3 w^3$, we deduce from Proposition\refer{regiultransportdiff2D}, Lemma\refer{Bernsteinaniso}  and
  the scaling property \eqref{scale},  the following corollary.
\begin{cor}
\label{regiultransportdiff2Dwh}
{\sl For any~$s$ in~$ [-2+\mu,0]$ and any~$s' \geq 1/2$,
$$
\|w^{\rm h} \|_{\cA^{s+1,s'-1}} \lesssim 
\bigl(\|w_0^3\|_{ \cB^{s, s'}} + \|w_0^3\|_{\cB^0} \cT_2(\|v_0\|_{S_\mu})\bigr)
\exp \bigl(\cT_1( \| v_0 \|_{\cB^0})\bigr) \, .
$$
}
\end{cor}

{}
\medbreak

\subsection{Conclusion of the proof of Theorem\refer{slowvarsimple}}\label{conclusionproofpropagprofile}
Using the definition of the approximate solution~$\Phi^{\rm app}$ given in \eqref{defuapppropagaprofile}, we infer  from Propositions\refer{regulNS2D}
 and \refer{regiultransportdiff2D} and
 Corollary~\ref{regiultransportdiff2Dwh}
 that 
 \beq
\label{slowvarsimpleconcludemoeq1}
\|\Phi^{\rm app} \|_{L^2(\R^+;\cB^1)}   \leq \cT_1(\|(v_0, w^3_0)\|_{\cB^0}) +\b \cT_2(\|(v_0, w^3_0)\|_{S_\mu}) \, .
\eeq
%{\color{black}and  for all~$(s,s')$ in~$ \wt D_\mu$, 
%\beq
%\label{slowvarsimpleconcludemoeq11}
%\|\Phi^{\rm app} \|_{\cA^{s,s'}}   \leq \b^{s'-\frac12} \cT_2(\|(v_0, w^3_0)\|_{S_\mu}) \, .
%\eeq
 Moreover, the error term~$\psi$ satisfies the following modified Navier-Stokes equation, with zero initial data:
\beq
%\label{definRemainder}
\begin{split}
&\partial_t \psi +\dive \bigl( \psi\otimes \psi+\Phi^{\rm app}\otimes\psi+\psi\otimes \Phi^{\rm app}\bigr) -\D\psi=  - \nabla q_\b +\sum_{\ell=1} ^4 E_\b^\ell \with\\
&\qquad\quad\qquad E^1_\b \eqdefa \partial_3^2[(v, 0)]_\b+\b(0,[\partial_3p]_\b)\,,\\
&\qquad\quad\qquad E^2_\b \eqdefa \b  {\color{black}\Bigl[}  \Bigl (w^3\partial_3(v,w^3)+\bigl( \nabla^{\rm h} \Delta_{\rm h}^{-1} \mbox{div}_{\rm h} \partial_3 (v  w^3),0\bigr)\Bigr) {\color{black}\Bigr]_\b}\,,\\
&\qquad\quad\qquad E^3_\b  \eqdefa \b {\color{black}\Bigl[}\Bigl(w^{\rm h} \cdot\nabla_h(v,w^3) +v \cdot \nabla^{\rm h} (w^{\rm h},0)\Bigr){\color{black}\Bigr]_\b} \andf\\
&\qquad\quad\qquad E^4_\b \eqdefa \b^{\color{black}2} {\color{black}\Bigl[}\Bigl(w^{\rm h}  \cdot \nabla^{\rm h} (w^{\rm h} ,0 )+  w^3 \partial_3( w^{\rm h},0)\Bigr){\color{black}\Bigr]_\b} \, .
\end{split}
\eeq
If we prove that 
\beq
\label{slowvarsimpleconcludemoeq2}
\Bigl\|\sum_{\ell=1} ^4 E_\b^\ell\Bigr\|_{\cF^0} \leq\b \cT_2\bigl(\|(v_0,w_0^3)\|_{S_\mu} \bigr) \, , 
\eeq
then according to the fact $ \psi_{|t=0} = 0$, Proposition\refer{existencepetitcB1} implies that~$\psi$ exists globally and satisfies
\begin{equation}\label{theestimateontheremainder}
\|\psi\|_{  L^2(\R^+;\cB^1)} \lesssim \b  \,\cT_2\bigl(\|(v_0,w_0^3)\|_{S_\mu} \bigr)  \, .
\end{equation}
This in turn implies that~$\Phi_0$  generates a global regular solution~$\Phi_\b$ in~$ L^2(\R^+;\cB^1)$ which satisfies
\begin{equation}\label{PhiisboundedL2B1}
\|\Phi_\b\|_{L^2(\R^+;\cB^1)} \leq  \cT_1\bigl(\|(v_0,w_0^3)\|_{\cB^0} \bigr)+ \b  \,\cT_2\bigl(\|(v_0,w_0^3)\|_{S_\mu} \bigr) \, . 
\end{equation}
Once this bound in~$   L^2(\R^+;\cB^1)$ is obtained, the bound in~$\cA^{0}$ follows  by heat flow estimates, and in~$\cA^{s,s'}$ by propagation of regularity for the Navier-Stokes equations as stated in Proposition~\ref{propaganisoNS3D} below.

\medskip
\noindent So all we need to do is  to prove Inequality\refeq{slowvarsimpleconcludemoeq2}.  
 Let us first estimate the term~$\partial_3^2[(v , 0)]_\b$. This requires the use of some~$\wt L^2(\R^+;\cB^{s,s'})$ norms. We get
$$
\|\partial_3^2[v]_\b \|_{\wt L^2(\R^+; \cB^{0,-\frac12})} \lesssim
\|[v]_\b \|_{\wt L^2(\R^+; \cB^{0,\frac32})}.
$$
Using the vertical scaling property \eqref{scale} of the space~$\cB^{0,\frac32}$, this gives
$$
\|\partial_3^2[v]_\b \|_{\wt L^2(\R^+; \cB^{0,-\frac12})}\lesssim \b  \,
\|v  \|_{\wt L^2(\R^+; \cB^{0,\frac32})}\, .
$$Using Proposition\refer{regulNS2D}, we get 
\begin{equation}
\label{partial23}
\|\partial_3^2[v]_\b \|_{\wt L^2(\R^+; \cB^{0,-\frac12})}  \leq \b  \,\cT_2(\|v_0\|_{S_\mu})\,.
 \end{equation}
Now let us  study   the pressure term. By applying the horizontal divergence to the equation satisfied by~$v$ we get, thanks to the fact that ~$\mbox{div} _{\rm h} v = 0$,
$$
\partial_{3} p = -\partial_{3} \Delta_{\rm h}^{-1} \sum_{\ell,m=1}^2 \partial_\ell\partial_m(v^\ell v^m) \, .
 $$
Using  the fact that~$\Delta_{\rm h}^{-1} \partial_\ell \partial_m $ is a zero-order horizontal Fourier multiplier (since~$\ell$ and~$m$ belong to~$\{1,2\}$), we infer that
$$
\begin{aligned}
  \big \|[\partial_{3} p]_\b \big\|_{L^1(\R^+;\cB^0)} & =   \|\partial_{3} p \|_{L^1(\R^+;\cB^0)} \\
& \lesssim   \| v \partial_3 v  \|_{L^1(\R^+;\cB^0)}\, .
\end{aligned}
$$
Laws of product  in anisotropic Besov as described by Proposition\refer{productlawsaniso} imply that
 $$
 \|v (t)\partial_3 v(t)\|_{\cB^0} \lesssim  \|v(t)\|_{\cB^1}
\|\partial_3 v(t)\|_{\cB^0}\, , $$ 
which gives rise to 
\ben
  \big \|[\partial_{3} p]_\b \big\|_{L^1(\R^+;\cB^0)} & \lesssim &
 \| v \|_{{L^2}(\R^+;\cB^1)}  \| \partial_3 v \|_{{L^2}(\R^+;\cB^0)}
\nonumber\\
\label{pressure}
&\lesssim &
 \| v \|_{L^2(\R^+;\cB^1)}  \| v \|_{L^2(\R^+;\cB^{0,\frac32})} \, .
\een
Combining \eqref{partial23} and \eqref{pressure}, we get by virtue of Proposition \ref{regulNS2D} and Lemma \ref{example}
\beq
\label{slowvarsimpleconcludemoeq211}
\|E_\b^1\|_{\cF^0} \leq\b \,\cT_2\bigl(\|v_0\|_{S_\mu} \bigr) \,  .
\eeq
Now we estimate~$E^2_\b$.  Applying again the laws of product  in anisotropic Besov spaces   (see Proposition\refer{productlawsaniso}) together  with the action of vertical derivatives, we obtain 
\beno
\|w^3(t)\partial_3(v ,w^3)(t)\|_{\cB^0} &\lesssim & \|w^3(t)\|_{\cB^1}
\|\partial_3 (v , w^3)(t)\|_{\cB^0}\\
&\lesssim & \|w^3(t)\|_{\cB^1}
\|(v , w^3)(t)\|_{\cB^{0,\frac32}} \,.
\eeno
Thus we infer that
\begin{equation}
 \label{estsec} 
 \|w^3\partial_3(v ,w^3)\|_{L^1(\R^+;\cB^0)} \lesssim
 \|w^3\|_{L^2(\R^+;\cB^1)}
\|(v , w^3)\|_{L^2(\R^+;\cB^{0,\frac32})} \,.
\end{equation}
For the other term of~$E^2_\b$, using the fact that $\nabla^{\rm h}\D_{\rm h}^{-1} \dive_{\rm h}$ is an order $0$ horizontal Fourier  multiplier and the  Leibniz formula, we infer from Lemma\refer{Bernsteinaniso} that
\beno
\|\nabla^{\rm h}\D_{\rm h}^{-1} \dive_{\rm h}\partial_3(v  w^3)(t)\|_{\cB^0} &
\lesssim & \|\partial_3(v  w^3)(t)\|_{\cB^0}\\
& \lesssim &  \|v (t)\partial_3 w^3(t)\|_{\cB^0}+ \| w^3(t)\partial_3v (t)\|_{\cB^0} \,.
\eeno
 In view of laws of product in anisotropic Besov spaces and the action of vertical derivatives,  this gives rise to 
$$
\|\nabla^{\rm h}\D_{\rm h}^{-1} \dive_{\rm h}\partial_3(v  w^3)(t)\|_{\cB^0} \lesssim
 \|v (t)\|_{\cB^{1}}\| w^3(t)\|_{\cB^{0,\frac32}}+ \| w^3(t)\|_{ \cB^{1}} \|v (t)\|_{\cB^{0,\frac32}} \,.
$$
Together with\refeq{estsec}, this leads to 
$$\begin{aligned}
\|E^2_\b\|_{L^1(\R^+;\cB^0)} &\lesssim \beta \, \|w^3\|_{L^2(\R^+;\cB^1)}
\|(v , w^3)\|_{L^2(\R^+;\cB^{0,\frac32})}\\
&\qquad\qquad\qquad\qquad\qquad {}+\beta \,\|w^3\|_{L^2(\R^+;\cB^{0,\frac32})}
\|v \|_{L^2(\R^+;\cB^1)} \,,
\end{aligned}
$$
hence by Propositions \ref{regulNS2D} and \ref{regiultransportdiff2D} along with Lemma \ref{example}
\beq
\label{slowvarsimpleconcludemoeq222}
\|E_\b^2\|_{\cF^0} \leq\b \, \cT_2\bigl(\|(v_0,w_0^3)\|_{S_\mu} \bigr) \,  .
\eeq
Let us estimate~$E^3_\b$.   Again by laws of product and the  action of horizontal derivatives, we obtain
\beno
\|w^{\rm h}  \cdot\nabla_h(v ,w^3)\|_{L^1(\R^+;\cB^0)} &\lesssim &
\|w^{\rm h} \|_{L^2(\R^+; \cB^1)}\|\nabla^{\rm h}(v ,w^3)\|_{L^2(\R^+;\cB^0)}\\
&\lesssim &
\|w^{\rm h} \|_{L^2(\R^+; \cB^1)}\|(v ,w^3)\|_{L^2(\R^+;\cB^1)} \,.
\eeno
Corollary\refer{regiultransportdiff2Dwh} and Propositions\refer{regulNS2D} and\refer{regiultransportdiff2D} imply that 
\beq
\label{estimatesvepshdemoeq3}
\|w^{\rm h} \cdot\nabla_h(v ,w^3)\|_{L^1(\R^+; \cB^0)} \leq\cT_2\bigl(\|(v_0,w_0^3)\|_{S_\mu} \bigr) \,.
\eeq
Following the same lines we get
$$
\|v  \cdot \nabla^{\rm h} (w^{\rm h} ,0)\|_{L^1(\R^+;\cB^{0})}\leq\cT_2\bigl(\|(v_0,w_0^3)\|_{S_\mu} \bigr) \,.
$$
Together with\refeq{estimatesvepshdemoeq3}, this gives thanks to Lemma\refer{example}
\beq
\label{slowvarsimpleconcludemoeq233}
 \|E_\b^3\|_{\cF^0} \lesssim \|E^3_\b \, \|_{L^1(\R^+;\cB^0  )}
 \leq\b\, \cT_2\bigl(\|(v_0,w_0^3)\|_{S_\mu} \bigr) \,  .
\eeq
Now let us estimate~$E^4_\b$. Laws of product  and the action of derivations give
\ben
\| w^{\rm h}  \cdot \nabla^{\rm h} w^{\rm h} \|_{L^1(\R^+;\cB^0 ) } 
&\lesssim &
\| w^{\rm h}  \|_{L^2(\R^+; \cB^1)} \| \nabla^{\rm h} w^{\rm h} (t)\|_{L^2(\R^+;\cB^0)}\nonumber\\
&\lesssim &
\label{estimatesvepshdemoeq5}
\| w^{\rm h} \|_{L^2(\R^+;\cB^1)}^2 \, .
\een
In the same way, we get 
$$
\|w^3  (t) \partial_3 w^{\rm h} \|_{L^1(\R^+; \cB^0)} \lesssim 
\|w^3\|_{L^2(\R^2; \cB^0)} \|w^{\rm h}\|_{L^2(\R^+;\cB^{1,\frac32})} \, .
$$
Together with\refeq{estimatesvepshdemoeq5}, this gives thanks to Corollary\refer{regiultransportdiff2Dwh} and Propositions\refer{regiultransportdiff2D}
$$
\|E^4_\b\|_{L^1(\R^+;\cB^0)} \leq\b^2\, \cT_2\bigl(\|(v_0,w_0^3)\|_{S_\mu} \bigr) \,  .
$$
Lemma\refer{example} implies that 
$$
\|E_\b^4\|_{\cF^0} \leq\b^2\, \cT_2\bigl(\|(v_0,w_0^3)\|_{S_\mu} \bigr) \,  .
$$
Together with Inequalities\refeq{slowvarsimpleconcludemoeq211},\refeq{slowvarsimpleconcludemoeq222} and\refeq{slowvarsimpleconcludemoeq233}, this gives 
$$
\|E_\b\|_{\cF^0} \leq\b \,  \cT_2\bigl(\|(v_0,w_0^3)\|_{S_\mu} \bigr) \,  .
$$
Thanks to Proposition\refer{existencepetitcB1} we obtain  that the solution~$\Phi_\b$   of (NS) with intial data
$$
 \Phi_{0}= \big[  (  v_0 -\b\nabla^{\rm h} \Delta_{\rm h}^{-1} \partial_3w^3_0, w^3_0 )  \big]_{\b}$$ is global and belongs to~$L^2(\R^+;\cB^1)$.  The whole Theorem\refer{slowvarsimple}  follows from  the next propagation result  proved in Section\refer{lpanisodivers}.
  \qed
 
 \medbreak

 \begin{prop}
 \label{propaganisoNS3D}
{\sl Let~$u$ be a solution of~{\rm(NS)} which belongs to~$L^2(\R^+;\cB^1)$ and with initial data ~$u_0$ in~$\cB^0$. Then~$u$ belongs to~$\cA^0$ and satisfies  
\beq
\label{propaganisoNS3Deq1}
\|u\|_{L^1(\R^+;\cB^{2})} + \|u\|_{L^1(\R^+;\cB^{1,\frac 32})} \lesssim \|u_0\|_{\cB^0}
+\|u\|_{L^2(\R^+;\cB^1)}^2 \, .
 \eeq
 Moreover, if the initial data~$u_0$ belongs in addition  to~$\cB^s$ for some~$s$ in~$[-1+\mu,1-\mu]$,  then
\beq
\label{propaganisoNS3Deq2}
\forall r\in [1,\infty]\,,\  \|u\|_{L^r(\R^+; \cB^{s+\frac 2r})} \leq \cT_1(\|u_0\|_{\cB^{s}}) \cT_0(\|u_0\|_{\cB^0},\|u\|_{L^2(\R^+;\cB^1)}) \, . 
\eeq
 Finally, if~$u_0$ belongs to~$\cB^{0,s'}$ for some~$s'$ greater than~$1/2$, then 
\beq
\label{propaganisoNS3Deq3}
\forall r\in [1,\infty]\,,\  \|u\|_{L^r(\R^+; B^{\frac 2 r,s'})} \leq \cT_1(\|u_0\|_{\cB^{0,s'}}) \cT_0(\|u_0\|_{\cB^0},\|u\|_{L^2(\R^+;\cB^1)}) \, .
\eeq
 }
  \end{prop}

%%%%%%%%%%%%%%%%%%%%%%%%%%%%%%%%%%%%%%%%%%

\section{Interaction between profiles of scale 1: proof of Theorem~\ref{interactionprofilescale1}} 
\label{interactionprofiles1case}
The goal of this section is to prove  Theorem~\ref{interactionprofilescale1}. In the next paragraph we define an approximate solution, using   results proved in the previous section, and Paragraph~\ref{localizationpropertiesapprox} is devoted to the proof of useful localization results on the different parts entering the definition of the approximate solution. Paragraph~\ref{conclusionproofinteractionprofilescale1} concludes the proof of the theorem, using those localization results.

\subsection{The approximate solution}\label{decompositioninitialdataprofile1}
Consider the divergence free vector field $$
\Phi^0_{0,n,\alpha,L} \eqdefa u_{0,\alpha} +  \bigl[\bigl( v_{0,n,\alpha,L}^{0,\infty  }+h_n^0 w_{0,n,\alpha,L}^{0,\infty, {\rm h} },   w_{0,n,\alpha,L}^{0,\infty, 3 }\bigr)  \bigr]_{h_n^0}+  \big[(  {v_{0,n,\alpha,L}^{0,\rm{loc}}+ h_{n }^0 w_{0,n,\alpha,L}^{0,\rm{loc},{\rm   h}}},  {w_{0,n,\alpha,L}^{0,\rm{loc},3}})\big]_{  {h_{n}^0}}  \, , 
$$
with the notation of Theorem~\ref{mainresultprofilesdi}. We want to prove that for~$h_n^0$ small enough, depending only on~$u_0$ and on~$\big\|(v_{0,n,\alpha,L}^{0,\infty  },
  w_{0,n,\alpha,L}^{0,\infty, 3 }
)\big\|_{S_\mu}$  as well as~$\big\|(v_{0,n,\alpha,L}^{0,\rm{loc}  },
  w_{0,n,\alpha,L}^{0,\rm{loc}, 3 }
) \big \|_{S_\mu}$,  there is a unique, global smooth solution to (NS) with data~$\Phi^0_{0,n,\alpha,L} $.

\medskip
\noindent
Let us start by
solving globally (NS) with the data~$u_{0,\alpha}$.  By
 using the global strong stability of (NS) in~$\cB_{1,1}$ (see~\cite{bg}, Corollary~3) and the convergence result~(\ref{uoalphaclosetouo}) we deduce that for~$\alpha$ small enough there is a unique, global solution to (NS) associated with~$u_{0,\alpha}$, which we shall denote by~$
u_\alpha$ and which lies in~$L^2(\R^+;B^{2,\frac12}_{1,1})$. Moreover by the embedding of~$B^{2,\frac12}_{1,1}$ into~$\cB^1$ we have~$u_\alpha \in L^2(\R^+;\cB^1)$.

\medskip
\noindent
Next let us define
$$
 \Phi^{0, \infty}_{0,n,\alpha,L}\eqdefa  \bigl[\bigl( v_{0,n,\alpha,L}^{0,\infty  }+h_n^0 w_{0,n,\alpha,L}^{0,\infty, {\rm h} },   w_{0,n,\alpha,L}^{0,\infty, 3 }\bigr)  \bigr]_{h_n^0}\, .
$$
Thanks to Theorem~\ref{slowvarsimple}, we know that for~$h_n^0$ smaller than~$\eps_1\big(\big\|(v_{0,n,\alpha,L}^{0,\infty  },
  w_{0,n,\alpha,L}^{0,\infty, 3 }
)\big\|_{S_\mu}\big)$ there is a unique global smooth solution~$  \Phi^{0, \infty}_{n,\alpha,L}$ associated with~$  \Phi^{0, \infty}_{0,n,\alpha,L}$, which belongs to~$\cA_0$, and using the notation  and results of Section~\ref{propagationprofiles}, in particular~(\ref{defuapppropagaprofile}) and~(\ref{theestimateontheremainder}), we can write
\begin{equation}\label{decompositioninftypart}
\begin{aligned}
\Phi^{0, \infty}_{n,\alpha,L}&  \eqdefa   \Phi^{0, \infty, {\rm app}}_{n,\alpha,L} +\psi^{0, \infty}_{n,\alpha,L} 
\quad \mbox{with} \\
  \Phi^{0, \infty, {\rm app}}_{n,\alpha,L} & \eqdefa \big[
 v^{0, \infty }_{n,\alpha,L}+h_n^0 w^{0, \infty, {\rm h}}_{n,\alpha,L}, w^{0, \infty ,3}_{n,\alpha,L}
 \big]_{h_n^0} \quad \mbox{and} \\
 \|\psi^{0, \infty}_{n,\alpha,L} \|_{L^2(\R^+;\cB^1)} &\lesssim h_n^0 \cT_2\big(\big\|(v_{0,n,\alpha,L}^{0,\infty  },
  w_{0,n,\alpha,L}^{0,\infty, 3 }
)\big\|_{S_\mu}\big)  \, ,
\end{aligned}
\end{equation}
where~$v^{0, \infty }_{n,\alpha,L}$ solves ${\rm(NS2D)}_{x_3}$ with data~$v_{0,n,\alpha,L}^{0,\infty  }$ and~$w^{0, \infty,3}_{n,\alpha,L}
$ solves the transport-diffusion equation~$(T_{h_n^0}) $ defined page~\pageref{defTbetapage} with data~$ w^{0, \infty, 3}_{0,n,\alpha,L}$. Finally we recall that
$$
 w^{0, \infty, {\rm h}}_{n,\alpha,L} = - \nabla^{\rm h} \Delta_{\rm h}^{-1} \partial_3 w^{0, \infty, 3}_{n,\alpha,L}\,.
 $$  
 Similarly defining
 $$
  \Phi^{0,  {\rm loc}}_{0,n,\alpha,L} \eqdefa  \bigl[\bigl( v_{0,n,\alpha,L}^{0, {\rm loc}  }+h_n^0 w_{0,n,\alpha,L}^{0, {\rm loc}, {\rm h} },   w_{0,n,\alpha,L}^{0, {\rm loc}, 3 }\bigr)  \bigr]_{h_n^0}\, ,
 $$
 then for~$h_n^0$ smaller than~$\eps_1\big(\big\|(v_{0,n,\alpha,L}^{0,{\rm loc}  },
  w_{0,n,\alpha,L}^{0, {\rm loc}, 3 }
)\big\|_{S_\mu}\big)$ there is a unique global smooth solution~$\Phi^{0,  {\rm loc}}_{n,\alpha,L}$  associated with~$  \Phi^{0, {\rm loc}}_{0,n,\alpha,L}$, which belongs to~$\cA_0$, and 
\begin{equation}\label{decompositionlocpart}
\begin{aligned}
\Phi^{0,  {\rm loc}}_{n,\alpha,L}&  \eqdefa   \Phi^{0,  {\rm loc}, {\rm app}}_{n,\alpha,L} +\psi^{0,  {\rm loc}}_{n,\alpha,L} 
\quad \mbox{with} \\
  \Phi^{0,  {\rm loc}, {\rm app}}_{n,\alpha,L} & \eqdefa \big[
 v^{0,  {\rm loc} }_{n,\alpha,L}+h_n^0 w^{0,  {\rm loc}, {\rm h}}_{n,\alpha,L}, w^{0,  {\rm loc}, 3}_{n,\alpha,L}
 \big]_{h_n^0} \quad \mbox{and} \\
 \|\psi^{0,  {\rm loc}}_{n,\alpha,L} \|_{L^2(\R^+;\cB^1)} &\lesssim h_n^0 \cT_2\big(\big\|(v_{0,n,\alpha,L}^{0, {\rm loc}  },
  w_{0,n,\alpha,L}^{0, {\rm loc}, 3 }
)\big\|_{S_\mu}\big)  \, ,
\end{aligned}
\end{equation}
where~$v^{0,  {\rm loc} }_{n,\alpha,L}$ solves ${\rm(NS2D)}_{x_3}$ with data~$v_{0,n,\alpha,L}^{0, {\rm loc}  }$ and~$w^{0,  {\rm loc} ,3}_{n,\alpha,L}
$ solves~$(T_{h_n^0} )$ with data~$ w^{0,  {\rm loc}, 3}_{0,n,\alpha,L}$. Finally we recall that~$ w^{0,  {\rm loc}, {\rm h}}_{n,\alpha,L} = - \nabla^{\rm h} \Delta_{\rm h}^{-1} \partial_3 w^{0,  {\rm loc},3}_{n,\alpha,L}$.

\medskip
\noindent Now we look for 
the 
solution under the form
$$
\Phi^{0}_{n,\alpha,L} \eqdefa u_\alpha +  \Phi^{0, \infty}_{n,\alpha,L} +\Phi^{0,{\rm loc}}_{n,\alpha,L} + \psi_{n,\alpha,L} \, .
$$
In the next section we shall prove localization properties on~$ \Phi^{0, \infty}_{n,\alpha,L} $ and~$\Phi^{0,{\rm loc}}_{n,\alpha,L} $, namely the fact that~$\Phi^{ 0,\infty, {\rm app}}_{n,\alpha,L} $ escapes to infinity in the space variable, while~$\Phi^{0, {\rm loc}, {\rm app}}_{n,\alpha,L}$ remains localized (approximately), and we shall also prove that~$\Phi^{0, {\rm loc}, {\rm app}}_{n,\alpha,L}$ remains small near~$x_3 = 0$. Let us  recall that as claimed by~(\ref{smallatzerothmdata}),~(\ref{mainresultprofilesdieq-2}) and~(\ref{mainresultprofilesdieq-1}), those properties are true for their respective initial data. Those   localization properties will enable us to prove, in Paragraph~\ref{conclusionproofinteractionprofilescale1}, that the function~$u_\al+  \Phi^{0, \infty}_{n,\alpha,L} +\Phi^{0,{\rm loc}}_{n,\alpha,L} $   is itself an approximate solution to (NS) for the Cauchy data~$u_{0,\al} +  \Phi^{0, \infty}_{0,n,\alpha,L} +\Phi^{0,{\rm loc}}_{0,n,\alpha,L}$.

\subsection{Localization properties of the approximate solution}\label{localizationpropertiesapprox}
One   important step in the proof of Theorem~\ref{interactionprofilescale1} consists in the following result.
\begin{prop}
\label{regulNS2D+}
{\sl Under the assumptions of Proposition~{\rm\refer{regulNS2D}}, the control of the value  of~$v$ at the point~$x_3=0$  is given by
 \beq
 \label{eqdim2value0}
\forall r\in [1,\infty] \,,\  \|v (\cdot, 0)\|_{\wt L^r (\R^+;B^{\frac2 r}_{2,1}(\R^2))} \lesssim \|v_0(\cdot,0)\|_{B^0_{2,1}(\R^2)} +\|v (\cdot,0)\|_{L^2(\R^2)} ^2 \, .
 \eeq
 Moreover  we have {\color{black}{for all~$\eta$ in~$ ]0,1[$ and~$\gamma$ in~$ \{0,1\}$}}, 
 \beq
 \label{eqdim2pseudoloc}
 \|(\g-\theta_{\rm h,\eta})v\|_{\cA^0}   \leq  \bigl\|(\g-\theta_{\rm h,\eta}) v_0\bigr \|_{\cB^0}\exp \cT_1(\| v_0\|_{\cB^0}) +\eta \cT_2(\|v_0\|_{S_\mu}) \, ,
 \eeq
 with $\theta_{\rm h,\eta}$ is the  truncation function defined by \eqref{defcutof}.
} \end{prop}
\begin{proof}
In this proof we omit for simplicity the dependence of the function spaces on the space~$\R^2$.
Let us remark that the proof of Lemma~1.1 of\ccite{chemin10} claims that for all~$x_3$ in~$ \R,$
\begin{equation}\label{proofofLemma1.1}
\begin{aligned}
\bigl(\D_k^{\rm h} (v(t,\cdot,x_3)\cdot\nabla^{\rm h} v(t,\cdot,x_3)) &\big  |\D_k^{\rm h} v(t,\cdot,x_3)\bigr)_{L^2 }
\\
& \quad \lesssim d_k(t,x_3) \|\nabla^{\rm h} v(t,\cdot,x_3)\|_{L^2}^2 \|\D_k^{\rm h} v(t,\cdot,x_3)\|_{L^2}
\end{aligned}
\end{equation}
where~$(d_k(t,x_3))_{k\in \ZZ} $ is a generic element  of the sphere of~$\ell^1(\ZZ)$.
A $L^2$ energy estimate in~$\R^2$  gives therefore, taking~$x_3=0$,
$$
\frac 1 2 \frac d {dt} \|\D_k^{\rm h} v (t,\cdot,0)\|_{L^2 }^2 +c2^{2k} \|\D_k^{\rm h} v(t,\cdot,0)\|_{L^2 }^2 
\lesssim d_k(t) \|\nabla^{\rm h} v(t,\cdot,0)\|_{L^2}^2 \|\D_k^{\rm h} v(t,\cdot,0)\|_{L^2} \, , 
$$
 {\color{black}{where~$(d_k(t))_{k\in \ZZ} $ belongs to the sphere of~$\ell^1(\ZZ)$.}}
After division by~$\|\D_k^{\rm h} v(t,\cdot,0)\|_{L^2}$  and time integration, we get
\begin{equation}\label{divisionetcv}
\begin{aligned}
 \|\D_k^{\rm h} v(\cdot,0)\|_{L^\infty(\R^+; L^2 )}& +c2^{2k}  \|\D_k^{\rm h} v(\cdot,0)\|_{L^1(\R^+; L^2 )}
\\& \quad {} \leq \|\D_k v_0(\cdot,0)\|_{L^2 } + C \int_0^\infty  d_k(t)\|\nabla^{\rm h} v(t,\cdot,0)\|_{L^2 }^2 dt \, .
\end{aligned}
\end{equation}By summation over~$k$ and in view of \refeq{standen}, we obtain Inequality\refeq{eqdim2value0} of Proposition\refer{regulNS2D+}.

\medbreak
\noindent In order to prove Inequality\refeq{eqdim2pseudoloc}, let us define~$v_{\gamma,\eta} \eqdefa (\g-\theta_{\rm h,\eta})v $ and write that 
\beq
\begin{split}
&\partial_t v_{\gamma,\eta} -\D_{\rm h} v_{\gamma,\eta}+\dive_{\rm h} \bigl( v\otimes v_{\gamma,\eta} \bigr)  = E_\eta(v) =\sum_{i=1} ^3 E^i_\eta(v)\with\\
& E^1_\eta(v) \eqdefa-2 \eta(\nabla^{\rm h}\theta)_{{\rm h},\eta} \nabla^{\rm h} v-\eta^2 (\D_{\rm h} \theta)_{{\rm h},\eta} v\,,\ \\
& E^2_\eta (v) \eqdefa  \eta\, v\cdot(\nabla^{\rm h} \theta)_{{\rm h},\eta} v\andf\\
&E^3_\eta(v) \eqdefa - (\g-\theta_{{\rm h},\eta}) \nabla^{\rm h} \Delta_{\rm h}^{-1}\!\!\! \sum_{1\leq \ell,m\leq 2}\!\!\! \partial_\ell\partial_m\bigl( v^\ell v^m\bigr) \, .
\end{split}
\eeq
Let us prove that
\beq
\label{eqdim2pseudolocdemoeq111}
\|E_\eta(v)\|_{L^1(\R^+;\cB^{0})} \lesssim \eta  \,
\cT_2(\|v_0\|_{S_\mu}) \, .
\eeq
Using  Inequality\refeq{eqdim21demoeq1111} applied with~$r=1$ and~$s=-1$ (resp.~$r=2$ and~$s=-1/2$) this will follow from  
\beq
\label{eqdim2pseudolocdemoeq1}
\|E_\eta(v)\|_{L^1(\R^+;\cB^{0})} \lesssim \eta 
\bigl(\|v\|_{L^1(\R^+;\cB^1)} + \|v\|^2_{L^2(\R^+;\cB^{\frac 12 })}\bigr) \, .
\eeq
Proposition\refer{lawproductaniso2/3D} and the scaling  properties of homogeneous Besov spaces  give
\beno
\|(\nabla^{\rm h}\theta)_{{\rm h},\eta} \nabla^{\rm h} v (t)\|_{\cB^{0}} & \lesssim & 
 \|(\nabla^{\rm h}\theta)_{{\rm h},\eta}\|_{B^{1}_{2,1}(\R^2)}\| \nabla^{\rm h} v(t)\|_{\cB^{0}}\\
& \lesssim & 
 \|\nabla^{\rm h}\theta\|_{B^{1}_{2,1}(\R^2)}\|  v(t)\|_{\cB^1} \,  .
\eeno
Following the same lines, we get
\beno
\| (\D_{\rm h} \theta)_{{\rm h},\eta} v (t)\|_{\cB^{0}} & \lesssim & 
 \|(\D_{\rm h} \theta)_{{\rm h},\eta}\|_{B^{0}_{2,1}(\R^2)}\|  v (t)\|_{\cB^1}\\
& \lesssim & 
\frac 1 \eta  \|\D_{\rm h}\theta\|_{B^{0}_{2,1}(\R^2)}\|  v(t)\|_{\cB^1} \, , 
\eeno
hence
\beq
\label{eqdim2pseudolocdemoeq2}
\| E^1_\eta(v)\|_{L^1(\R^+;\cB^{0})} \lesssim \eta \|v\|_{L^1(\R^+; \cB^1)} \,.
\eeq
Let us study the term~$E_\eta^2(v)$.  Proposition\refer{lawproductaniso2/3D} implies
\beno
\|v(t)\cdot(\nabla^{\rm h} \theta)_{{\rm h},\eta} v (t)\|_{\cB^0} 
& \lesssim & 
\|(\nabla^{\rm h} \theta)_{{\rm h},\eta}\|_{B^1_{2,1}(\R^2)} \sup_{\ell, m } \|v^\ell (t) v^m (t)\|_{\cB^0} \\
& \lesssim & 
\|\nabla^{\rm h} \theta\|_{B^1_{2,1}(\R^2)}\|v(t)\|_{\cB^{\frac12}}^2.
\eeno
Thus we get
\beq
\label{eqdim2pseudolocdemoeq3}
\| E^2_\eta(v)\|_{L^1(\R^+;\cB^{0})} \lesssim \eta \|v\|_{L^2(\R^+;\cB^{\frac 12 })}^2 \, .
\eeq
Let us study the term~$E_\eta^3(v)$ which is related to the pressure. For that purpose, we shall make use of  the horizontal    paraproduct decomposition:
$$ av  =  T^{\rm h} _v a+ T^{\rm h} _a v +R^{\rm h}(a,b) \with
 T^{\rm h} _a b   \eqdefa \sum_{k} S^{\rm h}_{k-1} a \D^{\rm h}_kb \andf
 R^{\rm h} (a ,b)  \eqdefa \sum_{k}\wt  \D^{\rm h}_{k} a \D^{\rm h}_kb\,. 
$$
This allows us to write
\beq
\label{eqdim2pseudolocdemoeq4}
\begin{split}
E_\eta^3(v) & = \sum_{\ell=1}^3 E_\eta^{3,\ell}(v)\with\\
E_\eta^{3,1}(v) & \eqdefa \wt T^{\rm h} _{\nabla^{\rm h} p}  \theta_{{\rm h},\eta} \with \nabla^{\rm h} p = \nabla^{\rm h} \D_{\rm h}^{-1}\sum_{1\leq \ell,m\leq 2} \partial_\ell\partial_m (v^\ell v^m)\,,\\
E_\eta^{3,2}(v) & \eqdefa -\sum_{1\leq \ell,m\leq 2} \bigl[T^{\rm h}_{\g-\theta_{{\rm h},\eta}}, \nabla^{\rm h}\D_{\rm h} ^{-1} \partial_\ell\partial_m\bigr]v^\ell v^m \andf \\
E_\eta^{3,3}(v) & \eqdefa  \sum_{1\leq \ell,m\leq 2} \nabla^{\rm h}\D_{\rm h} ^{-1} \partial_\ell\partial_m \wt T^{\rm h}_ {v^\ell v^m} \theta_{{\rm h},\eta}.
\end{split}
\eeq
Laws of (para)product, as given in~(\ref{lawproductaniso2/3Ddemoeq1}), and scaling properties of Besov spaces give
\beno
\|\wt T^{\rm h} _{\nabla^{\rm h} p(t)} \theta_{{\rm h},\eta} \|_{\cB^{0 }} & \lesssim & \|\nabla^{\rm h} p(t)\|_{\cB^{-1}} \|\theta_{{\rm h},\eta}\|_{B^2_{2,1}(\R^2)} \\
 & \lesssim & \eta \sup_{1\leq \ell,m\leq 2} \|v^\ell (t) v^m (t)\|_{\cB^{0 }} \|\theta\|_{B^2_{2,1}(\R^2)}\\
 & \lesssim & \eta  \,  \|v(t)\|_{\cB^{\frac12 }}^2  \|\theta\|_{B^2_{2,1}(\R^2)} \,.
\eeno
Along the same lines we get
\beno
\|\nabla^{\rm h}\D_{\rm h} ^{-1} \partial_\ell\partial_m \wt T^{\rm h}_ {v^\ell(t)v^m(t)} \theta_{{\rm h},\eta}\|_{\cB^{0}} 
& \lesssim & \|\wt T^{\rm h}_ {v^\ell (t) v^m (t)} \theta_{{\rm h},\eta}\|_{\cB^{1}}\\
& \lesssim &  \|v^\ell (t) v^m (t)\|_{\cB^{0}}\|\theta_{{\rm h},\eta}\|_{B^2_{2,1}(\R^2)}\\
& \lesssim & \eta   \, \|v (t)\|_{\cB^{\frac12}}^2  \|\theta\|_{B^2_{2,1}(\R^2)} \,.
 \eeno
 This gives
 \beq
 \label{eqdim2pseudolocdemoeq5}
 \|E_\eta^{3,1}(v)+ E_\eta^{3,3}(v)\|_{L^1(\R^+;\cB^0)} \lesssim \eta  \, \|v\|_{L^2(\R^+;\cB^{\frac 12 })}^2 \, .
  \eeq
  Now let us estimate~$E_\eta^{3,2}(v)$. By definition, we have
 $$
 \displaylines{
  \bigl [T^{\rm h}_{\g-\theta_{{\rm h},\eta}}, \nabla^{\rm h}\D_{\rm h} ^{-1} \partial_\ell\partial_m\bigr]v^\ell 
 v^m   = \sum_{k} \cE_{k,\eta}(v)\with \cr 
 \cE_{k,\eta}(v)\eqdefa \bigl[S^{\rm h}_{k-N_0} (\gamma-\theta_{{\rm h},\eta}), \wt \D_k^{\rm h}\nabla^{\rm h}\D_{\rm h} ^{-1} \partial_\ell\partial_m\bigr] \D^{\rm h}_k(v^\ell v^m)
 }
$$
where~$\wt \D^{\rm h}_k\eqdefa \wt \vf (2^{-k}\xi_{\rm h})$ with~$\wt \vf$ is a smooth compactly supported (in~$\R^2\setminus\{0\}$) function which has value~$1$  near~$B(0, 2^{-N_0})+\cC$, where $\cC$ is an adequate annulus. Then by commutator estimates (see for instance Lemma 2.97 in  \cite{BCD}) 
\beno
\|\D_j^{\rm v} \cE_{k,\eta}(v(t))\|_{L^2} &  \lesssim & \|\nabla \theta_{{\rm h},\eta}\|_{L^\infty} \|\D_k^{\rm h}\D_j^{\rm v}(v^\ell (t) v^m (t))\|_{L^2} \,.
\eeno
As~$\|\nabla \theta_{{\rm h},\eta}\|_{L^\infty}=\eta\|\nabla \theta\|_{L^\infty}$,  by characterization of anisotropic Besov spaces and laws of product, we get 
$$
 \|E_\eta^{3,2}(v) \|_{L^1(\R^+;\cB^0)}  \lesssim \eta  \|v\|_{L^2(\R^+;\cB^{\frac 12 })}^2 \, .
$$
 Together with estimates\refeq{eqdim2pseudolocdemoeq2}--(\ref{eqdim2pseudolocdemoeq5}), this gives\refeq{eqdim2pseudolocdemoeq1}, hence~(\ref{eqdim2pseudolocdemoeq111}).

 \noindent Applying Lemma\refer{propagationlemmaheathorizontal} with ~$s=0$, ~$s'=1/2$, ~$a=v_{\gamma,\eta}$, ~$\cQ(v,a)=\dive_{\rm h} (v\otimes a)$, ~$f=E_\eta(v)$ and ~$\beta = 0$   allows to conclude the proof of Proposition\refer{regulNS2D+}.
 \end{proof}

\medbreak

\noindent A similar result holds for the solution~$w^3$
 of
$$
(T_\b) \quad 
\partial_t w^3 + v\cdot \nabla^{\rm h}w^3-\Delta_{\rm h} w^3- \b^2 \partial_3^2 w^3= 0 \andf
w^3_{|t=0} = w_{0}^3 \, ,
$$ 
where~$\b$ is any non negative real number. In the following statement,   all the constants are independent of~$\b$.
% \begin{equation}
% \label{eqdim21transportiff}
% \|  w^3 \|_{\cA^0} \lesssim \| w^3_0\|_{\cB^0}\exp \bigl(\cT_1( \| v_0 \|_{\cB^0})\bigr).
% \end{equation}
 \begin{prop}
\label{regiultransportdiff2D+}
{\sl Let~$v$ and~$w_3$  be as in Proposition~{\rm\refer{regiultransportdiff2D}}.  
 The control of the value of~$w^3$ at the point~$x_3=0$  is given by the following inequality. For any~$r$ in~$ [2,\infty]$,
 \beq
 \label{eqdim2value0transportiff}  
\begin{aligned}
  \|w^3 (\cdot, 0)\|_{\wt L^r (\R^+;B^{\frac2 r}_{2,1}(\R^2))} \leq 
   \cT_2(\|(v_0,w_0^3)\|_{S_\mu}) \Big(
     \|w^3_0(\cdot,0)\|^\frac{1-2\mu}{4(1-\mu)}_{B^0_{2,1}(\R^2)} 
     %+ \|v^{\rm h} (\cdot,0)\|_{L^2(\R^2)} 
     + \beta
   \Big) \, .
  \end{aligned}
 \eeq
 Moreover, with the notations of Theorem{\rm\refer{slowvarsimple}}, we have {\color{black}{for all~$\eta$ in~$ ]0,1[$ and~$\gamma$ in~$\{0,1\}$}}, 
 \beq
 \label{eqdim2pseudoloctransportiff}
 \|(\g-\theta_{{\rm h},\eta})w^3\|_{\cA^0}   \leq  \bigl\|(\g-\theta_{{\rm h},\eta}) w^3_{0}\bigr \|_{\cB^0}\exp \cT_1(\| v_0\|_{\cB^0}) +\eta \cT_2 (\|(v_0,w_0^3)\|_{S_\mu})\, .
 \eeq
}\end{prop}

\begin{proof}
The proof is
very similar to the proof of Proposition~\ref{regulNS2D+}. The main difference lies in the proof of~(\ref{eqdim2value0transportiff}) due to the presence of the extra term~$ \b^2 \partial_3^2 w^3$, so let us detail that estimate: we shall first prove an estimate for~$ w^3 (t,x_h,0) $ in~$\wt L^r (\R^+;B^{\frac12+\frac2 r}_{2,1}(\R^2))$, and then  we shall interpolate that estimate with the  known a priori estimate \eqref{eqdim22transportiff} of~$w^3$ in~$\wt L^r (\R^+;\cB^{-\frac12+\frac2 r}_{2,1}(\R^2))$ to find the result.

\noindent Let us be more precise, and first obtain a bound for~$ w^3 (t,x_h,0) $ in~$\wt L^r (\R^+;B^{\frac12+\frac2 r}_{2,1}(\R^2))$.
Defining
$$
\widetilde w^3 (t,x_h) \eqdefa w^3 (t,x_h,0)  \, , \quad \widetilde w^3_0 (x_h) \eqdefa w^3_0 (x_h,0) \quad \mbox{and} \quad \widetilde v (t,x_h) \eqdefa v (t,x_h,0) \, , 
$$
we have
 \beq
 \label{eqdim2wtilde3}
\partial_t \widetilde w^3 + \widetilde v\cdot \nabla^{\rm h}\widetilde w^3-\Delta_{\rm h} \widetilde w^3=  \b^2 (\partial_3^2 w^3)  (\cdot ,0)\andf
 \widetilde w^3_{|t=0} =    \widetilde w^3_0  \, .
 \eeq
Similarly to~(\ref{proofofLemma1.1}) we write (dropping for simplicity the dependence of the spaces on~$\R^2$)
$$
\bigl(\D_k^{\rm h} ( \widetilde v \cdot\nabla^{\rm h}  \widetilde  w^3 )\big |\D_k^{\rm h}  \widetilde  w^3\bigr)_{L^2 }
\lesssim d_k(t )\, 2^{-\frac k2} \|\nabla^{\rm h}  \widetilde  v \|_{L^2}   \|\nabla^{\rm h}  \widetilde  w^3 \|_{B^{\frac12}_{2,1}}   \|\D_k^{\rm h}  \widetilde  w^3\|_{L^2} \, , 
$$ 
{\color{black}{where~$(d_k(t))_{k\in \ZZ} $ belongs to the sphere of~$\ell^1(\ZZ)$.}}
Taking the $L^2$ scalar product of $\D_k^{\rm h}$ of  Equation \eqref{eqdim2wtilde3} with $\D_k^{\rm h}  \widetilde  w^3$ implies that
$$
\begin{aligned}
\frac 1 2 2^{\frac k2}\frac d {dt} \|\D_k^{\rm h}   \widetilde  w^3\|_{L^2 }^2 +c2^{\frac {5k}2} \|\D_k^{\rm h}   \widetilde  w^3\|_{L^2 }^2 
& \lesssim d_k(t)  \|\nabla^{\rm h}  \widetilde  v (t)\|_{L^2}   \|\nabla^{\rm h}  \widetilde  w^3 \|_{B^{\frac 12}_{2,1}}   \|\D_k^{\rm h}  \widetilde  w^3\|_{L^2}
\\
& \quad +\beta^22^{\frac k2} \| \D_k^{\rm h}(\partial_3^2 w^3)  (\cdot ,0) \|_{L^2}   \|\D_k^{\rm h}  \widetilde  w^3\|_{L^2}
\, , 
\end{aligned}
$$
so as in~(\ref{divisionetcv}) we find
$$
\begin{aligned}
& 2^{\frac k2} \|\D_k^{\rm h} \widetilde  w^3\|_{L^\infty(\R^+; L^2 )} +c2^{\frac {5k}2}  \|\D_k^{\rm h} \widetilde  w^3\|_{L^1(\R^+; L^2 )}\leq 2^{\frac k2} \|\D_k \widetilde  w^3_0 \|_{L^2 } 
\\& \quad {} + C \int_0^\infty  d_k(t)\|\nabla^{\rm h}\widetilde  v (t) \|_{L^2 }  \|\nabla^{\rm h}  \widetilde  w^3  (t) \|_{B^{\frac 12}_{2,1}} dt 
+ C\beta^2  \int_0^\infty  2^{\frac k2} \| \D_k^{\rm h}(\partial_3^2 w^3)  (t,\cdot ,0) \|_{L^2}   dt  
\, .
\end{aligned}
$$
After summation  we find that
$$
\longformule{
\|\widetilde  w^3\|_{\wt L^\infty(\R^+; B^{\frac 12}_{2,1} )}+\|\widetilde  w^3\|_{L^1(\R^+; B^{\frac 52}_{2,1} )}  
}
{
{}\lesssim\| \widetilde  w^3_0 \|_{B^{\frac 12}_{2,1} } +   \|\widetilde  w^3\|_{  L^2(\R^+; B^{\frac 32}_{2,1} )}\|\nabla^{\rm h}\widetilde  v\|_{L^2(\R^+;L^2)}+ \beta^2 \|(\partial_3^2 w^3)  (\cdot ,0) \|_{L^1(\R^+;B^{\frac 12}_{2,1})} \,.
}
$$
This is 
exactly an inequality of the type~(\ref{propagationlemmaheathorizontaldeoeq1}), up to a harmless localization in time, so by the same arguments we obtain the same conclusion as in Lemma~\ref{propagationlemmaheathorizontal}, namely the fact that for all~$ r\in [1,\infty]$, 
 $$
\|\widetilde  w^3\|_{\wt L^r(\R^+; B^{\frac12+\frac2 r}  )} \lesssim
 \bigl(\| \widetilde  w^3_0 \|_{B^{\frac 12}_{2,1} }+ \beta^2 \|(\partial_3^2 w^3)  (\cdot ,0) \|_{L^1(\R^+;B^{\frac 12}_{2,1})} \bigr)\exp C\|v_0 (\cdot,0)\|_{L^2 }^2\, .
 $$
Since we have 
$$
\| (\partial_3^2 w^3)  (\cdot ,0) \|_{L^1(\R^+;B^{\frac 12}_{2,1} (\R^2))} \lesssim \| w^3\|_{L^1(\R^+;\cB^{\frac 12,\frac52})}  
$$
we infer from the a priori bounds \eqref{useestwithr1} obtained on~$w^3$ in the previous section  that
$$
\| (\partial_3^2 w^3)  (\cdot ,0) \|_{L^1(\R^+;B^{\frac 12}_{2,1} (\R^2))} \lesssim \cT_2 (\|(v_0,w_0^3)\|_{S_\mu})\, ,
$$
so we obtain that for any~$r$ in~$ [1,\infty]$, 
 \begin{equation}  \label{one}
  \|w^3 (\cdot, 0)\|_{\wt L^r (\R^+;B^{\frac12+\frac2 r}_{2,1}(\R^2))} \leq  \big( \|w^3_0(\cdot,0)\|_{B^\frac12_{2,1}(\R^2)}   + \beta^2\big) \cT_2(\|(v_0,w_0^3)\|_{S_\mu})\, . 
  \end{equation}
  Recalling that $w^3_0$ belongs to the space  $S_\mu$ introduced in Definition \ref{definitionspacesmu}, we find that 
  $$ 
  w^3_0 (\cdot, 0) \in \bigcap_{s\in [-2+\mu,1-\mu]} B^{s}_{2,1} (\R^2)\, .
  $$ 
Since $   \ds 0 <\mu <   \frac 1 2 $, we get by interpolation and Sobolev embeddings that   
$$
 \|w^3_0(\cdot,0)\|_{B^\frac12_{2,1}(\R^2)} \lesssim  \|w^3_0(\cdot,0)\|^\frac{1-2\mu}{2(1-\mu)}_{B^0_{2,1}(\R^2)} \|w^3_0 \|^\frac{1}{2(1-\mu)}_{S_\mu}\, ,
 $$  
which implies that   \eqref{one}  can be written under the form
  $$
 \begin{aligned}
  \|w^3 (\cdot, 0)\|_{\wt L^r (\R^+;B^{\frac12+\frac2 r}_{2,1}(\R^2))} \leq  \Big( \|w^3_0(\cdot,0)\|^\frac{1-2\mu}{2(1-\mu)}_{B^0_{2,1}(\R^2)} \ + \beta^2 \Big)\cT_2(\|(v_0,w_0^3)\|_{S_\mu})\, .
  \end{aligned}
$$
Now interpolating with the a priori bound obtained in Proposition~\ref{regiultransportdiff2D}, we find
  $$
\begin{aligned}
   \|w^3 (\cdot, 0)\|_{\wt L^r (\R^+;B^{-\frac12+\frac2 r}_{2,1}(\R^2))} & \lesssim    \|w^3 \|_{\wt L^r (\R^+;\cB^{-\frac12+\frac2 r})}  \\
   & \lesssim  \cT_2(\|(v_0,w_0^3)\|_{S_\mu}\, ,
      \end{aligned}
$$
so we obtain finally
$$
\begin{aligned}
  \|w^3 (\cdot, 0)\|_{\wt L^r (\R^+;B^{\frac2 r}_{2,1}(\R^2))} \leq 
   \cT_2(\|(v_0,w_0^3)\|_{S_\mu}) \Big(
     \|w^3_0(\cdot,0)\|^\frac{1-2\mu}{4(1-\mu)}_{B^0_{2,1}(\R^2)} +  \beta
   \Big) \, . 
  \end{aligned}
$$
This ends the proof of~(\ref{eqdim2value0transportiff}).
 
\medskip
\noindent We shall not detail the proof of~(\ref{eqdim2pseudoloctransportiff}) as it is very similar to the proof of~(\ref{eqdim2pseudoloc}). Proposition~\ref{regiultransportdiff2D+} is therefore proved. \end{proof}

\medskip

\noindent Propositions~\ref{regulNS2D+} and~\ref{regiultransportdiff2D+} imply easily  the following result,  using the special form of~$ \Phi^{0, \infty}_{n,\alpha,L} $ and~$\Phi^{0,{\rm loc}}_{n,\alpha,L}$  recalled in~(\ref{decompositioninftypart}) and~(\ref{decompositionlocpart}), and thanks to~(\ref{smallatzerothmdata}),~(\ref{mainresultprofilesdieq-2}) and~(\ref{mainresultprofilesdieq-1}).

\begin{cor}\label{corollarylocalization}
{\sl The vector fields~$\Phi^{0,{\rm loc}}_{n,\alpha,L} $ and~$  \Phi^{0, \infty}_{n,\alpha,L} $ satisfy the following:~$\Phi^{0,{\rm loc}}_{n,\alpha,L} $ vanishes at~$x_3 = 0$, in the sense that for all~$r$ in~$[2, \infty]$,
$$
 \lim_{L \to \infty} \lim_{\alpha \to 0}
 \limsup_{n \to \infty}
 \|\Phi^{0,{\rm loc}}_{n,\alpha,L} (\cdot, 0)\|_{\wt L^r (\R^+;B^{\frac2 r}_{2,1}(\R^2))} = 0 \, ,
$$
and there is a constant~$ C(\alpha, L) $ such that for all~$\eta$ in~$ ]0,1[$,  
 $$
 \limsup_{n \to \infty}
 \Big(  \|(1-\theta_{{\rm h},\eta})\Phi^{0,{\rm loc}}_{n,\alpha,L}\|_{\cA^0} +    \| \theta_{{\rm h},\eta}\Phi^{0,\infty}_{n,\alpha,L}\|_{\cA^0} \Big) \leq C(\alpha, L) \eta \,  .
$$
}\end{cor}
\subsection{Conclusion of the proof of Theorem~\ref{interactionprofilescale1}} \label{conclusionproofinteractionprofilescale1}
Recall that we look for the solution of~(NS) under the form
$$
\Phi^{0}_{n,\alpha,L}  = u_\alpha+  \Phi^{0, \infty}_{n,\alpha,L} +\Phi^{0,{\rm loc}}_{n,\alpha,L} + \psi_{n,\alpha,L} \, , 
$$
with the notation introduced in Paragraph~\ref{decompositioninitialdataprofile1}. In particular the two vector fields~$\Phi^{0,{\rm loc}}_{n,\alpha,L} $ and~$  \Phi^{0, \infty}_{n,\alpha,L} $ satisfy Corollary~\ref{corollarylocalization}, and  furthermore     thanks to the Lebesgue theorem,
\begin{equation}\label{ugoestozeroaswell}
\lim_{\eta \to 0} \|(1-\theta_{\eta}) u_\alpha\|_{  L^2(\R^+;\cB^1)} = 0 \, .
\end{equation}
Given a small number~$\e>0$, to be chosen later, we choose~$L $,~$\alpha$  and~$\eta = \eta(\alpha,L,u_0)$ so that thanks to  Corollary~\ref{corollarylocalization} and~(\ref{ugoestozeroaswell}), for all~$r$ in~$ [2,\infty]$, and for~$n$ large enough,
\begin{equation}\label{choiceofallparametersepsilon}
 \begin{aligned}
  \|\Phi^{0,{\rm loc}}_{n,\alpha,L} (\cdot, 0)\|_{  L^r (\R^+;B^{\frac2 r}_{2,1}(\R^2))} + \|(1-\theta_{{\rm h},\eta})\Phi^{0,{\rm loc}}_{n,\alpha,L}\|_{\cA^0}   + \|(1-\theta_{\eta}) u_\alpha\|_{  L^2(\R^+;\cB^1)} \\
  +    \| \theta_{{\rm h},\eta}\Phi^{0,\infty}_{n,\alpha,L}\|_{\cA^0} \leq \e \, .
\end{aligned}
\end{equation}
In the following 
we denote for simplicity
$$
\begin{aligned}
 (\Phi^{0, \infty}_{\eps}, \Phi^{0,{\rm loc}}_{\e} ,  \psi_{\e}) & \eqdefa(  \Phi^{0, \infty}_{n,\alpha,L},  \Phi^{0,{\rm loc}}_{n,\alpha,L}  , \psi_{n,\alpha,L} ) \andf  \Phi^{{\rm app}}_{\eps} \eqdefa u_\alpha+ \Phi^{0, \infty}_{\eps} +\Phi^{0,{\rm loc}}_{\eps}   \, ,
\end{aligned}
$$
so the vector field~$ \psi_{\e} $ satisfies the following equation, with zero initial data:
\beq
\label{definRemainder}
\begin{split}
&\partial_t \psi_{\e}-\D\psi_{\e} +\dive \bigl( \psi_{\e}\otimes \psi_{\e}+\Phi^{\rm app}_{\e}\otimes\psi_{\e}+\psi_{\e}\otimes \Phi^{\rm app}_{\e}\bigr)=  - \nabla q_\e + E_\e \,,\\
&\quad \with E_\e=  E_\e^1 +E_\e^2  \andf\\
&\qquad\quad\qquad E^1_\e \eqdefa\dive\Big (\Phi^{0, \infty}_{\eps}\otimes (\Phi^{0,{\rm loc}}_{\e}+u_\alpha) +(   \Phi^{0,{\rm loc}}_{\e} +u_\alpha)\otimes\Phi^{0, \infty}_{\eps} \\
&\qquad\quad\qquad\qquad\quad\qquad\qquad\quad\qquad
+ 
\Phi^{0,{\rm loc}} \otimes  (1- \theta_{ \eta} )u_\alpha +  (1-  \theta_{ \eta} )u_\alpha \otimes\Phi^{0,{\rm loc}}
\Big)\,,\\
&\qquad\quad\qquad E^2_\e \eqdefa  \dive\big(\Phi^{0,{\rm loc}}_{\e}\otimes \theta_{ \eta}u_\alpha  +  \theta_{\eta}u_\alpha \otimes \Phi^{0,{\rm loc}}_{\e}\big)  \, .
\end{split}
\eeq
If we prove that 
\beq
\label{slowvarsimpleconcludemoeq2profile1}
\lim_{\e \to 0}\|E_\e\|_{\cF^0}=0 \, , 
\eeq
then Proposition\refer{existencepetitcB1} implies  that~$\psi_\e$ belongs to~$  L^2(\R^+;\cB^1)$, with
$$
\lim_{\e \to 0}\|\psi_\e\|_{  L^2(\R^+;\cB^1)}= 0 \, ,
$$
and we conclude the proof of Theorem~\ref{interactionprofilescale1} exactly as in the proof of Theorem~\ref{slowvarsimple}, by resorting to Proposition~\ref{propaganisoNS3D}.

\medskip
\noindent So let us prove~(\ref{slowvarsimpleconcludemoeq2profile1}). The term~$E_\e^1$ is the easiest, thanks to the separation of the spatial supports. Let us first write~$E_\e^1= E_{\e,\rm h}^1+E_{\e,3}^1$ with
$$
\begin{aligned}
 E_{\e,\rm h}^1 & \eqdefa\dive_{\rm h}\Big ( (\Phi^{0,{\rm loc} }_{\e}+u_\alpha ) \otimes \Phi^{0, \infty,{\rm h}}_{\eps}+\Phi^{0, \infty}_{\eps}\otimes (   \Phi^{0,{\rm loc},{\rm h}}_{\e} +u_\alpha^{\rm h}) \\
& \quad\quad 
+ 
 (1- \theta_{ \eta} )u_\alpha \otimes \Phi^{0,{\rm loc},{\rm h}} + \Phi^{0,{\rm loc}} \otimes (1-  \theta_{ \eta} )u_\alpha^{\rm h}
\Big) \quad \mbox{and}\\
 E_{\e,3}^1 & \eqdefa\partial_3   \Big ((\Phi^{0,{\rm loc} }_{\e}+u_\alpha )\Phi^{0, \infty,3}_{\eps} +\Phi^{0, \infty}_{\eps}(    \Phi^{0,{\rm loc},3}_{\e} +u_\alpha^3) \\
& \quad \quad
+ 
   (1- \theta_{ \eta} )u_\alpha \Phi^{0,{\rm loc},3}  +   \Phi^{0,{\rm loc}}
(1-  \theta_{ \eta} )u_\alpha^{3} \Big)  \,.
\end{aligned}
$$
Next let us write, for any two functions~$a$ and~$b$,
$$
ab = (\theta_{{\rm h},\eta} a)b + a \big ( (1- \theta_{{\rm h},\eta}) b\big) \, .
$$
Denoting
$$
u^{\infty}_\eps \eqdefa ( 1-\theta_{ \eta})u_\alpha
$$
and using by now as usual the action of   derivatives and  the fact that~$\cB^1$ is an algebra,  we infer that
\[
\begin{split}
\| E_{\e,\rm h}^1 \|_{L^1(\R^+;\cB^0)} + \| E_{\e,3}^1 \|_{L^1(\R^+;B^{1,-\frac12}_{2,1})} & \leq \| \theta_{{\rm h},\eta}   \Phi^{0, \infty}_{\eps} \|_{L^2(\R^+;\cB^1)} \|\Phi^{0,{\rm loc}}_{\e}+u_\alpha \|_{L^2(\R^+;\cB^1)}\\
&\quad {}+ \|  (  1- \theta_{{\rm h},\eta}) (\Phi^{0,{\rm loc}}_{\e}+u_\alpha)\|_{L^2(\R^+;\cB^1)} \|  \Phi^{0, \infty}_{\eps} \|_{L^2(\R^+;\cB^1)}\\
& \qquad\qquad 
 {}
+ \|\Phi^{0,{\rm loc}}_{\e}\|_{L^2(\R^+;\cB^1)}  \| u^{\infty}_\eps\|_{L^2(\R^+;\cB^1)}  \, .
 \end{split}
\]
Thanks to~(\ref{choiceofallparametersepsilon})
and to the a priori bounds on~$  \Phi^{0, \infty}_{\eps} $,   $\Phi^{0, {\rm loc}}_{\eps} $ and~$u_\alpha$, we get directly in view of the examples page~\pageref{examplespage} that
$$
\lim_{\e \to 0}\|E^1_\e\|_{\cF^0}=0 \, .
$$
Next let us turn to~$E^2_\e$. We shall follow the method of~\cite{cgz}, and in particular the following lemma will be very useful.
\begin{lem}
\label{lemmainteracslowvarying}
 {\sl  There is a constant $C$ such that for all  
functions~$a$ and~$b$, we have
$$
\|ab\|_{\cB^1} \leq C\|a\|_{\cB^1}
\|b(\cdot ,0)\|_{B^{1}_{2,1} (\R^2)} + C\|x_3 a \|_{\cB^1}
\|\partial_3b\|_{\cB^1} \, .
$$
}\end{lem}
\noindent We postpone the proof of that lemma. Let us apply it to estimate $E^2_\e$. We write, as in the case of~$E^1_\e$ and defining~$u^{\rm loc}_\eps \eqdefa  \theta_{ \eta}u_\al $,
$$
\begin{aligned}
\|E^2_\e\|_{\cF^0}  \lesssim \|u^{\rm loc}_\eps \|_{L^2(\R^+;\cB^1)} \|\Phi^{0,{\rm loc}}_{\e}(\cdot, 0) \|_{L^2(\R^+;B^{1}_{2,1} (\R^2))}
\\
+ \|   x_3 u^{\rm loc}_\eps \|_{L^2(\R^+;\cB^1)} \|\partial_3 \Phi^{0,{\rm loc}}_{\e} \|_{L^2(\R^+;\cB^{1} )}\, .
\end{aligned}
$$
Thanks to~(\ref{choiceofallparametersepsilon}) as well as Inequality~(\ref{estimatesolutionunbis}) of Theorem~\ref{slowvarsimple},  we obtain
$$
\lim_{\e \to 0}\|E^2_\e\|_{\cF^0}=0 \, .
$$
This proves~(\ref{slowvarsimpleconcludemoeq2profile1}), hence Theorem~\ref{interactionprofilescale1}. \qed
 
\begin{proof}[Proof of Lemma~{\rm\ref{lemmainteracslowvarying}}]
This is essentially Lemma~3.3 of~\cite{cgz}, we recall the proof for the convenience of the reader.
Let us decompose~$b$ in the following way:
 \beq
\label{demointeractslowvareq1} b(x_{\rm h}, x_3) = b(x_{\rm h},0)+\int_0^{x_3}
\partial_3b(x_{\rm h},y_3)dy_3 \, .
 \eeq
  Laws of product   give directly on the one hand
  $$ 
   \|a (b_{|x_3=0}) \|_{\cB^1} \lesssim \|a\|_{\cB^1}
  \|b_{|x_3=0}\|_{  B^{1}_{2,1} (\R^2)} \,.
  $$
  On the other hand,  observe that 
   \beno \biggl \| a(\cdot
,x_3) \int_0^{x_3} \partial_3b(\cdot, y_3) dy_3\biggl\|_{  B^{1}_{2,1} (\R^2)} & \lesssim&  \|a(\cdot, x_3)\|_{ B^{1}_{2,1} (\R^2)}
\int_0^{x_3 } \| \partial_3b (\cdot
,y_3)\|_{  B^1_{2,1} (\R^2)} dy_3 \\
& \leq & C |x_3| \|a(\cdot ,x_3)\|_{  B^{1}_{2,1} (\R^2)} \|\partial_3
b\|_{L^\infty_v(  B^1_{2,1}(\R^2_{\rm h}))} \, .
\eeno 
The result follows.
\end{proof}

\medbreak

%%%%%%%%%%%%%%%%%%%%%%%%%%%%%%%%%%%%%%

\section{Some results in anisotropic  Besov spaces}
\label{lpanisodivers}
\subsection{Anisotropic Besov spaces}
\label{usefulresults}

 In this section we  first recall some basic facts about (aniso\-tropic) Littlewood-Paley theory and then we prove some basic properties of anisotropic Besov spaces  introduced in Definition\refer{deflpanisointro}, in particular laws of product  which have used all along this text.

 \medbreak
\noindent
First let us recall the following estimates which are the generalization  of the classical  Bernstein's inequalities in the context of anisotropic Littlewood-Paley theory (see Lemma~6.10 of\ccite{BCD}) describing the action of  horizontal and vertical derivatives on frequency localized distributions:
\begin{lem}
\label{Bernsteinaniso}
{\sl Let~$(p_1,p_2,r)$ be in~$[1,\infty]^3$ such that $p_1$ is less than or equal to~$p_2$. Let~$m$ be a real number and~$\s_{\rm h}$ (resp.~$\s_{\rm v})$  a  smooth homogeneous function  of degree~$m$ on~$\R^2$ (resp.~$\R$). Then we have
\beno
 \|\s_{\rm h} (D_{\rm h})\Delta_k^{\rm h}  f\|_{L^{p_2}_{\rm h} L^r_{\rm v}} 
 &\lesssim &
 2^{k(m+\frac2{p_1} - \frac2{p_2} )}
 \| \Delta_k^{\rm h}f\|_{L^{p_1}_{\rm h} L^r_{\rm v}} \andf 
 \\
\|\s_{\rm v}(D_3)  \Delta_{j}^{\rm v}   f\|_{L^r_{\rm h} L^{p_2}_{\rm v}} 
& \lesssim &
2^{j(m+\frac1{p_1} - \frac1{p_2})} \| \Delta_j ^{\rm v}f\|_{L^r_{\rm h} L^{p_1}_{\rm v}} \,.
\eeno
}
\end{lem}
\noindent
Now let us   recall  the action of the heat flow  on frequency localized  distributions in an anisotropic context.  
\begin{lem}
\label{anisoheat}
{\sl For any~$p$ in~$[1,\infty]$, we have 
 \begin{eqnarray*}
 \|e^{t\Delta}  \Delta_k^{\rm h} \Delta_{j}^{\rm v} f\|_{L^p} & \lesssim & e^{-ct (2^{2k} + 2^{2j})} \|  \Delta_k^{\rm h} \Delta_{j}^{\rm v} f\|_{L^p}\\
 \|e^{t\Delta_{\rm h}}  \Delta_k^{\rm h} \Delta_{j}^{\rm v} f\|_{L^p} &\lesssim& e^{-ct 2^{2k} } \|  \Delta_k^{\rm h} \Delta_{j}^{\rm v} f\|_{L^p}  \quad
 \mbox{and} \quad
 \\
\label{eq2}\|e^{t\partial_3^2}  \Delta_k^{\rm h} \Delta_{j}^{\rm v} f\|_{L^p} &\lesssim& e^{-ct 2^{2j} } \|  \Delta_k^{\rm h} \Delta_{j}^{\rm v} f\|_{L^p}\,.
\end{eqnarray*}
}
\end{lem}
\noindent The proof of this lemma consists  in a straightforward  (omitted) modification of the proof of Lemma~2.3 of\ccite{BCD}.
 
\medskip \noindent The following result was mentioned in the introduction of this article (see page~\pageref{examplespage}). We refer to~(\ref{deflrtildenorm}) and to Definition~\ref{definitionspaces}
‚àö√á¬¨¬®‚àö¬¢¬¨√Ñ¬¨‚Ä†for notations.
\begin{lem}
\label{example}
{\sl
The spaces~$\wt L^2(\R^+;\cB^{s-1,s'})$, $\wt L^2(\R^+;\cB^{s,s'-1})$ are~$\cF^{s,s'}$ spaces, as well as the spaces~$  L^1(\R^+;\cB^{s,s'})$ and~$  L^1(\R^+;\cB^{s+1,s'-1})$.
}
\end{lem}
\begin{proof}
Let~$f$ be a function in~$\wt L^2(\R^+;\cB^{s-1,s'})$, and let us   show that
$$
\|L_0 f\|_{\cA^{s,s'}} \lesssim \|f\|_{\wt L^2(\R^+;\cB^{s-1,s'})}.% \, ,
$$
%recalling that~$\ds L_0 f = \int^t_0 e^{(t-t')\Delta} f(t') \, dt$. 
Applying Lemma\refer{anisoheat} gives
$$
\| \Delta_k^{\rm h} \Delta_{j}^{\rm v} L_0 f\|_{L^2}  \lesssim \int^t_0e^{-ct '(2^{2k} + 2^{2j})}\|  \Delta_k^{\rm h} \Delta_{j}^{\rm v} f(t')\|_{L^2} \, dt'
$$
so there is a sequence~$d_{j,k}(t')$ in the sphere of~$\ell^1(\ZZ\times \ZZ;L^2(\R^+))$ such that
$$
\| \Delta_k^{\rm h} \Delta_{j}^{\rm v} L_0 f\|_{L^2}  \lesssim\|f\|_{\wt L^2(\R^+;\cB^{s-1,s'})} 2^{-k(s-1)} 2^{-js'}  \int^t_0e^{-ct '(2^{2k} + 2^{2j})}d_{j,k}(t')\, dt' \, .
$$
Young's inequality in time therefore gives   
$$
%\| \Delta_k^{\rm h} \Delta_{j}^{\rm v} L_0 f\|_{  L^\infty(\R^+; L^2)} 
%+2^{k+j} 
\| \Delta_k^{\rm h} \Delta_{j}^{\rm v} L_0 f\|_{  L^2(\R^+;L^2)} \lesssim\|f\|_{\wt L^2(\R^+;\cB^{s-1,s'})} 2^{-k(s-1)-js'}  d_{j,k}\, ,
$$
where~$d_{j,k}$ is a generic sequence in the sphere of~$\ell^1(\ZZ\times \ZZ)$, which proves the result in the case when~$f$ belongs to~$\wt L^2(\R^+;\cB^{s-1,s'})$.
The argument is similar in the other cases.
\end{proof} 

\medbreak
\noindent Now let us study  laws of product.
 \begin{prop}
 \label{productlawsaniso}
{\sl 
 Let~$(\s,\s',\wt \s,\wt \s')$ be in~$]-1,1]^4$ such that 
  $$
 \s+\s'=\wt \s+\wt \s'\eqdefa \overline \s >0 \, .
 $$
 If~$s'$ is in~$]-1/2,1/2]$, we have
 \begin{equation}
 \label{lawsanisogen}
 \|ab\|_{\cB^{\overline \s-1,s'}} \lesssim \|a\|_{\cB^{\s}} \|b\|_{\cB^{\s',s'}} \, .
  \end{equation}
 If~$s'$ is greater than~$1/2$, then we have
  \begin{equation}
 \label{lawsanisogen2}
 \|ab\|_{\cB^{\overline \s-1,s'}} \lesssim \|a\|_{\cB^{\s}} \|b\|_{\cB^{\s',s'}}+\|a\|_{\cB^{\wt \s',s'}}\|b\|_{\cB^{\wt \s}} \, .
 \end{equation}
 }
 \end{prop}
 \begin{proof}
 Let us use Bony's decomposition in the vertical variable introduced in\refeq{bonyvertical}, namely
$$ ab = T^{\rm v} _ab+T^{\rm v}_ba+R^{\rm v}(a,b).
$$
 The first  two terms are almost the same (up to the interchanging of~$a$ and~$b$). Thus we only estimate~$T^{\rm v}_ab$. This is done through the following lemma.
 \begin{lem}
 \label{productlawsanisodemolem}
{\sl  Let us consider~$(\s,\s')$ in~$]-1,1]^2$ such that $\s+\s'$ is positive and~$(s,s')$ in~$\R^2$. If~$s$ is less than or equal to~$1/2$, we have
\begin{equation}
 \label{lawsanisogenpara}
 \|T^{\rm v} _ab\|_{\cB^{\s+\s'-1,s+s'-\frac 12}} \lesssim\|a\|_{\cB^{\s,s}} \|b\|_{\cB^{\s',s'}} \, .
  \end{equation} 
If~$s+s'$ is positive, we have
\begin{equation}
 \label{lawsanisogenremain}
 \|R^{\rm v} (a,b) \|_{\cB^{\s+\s'-1,s+s'-\frac12}} \lesssim \|a\|_{\cB^{\s,s}} \|b\|_{\cB^{\s',s'}} \, .
  \end{equation}
}\end{lem}
 \begin{proof}
 Let us use Bony's decomposition of~$T^{\rm v}_ab$ with respect to the horizontal variable. 
 \beno
T^{\rm v}_a b & =& T^{\rm v}T^{\rm h}_a b + T^{\rm v}\wt T^{\rm h}_b a +  T^{\rm v}R^{\rm h}(a, b)\with\\
T^{\rm v}T^{\rm h}_a b & \eqdefa & \sum_{j,k} S_{j-1}^{\rm v} S_{k-1}^{\rm h} a\D_j^{\rm v}\D_k^{\rm h} b\,,\\
T^{\rm v}\wt T^{\rm h}_b a & \eqdefa & \sum_{j,k} S_{j-1}^{\rm v} \D_k^{\rm h} a\D_j^{\rm v}S_{k-1}^{\rm h} b\andf\\
T^{\rm v}R^{\rm h}(a, b) &\eqdefa & \sumetage {j,k}{-1\leq \ell\leq 1} 
S_{j-1}^{\rm v} \D_{k-\ell}^{\rm h} a\D_j^{\rm v}\D_k^{\rm h} b\,.
 \eeno
Following the same lines as in the proof of Proposition\refer{regulNS2D} (see the lines following decompostion\refeq{Bonydecompaniso}) we have for some large enough integer~$N_0$
$$
\D_j^{\rm v}\D_k^{\rm h} T^{\rm v}T^{\rm h}_a b= \sumetage{|j'-j|\leq N_0}{|k'-k|\leq N_0}
\D_j^{\rm v}\D_k^{\rm h}\bigl(S_{j'-1}^{\rm v}S_{k'-1} ^{\rm h} a\D_{j'}^{\rm v} \D_{k'}^{\rm h}b\bigr).
$$
By definition of the~$\cB^{\s,s'}$ norms, this gives
\beno
2^{j\left (s+s'-\frac12\right)+k(\s+\s'-1)} \|\D_j^{\rm v}\D_k^{\rm h} T^{\rm v}T^{\rm h}_a b\|_{L^2} 
& \lesssim & \!\!\!\!
 \sumetage{|j'-j|\leq N_0}{|k'-k|\leq N_0} 2^{-(j'-j)\left (s+s'-\frac12\right) -(k'-k) (\s+\s'-1)}\\
&&\!\!\!\!\!\!\!\!\!\!\!\!\!\!\!\!\!\!\!\!\!\!{}\times
2^{j'\left(s-\frac 12\right) +k'(\s-1)} \|S_{j'-1}^{\rm v}S_{k'-1} ^{\rm h} a\|_{L^\infty} 2^{j's'+k'\s'} \|\D_{j'}^{\rm v} \D_{k'}^{\rm h}b\|_{L^2}\\
& \lesssim &  \|b\|_{\cB^{\s',s'}}  \!\!\!\! \sumetage{|j'-j|\leq N_0}{|k'-k|\leq N_0} 2^{-(j'-j)s' -(k'-k) (\s+\s'-1)}\\
&&\ \quad\qquad{}\times   d_{j',k'}2^{j'\left(s-\frac 12\right) +k'(\s-1)}\|S_{j'-1}^{\rm v}S_{k'-1} ^{\rm h} a\|_{L^\infty}
\eeno
where, as in all that follows, ~$( d_{j,k})_{(j,k)\in \ZZ^2}$ lies on the sphere of~$\ell^1(\ZZ^2)$. Using anisotropic Bernstein inequalities given by Lemma\refer{Bernsteinaniso} and the definition of the~$\cB^{\s,s}$ norm, we get
\beno
2^{j'\left(s-\frac 12\right) +k'(\s-1)}\|S_{j'-1}^{\rm v}S_{k'-1}^{\rm h} a \|_{L^\infty} 
& \lesssim &\sumetage{j''\leq j'-2} {k''\leq k'-2}  2^{(j'-j'') \left(s-\frac 12\right) +(k'-k'')(\s-1)}\\
&&\qquad\qquad{}\times 2^{j''\left(s-\frac 12\right) +k''(\s-1)} \|\D_{j''}^{\rm v} \D_{k''}^{\rm h} a\|_{L^\infty}\\
& \lesssim & 
\sumetage{j''\leq j'-2} {k''\leq k'-2}  2^{(j'-j'') \left(s-\frac 12\right) +(k'-k'')(\s-1)}\\
&&\qquad\qquad{}\times 2^{j''s +k''\s} \|\D_{j''}^{\rm v} \D_{k''}^{\rm h} a\|_{L^2}\\
& \lesssim & \|a\|_{\cB^{\s,s}} \!\!\!
\sumetage{j''\leq j'-2} {k''\leq k'-2}  2^{(j'-j'') \left(s-\frac 12\right) +(k'-k'')(\s-1)}%\\
%&&\qquad\qquad{}\times 2^{j''s +k''\s} \|\D_{j''}^{\rm v} \D_{k''}^{\rm h} a\|_{L^2}\\.
 d_{j''k''} \, .
 \eeno
As~$s\leq 1/2$ and~$\s\leq 1$, we get
$$
2^{j'\left(s-\frac 12\right) +k'(\s-1)}\|S_{j'-1}^{\rm v}S_{k'-1}^{\rm h} a \|_{L^\infty}  \lesssim
\|a\|_{\cB^{\s,s}} \, .
$$
Young's inequality on series leads to
\beq
\label{productlawsanisodemolemdemoeq1}
\|T^{\rm v}T^{\rm h}_a b\|_{\cB^{\s+\s'-1,s+s'-\frac12}} \lesssim \|a\|_{\cB^{\s,s}}\|b\|_{\cB^{\s',s'}} \, .
\eeq
Following  exactly the same lines, we can prove \beq
\label{productlawsanisodemolemdemoeq2}
\|T^{\rm v}\wt T^{\rm h}_b a\|_{\cB^{\s+\s'-1,s+s'-\frac12}} \lesssim \|a\|_{\cB^{\s,s}}\|b\|_{\cB^{\s',s'}} \, .
\eeq
The estimate of~$T^{\rm v}R^{\rm h}(a, b) $ is a little bit different. Let us write that 
$$
\D_j^{\rm v}\D_k^{\rm h} T^{\rm v}R^{\rm h}(a, b)= 
 \sumetage {j',k'}{-1\leq \ell\leq 1} \D_j^{\rm v}\D_k^{\rm h}\bigl(S_{j'-1}^{\rm v}\D_{k'-\ell} ^{\rm h} a\D_{j'}^{\rm v} \D_{k'}^{\rm h}b\bigr)\,.
$$
Arguing as in the proof of Proposition\refer{regulNS2D}  we have for some large enough integer~$N_0$
$$
\D_j^{\rm v}\D_k^{\rm h} T^{\rm v}R^{\rm h}(a, b)= \sumetage{|j'-j|\leq N_0}{k'\geq k- N_0}\sum_{-1\leq \ell\leq 1}
\D_j^{\rm v}\D_k^{\rm h}\bigl(S_{j'-1}^{\rm v}\D_{k'-\ell} ^{\rm h} a\D_{j'}^{\rm v} \D_{k'}^{\rm h}b\bigr).
$$
Anisotropic Bernstein inequalities given by Lemma\refer{Bernsteinaniso} imply that 
\beno
\bigl\|\D_j^{\rm v}\D_k^{\rm h}\bigl(S_{j'-1}^{\rm v}\D_{k'-\ell} ^{\rm h} a\D_{j'}^{\rm v} \D_{k'}^{\rm h}b\bigr)\bigr\|_{L^2} 
& \lesssim  & 
2^k \bigl\|S_{j'-1}^{\rm v}\D_{k'-\ell} ^{\rm h} a\D_{j'}^{\rm v} \D_{k'}^{\rm h}b\bigr\|_{L^1_{\rm h}(L^2_{\rm v}) }\\
& \lesssim  & 
2^k \|S_{j'-1}^{\rm v}\D_{k'-\ell} ^{\rm h} a\|_{L^2_{\rm h}(L^\infty_{\rm v})} \|\D_{j'}^{\rm v} \D_{k'}^{\rm h}b\bigr\|_{L^2} \, .
\eeno
Thus we infer that 
\[
\begin{split}{
2^{k(\s+\s'-1)+j\left(s+ s'-\frac12\right)} \|\D_j^{\rm v}\D_k^{\rm h} T^{\rm v}R^{\rm h}(a, b)\|_{L^2} \lesssim 
\!\!\!\!\sumetage{|j'-j|\leq N_0}{k'\geq k- N_0}\!\!\sum_{-1\leq \ell\leq 1} 2^{-(k'-k)(\s+\s')}
}\\
{ {}
\times 
2^{j'\left(s-\frac12\right)+k'\s}  \|S_{j'-1}^{\rm v}\D_{k'-\ell} ^{\rm h} a\|_{L^2_{\rm h}(L^\infty_{\rm v})} 
2^{k'\s'+j' s'} \|\D_{j'}^{\rm v} \D_{k'}^{\rm h}b\|_{L^2} \, .\qquad\qquad\qquad
}
\end{split}
\]
Using again anisotropic Bernstein inequalities and by definition of the~$\cB^{\s,s}$ norm, we get
\beno
 2^{j'\left(s-\frac12\right)+k'\s}  \|S_{j'-1}^{\rm v}\D_{k'-\ell} ^{\rm h} a\|_{L^2_{\rm h}(L^\infty_{\rm v})} 
 & \lesssim & \sum_{j''\leq j'-2}  2^{(j'-j'')\left(s-\frac12\right)} 2^{j''s+k'\s} \|\D_{j''}^{\rm v}\D_{k'-\ell} ^{\rm h} a\|_{L^2} \\
 & \lesssim &  \|a\|_{\cB^{\s,s}}  \sum_{j''\leq j'-2}  2^{(j'-j'')\left(s-\frac12\right)} d_{j'',k'} \, .
\eeno
As~$s$ is less than or equal to~$1/2$, we get 
$$
 2^{j'\left(s-\frac12\right)+k'\s}  \|S_{j'-1}^{\rm v}\D_{k'-\ell} ^{\rm h} a\|_{L^2_{\rm h}(L^\infty_{\rm v})} 
\leq \|a\|_{\cB^{\s,s}} \, .
$$
By definition of  the~$\cB^{\s',s'}$ norm, this gives
$$
\longformule{
2^{k(\s+\s'-1)+j\left(s+ s'-\frac12\right)} \|\D_j^{\rm v}\D_k^{\rm h} T^{\rm v}R^{\rm h}(a, b)\|_{L^2} \lesssim  \|a\|_{\cB^{\s,s}} \|b\|_{\cB^{\s',s'}}
}
{ {}
\times 
\!\!\!\!\!\!\sumetage{|j'-j|\leq N_0}{k'\geq k- N_0}\!\!\sum_{-1\leq \ell\leq 1} 2^{-(k'-k)(\s+\s')-(j'-j)\left(s+s'-\frac12\right)}
d_{j',k'} \, .
}
$$
As~$\s+\s'$ is positive, we get that 
$$
2^{k(\s+\s'-1)+j\left(s+ s'-\frac12\right)} \|\D_j^{\rm v}\D_k^{\rm h} T^{\rm v}R^{\rm h}(a, b)\|_{L^2}
\lesssim d_{j,k} \|a\|_{\cB^{\s,s}} \|b\|_{\cB^{\s',s'}} \, .
$$
Together with\refeq{productlawsanisodemolemdemoeq1} and\refeq{productlawsanisodemolemdemoeq2} this concludes the proof of  Inequality\refeq{lawsanisogenpara}.

\medbreak
\noindent In order to prove Inequality\refeq{lawsanisogenremain}, let us use again the horizontal Bony  decomposition. Defining
$$
\wt \D_j^{\rm v}\  \hbox{(resp. $\wt\D_k^{\rm h}$)} = \sum_{\ell=-1}^1 \D_{j-\ell} ^{\rm v} \ \hbox{(resp. $\D_{k-\ell}^{\rm h}$)}
$$  
let us write that 
\beno
R^{\rm v}_a b & =& R^{\rm v}T^{\rm h}_a b + R^{\rm v} T^{\rm h}_b a +  R^{\rm v}R^{\rm h}(a, b)\with\\
R^{\rm v}T^{\rm h}_a b & \eqdefa & \sum_{j,k} \wt \D_j^{\rm v} S_{k-1}^{\rm h} a\D_j^{\rm v}\D_k^{\rm h} b\andf\\
R^{\rm v}R^{\rm h}(a, b) &\eqdefa & \sum_{j,k}
\wt \D_j^{\rm v} \wt \D_{k-\ell}^{\rm h} a\D_j^{\rm v}\D_k^{\rm h} b\,.
 \eeno
 We have for~$N_0$ a  large enough integer,
$$
\D_j^{\rm v}\D_k^{\rm h} R^{\rm v}T^{\rm h}_a b= \sumetage{j'\geq j- N_0}{|k'-k|\leq N_0}
\D_j^{\rm v}\D_k^{\rm h}\bigl(\wt \D_{j'}^{\rm v}S_{k'-1} ^{\rm h} a\D_{j'}^{\rm v} \D_{k'}^{\rm h}b\bigr).
$$
Using anisotropic Bernstein inequalities, this gives  by definition of the~$\cB^{\s,s'}$ norm, 
\beno
2^{j\left (s+s'-\frac12\right)+k(\s+\s'-1)} \|\D_j^{\rm v}\D_k^{\rm h} R^{\rm v}T^{\rm h}_a b\|_{L^2} 
&\lesssim &  2^{j(s+s')+k(\s+\s'-1)} \|\D_j^{\rm v}\D_k^{\rm h} R^{\rm v}T^{\rm h}_a b\|_{L^2_{\rm h}(L^1_{\rm v})} \\
& \lesssim & \!\!\!\! \!\!\!\! 
 \sumetage{j'\geq j- N_0}{|k'-k|\leq N_0} 2^{-(j'-j)(s+s') -(k'-k) (\s+\s'-1)}\\
&&\!\!\!\!\!\!\!\!\!\!\!\!\!\!\!\!\!\!\!\!\!\!\!\!{}\times
2^{j's +k'(\s-1)} \|\wt \D_{j'}^{\rm v}S_{k'-1} ^{\rm h} a\|_{L^\infty_{\rm h}(L^2_{\rm v})}2^{j's'+k'\s'} \|\D_{j'}^{\rm v} \D_{k'}^{\rm h}b\|_{L^2}\\
& \lesssim &  \|b\|_{\cB^{\s',s'}}  \!\!\!\! \sumetage{j'\geq j- N_0}{|k'-k|\leq N_0} 2^{-(j'-j)(s+s')}\\
&&\ \qquad\quad{}\times   d_{j',k'}2^{j's +k'(\s-1)}\|\wt \D_{j'}^{\rm v}S_{k'-1} ^{\rm h} a\|_{L^\infty_{\rm h}(L^2_{\rm v})} \, .
\eeno
Using anisotropic Bernstein inequalities  and the definition of the~$\cB^{\s,s}$ norm, we get
\beno
2^{j's +k'(\s-1)}\|\wt \D_{j'}^{\rm v}S_{k'-1}^{\rm h} a \|_{L^\infty_{\rm h}(L^2_{\rm v})} 
& \lesssim & 
\sumetage{j'-1\leq j''\leq  j'+1} {k''\leq k'-2}  2^{(k'-k'')(\s-1)}\\
&&\qquad\qquad{}\times 2^{j''s +k''(\s-1)} \|\D_{j''}^{\rm v} \D_{k''}^{\rm h} a\|_{L^\infty_{\rm h}(L^2_{\rm v})}\\
& \lesssim & 
\sumetage{j'-1\leq j''\leq  j'+1} {k''\leq k'-2}  2^{(k'-k'')(\s-1)}2^{j''s +k''\s} \|\D_{j''}^{\rm v} \D_{k''}^{\rm h} a\|_{L^2}\\
& \lesssim & \|a\|_{\cB^{\s,s}} \!\!\!
\sumetage{j'-1\leq j''\leq  j'+1} {k''\leq k'-2}  2^{(k'-k'')(\s-1)}
 d_{j''k''} \, .
 \eeno
As~$\s$ is less than or equal to~$1$, we get
$$
2^{j'\left(s-\frac 12\right) +k'(\s-1)}\|\wt \D_{j'}^{\rm v}S_{k'-1}^{\rm h} a \|_{L^\infty}  \lesssim
\|a\|_{\cB^{\s,s}} \, .
$$
Young's inequality on series leads to
\beq
\label{productlawsanisodemolemdemoremaineq1}
\|R^{\rm v}T^{\rm h}_a b\|_{\cB^{\s+\s'-1,s+s'-\frac12}} \lesssim \|a\|_{\cB^{\s,s}}\|b\|_{\cB^{\s',s'}} \, .
\eeq
By symmetry, we get
 \beq
\label{productlawsanisodemolemdemoremaineq2}
\|R^{\rm v}T^{\rm h}_b a\|_{\cB^{\s+\s'-1,s+s'-\frac12}} \lesssim \|a\|_{\cB^{\s,s}}\|b\|_{\cB^{\s',s'}} \, .
\eeq
The estimate of~$R^{\rm v}R^{\rm h}(a, b) $ is a little bit different. Arguing as in the proof of Proposition\refer{regulNS2D}, we obtain
$$
\D_j^{\rm v}\D_k^{\rm h} R^{\rm v}R^{\rm h}(a, b)= \sumetage{j'> j- N_0}{k'\geq k- N_0}
 \D_j^{\rm v} \D_k^{\rm h}\bigl(\wt \D_{j'}^{\rm v}\wt\D_{k'-\ell} ^{\rm h} a\D_{j'}^{\rm v} \D_{k'}^{\rm h}b\bigr).
$$
Anisotropic Bernstein inequalities given by Lemma\refer{Bernsteinaniso} imply that 
\beno
\bigl\|\D_j^{\rm v}\D_k^{\rm h}\bigl(\wt \D_{j'}^{\rm v}\wt \D_{k'} ^{\rm h} a\D_{j'}^{\rm v} \D_{k'}^{\rm h}b\bigr)\bigr\|_{L^2} 
& \lesssim  & 
2^{\frac j 2+k} \bigl\|\wt \D_{j'}^{\rm v}\wt \D_{k'} ^{\rm h} a\D_{j'}^{\rm v} \D_{k'}^{\rm h}b\bigr\|_{L^1}\\
& \lesssim  & 
2^{\frac j 2+k}\|\wt \D_{j'}^{\rm v}\wt \D_{k'} ^{\rm h} a\|_{L^2} \|\D_{j'}^{\rm v} \D_{k'}^{\rm h}b\bigr\|_{L^2} \, .
\eeno
Thus we infer that 
\[
\begin{split}{
2^{k(\s+\s'-1)+j\left(s+ s'-\frac12\right)} \|\D_j^{\rm v}\D_k^{\rm h} R^{\rm v}R^{\rm h}(a, b)\|_{L^2} \lesssim 
\!\!\!\!\sumetage{j'> j- N_0}{k'\geq k- N_0} 2^{-(k'-k)(\s+\s')-(j'-j)(s+s')}
}\\
{ {}
\times 
2^{j's+k'\s}  \|\wt \D_{j'}^{\rm v}\D_{k'-\ell} ^{\rm h} a\|_{L^2} 
2^{k'\s'+j' s'} \|\D_{j'}^{\rm v} \D_{k'}^{\rm h}b\|_{L^2} \, .\qquad\qquad\qquad
}
\end{split}
\]
By definition of  the~$\cB^{\s',s'}$ norm, this gives
$$
\longformule{
2^{k(\s+\s'-1)+j\left(s+ s'-\frac12\right)} \|\D_j^{\rm v}\D_k^{\rm h} R^{\rm v}R^{\rm h}(a, b)\|_{L^2} \lesssim  \|a\|_{\cB^{\s,s}} \|b\|_{\cB^{\s',s'}}
}
{ {}
\times 
\!\!\!\!\!\!\sumetage{j'> j- N_0}{k'\geq k- N_0} 2^{-(k'-k)(\s+\s')-(j'-j)(s+s')}
d_{j',k'} \, .
}
$$
As~$\s+\s'$ and~$s+s'$  are positive, we get that 
$$
2^{k(\s+\s'-1)+j\left(s+ s'-\frac12\right)} \|\D_j^{\rm v}\D_k^{\rm h} R^{\rm v}R^{\rm h}(a, b)\|_{L^2}
\lesssim d_{j,k} \|a\|_{\cB^{\s,s}} \|b\|_{\cB^{\s',s'}} \, .
$$
Together with\refeq{productlawsanisodemolemdemoremaineq1} and\refeq{productlawsanisodemolemdemoremaineq2} this concludes the proof of  Inequality\refeq{lawsanisogenpara}.
 \end{proof}
\noindent  In order to conclude the proof of Proposition\refer{productlawsaniso}, it is enough to apply Lemma\refer{productlawsanisodemolem} with~$(\s,\s')$ to~$T^{\rm v}_a b$ 	and with~$(\wt\s',\wt \s)$ to~$\wt T^{\rm v}_ba$.
 \end{proof}

\medbreak
\noindent Now let us prove laws of product in the case when one of the functions does not depend on the vertical variable~$x_3$. We have the following proposition. 
\begin{prop}
\label{lawproductaniso2/3D}
{\sl Let~$a$ be in~$B^\s_{2,1}(\R^2)$ and~$b$ in~$\cB^{s,s'}$ with~$(s,\s)$ in~$]-1,1]^2$ such that~$s+\s$ is positive and~$s'$ greater than or equal to~$1/2$. We have
\begin{equation}
\label{dim2+3}
\|ab\|_{\cB^{s+\s-1,s'}} \lesssim \|a\|_{B^\s_{2,1}(\R_{\rm h}^2)} \|b\|_{\cB^{s,s'}} \, .
\end{equation}
}
\end{prop}
\begin{proof}
Using Bony's decomposition in the horizontal variable gives
$$
ab = T^{\rm h} _a b +T^{\rm h}_b a +R^{\rm h} (a,b).
$$
As~$a$ does not depend on the vertical variable, we have
$$
\D_j ^{\rm v}T^{\rm h} _a b = T^{\rm h} _a \D_j ^{\rm v} b\,,\ \D_j ^{\rm v}T^{\rm h}_b a= T^{\rm h}_{  \D_j ^{\rm v}b} a \andf \D_j ^{\rm v} R^{\rm h} (a,b) =  R^{\rm h} (a,\D_j ^{\rm v}b).
$$
Then,  the result follows from the  classical proofs of  mappings of paraproduct and remainder operators (see for instance Theorem 2.47 and Theorem  2.52 of\ccite{BCD}). We give a short sketch of the proof for the reader's convenience in the case of~$T^{\rm h}$. Let us write
\beno
2^{k(s+\s-1)+js'}  \|\D_j ^{\rm v}\D_k^{\rm h} T^{\rm h} _a b\|_{L^2} 
& \lesssim &
 \sum_{|k'-k|\leq N_0} 2^{k'(\s-1)} \|S_{k'-1}^{\rm h} a\|_{L^\infty_{\rm h}}  
2^{k's+js'} \|\D_j ^{\rm v}\D_{k'}^{\rm h} b\|_{L^2} \\
& \lesssim & \|b\|_{\cB^{s,s'}} \sum_{|k'-k|\leq N_0}   2^{k'(\s-1)} \|S_{k'-1}^{\rm h} a\|_{L^\infty_{\rm h}}  
 d_{k',j}\, .
\eeno
Bernstein inequalities imply that
\beno
2^{-k(1-\s)} \|S_{k-1}^{\rm h} a\|_{L^\infty_{\rm h}} 
 & \lesssim & 
 \sum_{k'\leq k-1} 2^{(k'-k)(1-\s)}  2^{k'\s}\|\D_{k'}^{\rm h} a\|_{L_{\rm h}^2}  \\
  & \lesssim & \|a\|_{B^\s_{2,1}(\R_{\rm h}^2)}
 \sum_{k'\leq k-1} 
 2^{(k'-k)(1-\s)}   d_{k'}\, .
\eeno
This gives, with no restriction on the parameter~$s$ and  with~$\s$ less than or equal to~$1$ and~$s'$ greater than or equal to~$1/2$,
\beq
\label{lawproductaniso2/3Ddemoeq1}
 \| T^{\rm h} _a b\|_{\cB^{s+\s-1,s'}} \lesssim \|a\|_{B^\s_{2,1}(\R_{\rm h}^2)} \|b\|_{\cB^{s,s'}} \, .\eeq
For the other (horizontal) paraproduct term, let us write \ben
\nonumber2^{k(s+\s-1)+js'} \|\D_j ^{\rm v}\D_k^{\rm h} T^{\rm h} _b a\|_{L^2} 
& \lesssim &
 \sum_{|k'-k|\leq N_0} 2^{k'(s-1)+js'} \|S_{k'-1}^{\rm h} \D_j^{\rm v} b\|_{L^\infty_{\rm h}(L^2_{\rm v})}  
2^{k'\s}\|\D_{k'}^{\rm h} a\|_{L^2_{\rm h}} \\
\label{lawproductaniso2/3Ddemoeq2}
& \lesssim & \|a\|_{B^\s_{2,1}(\R^2)}
 \sum_{|k'-k|\leq N_0} 2^{k'(s-1)+js'} \|S_{k'-1}^{\rm h} \D_j^{\rm v} b\|_{L^\infty_{\rm h}(L^2_{\rm v})}   d_{k'} \, .
 \een
Using Lemma\refer{Bernsteinaniso}, we get
\beno
2^{-k(1-s)+js'} \|S_{k-1}^{\rm h} \D_j^{\rm v} b\|_{L^\infty_{\rm h}(L^2_{\rm v})} 
 & \lesssim & 
 \sum_{k'\leq k-1} 2^{(k'-k)(1-s)}  2^{-k'(1-s)+js'}\|\D_{k'}^{\rm h} \D_j^{\rm v} b\|_{L^\infty_{\rm h}(L^2_{\rm v})} \\
  & \lesssim & 
 \sum_{k'\leq k-1} 
 2^{(k'-k)(1-s)}  2^{k's+js'}\|\D_{k'}^{\rm h} \D_j^{\rm v} b\|_{L^2} \, .
\eeno
By definition of  the $\cB^{s,s'}$ norm and using the fact that~$s\leq 1$, we infer that 
$$
2^{js'-k(1-s)} \|S_{k-1}^{\rm h} \D_j^{\rm v} b\|_{L^\infty_{\rm h}(L^2_{\rm v})} \leq d_j \|b\|_{\cB^{s,s'}} \, .
$$
Together with\refeq{lawproductaniso2/3Ddemoeq2}, this gives
\beq
\label{lawproductaniso2/3Ddemoeq3}
\| T^{\rm h} _b a\|_{\cB^{s+\s-1,s'}} \lesssim \|a\|_{B^\s_{2,1}(\R^2) } \|b\|_{\cB^{s,s'}} \, . 
\eeq
Now let us study the (horizontal) remainder term. Using Lemma\refer{Bernsteinaniso}, let us write that 
\beno
2^{k(s+\s-1)+js'} \|\D_j ^{\rm v}\D_k^{\rm h}  R^{\rm h} (a,b)\|_{L^2} 
& \lesssim &  
2^{k(s+\s)+js'} \|\D_j ^{\rm v}\D_k^{\rm h}  R^{\rm h} (a,b)\|_{L^2_{\rm v}(L^1_{\rm h})}\\
& \lesssim &  \sum_{k'\geq k-N_0}
2^{-(k'-k)(s+\s)}  2^{k'\s} \|\D_{k'}^{\rm h} a\|_{L^2_{\rm h}}2^{k's+js'} \|\D_j ^{\rm v}\D_{k'}^{\rm h}  b\|_{L^2} \, .
\eeno
By definition of the~$B^\s_{2,1}(\R^2)$ and~$\cB^{s,s'}$ norms, we get
$$
2^{k(s+\s-1)+js'}\|\D_j ^{\rm v}\D_k^{\rm h}  R^{\rm h} (a,b)\|_{L^2} 
\lesssim \|a\|_{B^\s_{2,1}(\R^2)} \|b\|_{\cB^{s,s'}} d_j
 \sum_{k'\geq k-N_0} 2^{-(k'-k)(s+\s)} d_{k'} \, .
$$
Together with\refeq{lawproductaniso2/3Ddemoeq1} and\refeq{lawproductaniso2/3Ddemoeq3}, this gives the result thanks to the fact that $s+\s$ is positive. Proposition\refer{lawproductaniso2/3D} is proved.
\end{proof}
  
\subsection{ Proof of Proposition\refer{existencepetitcB1} }
\label{prooftheoexistencepetitcB1}
The proof of Proposition\refer{existencepetitcB1} is reminiscent of that of Lemma\refer{propagationlemmaheathorizontal}, and we shall be using arguments of that proof here.

\medskip
\noindent
Let us recall that we want to prove that if~$U$ is in~$L^2(\R^+;\cB^1)$, if~$u_0$ is in~$\cB^0$ and~$f$ in~$\cF^0$, such that 
\beq
\label{recalltheo4}
\|u_0\|_{\cB^0} +\|f\|_{\cF^0} \leq \frac 1{C_0} \exp \Bigl(-C_0\int_0^\infty \|U(t)\|^2_{\cB^{1}} dt\Bigr) \, ,
\eeq
then the problem
$$
(NS_U)\quad 
\quad  \left\{
\begin{array}{c}
\partial_t u + \dive (u\otimes u+u\otimes U+U\otimes u)  -\Delta u= -\nabla p +f    \\\
\dive u= 0\andf u_{|t=0} = u_{0}
\end{array}
\right.
$$
has a unique global solution in~$L^2(\R^+;\cB^1)$ which  satisfies
$$
\|u\|_{L^2(\R^+;\cB^1)} \lesssim \|u_0\|_{\cB^0} +\|f\|_{\cF^0}\,.
$$
 Let us first prove that the system~$(NS_U)$ has a unique solution in~$L^2([0,T];\cB^1)$ for some small  enough~$T$.   Let us introduce some bilinear operators which distinguish the horizontal derivatives from the vertical one, namely for~$\ell $ belonging to~$ \{1,2,3\}$,
\beq
\label{prooftheoexistencepetitcB1eq0}
\cQ_{\rm h} (u,w)^\ell \eqdefa \dive_{\rm h} (w^\ell u^{\rm h})\andf \cQ_{\rm v} (u,w)^\ell\eqdefa \partial_3(w^\ell u^3).
\eeq
Then we define~$B_{{\rm h},\tau} \eqdefa L_{\tau} \cQ_{\rm h}$ and~$B_{{\rm v},\tau}\eqdefa L_{\tau} \cQ_{\rm v}$ where~$L_\tau$ is defined in Definition\refer{definitionspaces}. It is obvious that solving~$(NS_U)$ is equivalent to solving
$$
u= e^{t\D} u_0+L_0f +B_{{\rm h},0}(u,u)+B_{{\rm v},0}(u,u)
+ B_{{\rm h},0}(U,u)+B_{{\rm v},0}(U,u)+ B_{{\rm h},0}(u,U)+B_{{\rm v},0}(u,U) \, .
$$
Following an idea introduced by
G. Gui, J.  Huang and P. Zhang   in\ccite  {guihuangzhang}, let us define
$$
\cL_0\eqdefa e^{t\D} u_0+L_0f
$$
 and look for the solution~$u$ under the form~$u=\cL_0+\rho$. 
  As the horizontal and the vertical derivative are not treated exactly in the same way, let us decompose~$\rho$ into~$\rho= \rho_{\rm h} +\rho_{\rm v}$ with
 \beq
\label{prooftheoexistencepetitcB1eq1}
\begin{split}
  \rho_{\rm h} & \eqdefa B_{{\rm h},0}(\rho,\rho)+B_{{\rm h},0}(\cL_0+U,\rho)
  +B_{{\rm h},0}(\rho,\cL_0+U)+F_{\rm h}\,,\\ 
    \rho_{\rm v} & \eqdefa B_{{\rm v},0}(\rho,\rho)+B_{{\rm v},0}(\cL_0+U,\rho)
  +B_{{\rm v},0}(\rho,\cL_0+U)+F_{\rm v}\with\\
F_{\rm h} & \eqdefa   B_{{\rm h},0}(\cL_0,\cL_0)+B_{{\rm h},0}(\cL_0,U)
  +B_{{\rm h},0}(U,\cL_0)\andf\\
F_{\rm v} & \eqdefa   B_{{\rm v},0}(\cL_0,\cL_0)+B_{{\rm v},0}(\cL_0,U)
  +B_{{\rm v},0}(U,\cL_0)\, .
\end{split}
\eeq
The main lemma is the following.
\begin{lem}
\label{demopropexistencepetitcB1lemma1}
{\sl For any subinterval~$I= [a,b]$ of~$\R^+$, we have
$$
\begin{aligned}
\|B_{{\rm h},a}(u,w)\|_{L^\infty(I;\cB^0)}& +\|B_{{\rm h},a}(u,w)\|_{L^1(I;\cB^2\cap \cB^{1,\frac 32})} \\
&   +\|B_{{\rm v},a}(u,w)\|_{L^\infty(I;\cB^{1,-\frac12})} +\|B_{{\rm v},a}(u,w)\|_{L^1(I;\cB^2\cap \cB^{1,\frac 32})}\\
& \qquad\qquad\qquad    \qquad \lesssim \|u\|_{L^2(I;\cB^1)}  \|w\|_{L^2(I;\cB^1)} \, . 
\end{aligned}
$$}
\end{lem}
\begin{proof}
As~$\cB^1$ is an algebra and using Lemma\refer{Bernsteinaniso},  we get
\beno
\cQ_{j,k}(u,w)(t) & \eqdefa & 2^{\frac j 2} \|\D_j ^{\rm v}\D_k^{\rm h}\cQ_{\rm h} (u,w) (t)\|_{L^2} +2^{k-\frac j 2} \|\D_j ^{\rm v}\D_k^{\rm h}\cQ_{\rm v} (u,w) (t)\|_{L^2}\\
& \lesssim & d_{j,k}(t) \|u(t)\|_{\cB^1} \|w(t)\|_{\cB^1} \, , 
\eeno
where as usual we have denoted by~$d_{j,k}(t) $ a sequence in the unit sphere of~$ \ell^1(\ZZ^2)$ for each~$t$.
Lemma\refer{anisoheat} implies that, for any~$t$ in~¬†$[a,b]$, we have with the notation of Definition \ref{definitionspaces}
\beno
\cL_{a,j,k}(u,w)(t) & \eqdefa & 2^{\frac j 2} \| L_a\D_j ^{\rm v}\D_k^{\rm h}\cQ_{\rm h}(u,w)(t) \|_{L^2}
+2^{k-\frac j 2} \|L_a\D_j ^{\rm v}\D_k^{\rm h}\cQ_{\rm v}(u,w)(t) \|_{L^2} \\
&\lesssim & \int_a^t d_{j,k}(t') e^{-c2^{(2k+2j)}(t-t')}  \|u(t')\|_{\cB^1} \|w(t')\|_{\cB^1}dt' \, .
\eeno
Convolution inequalities imply that 
$$
\|\cL_{a,j,k}(u,w)  \|_{L^\infty(I;L^2)} +c2^{2k+2j}
\|\cL_{a,j,k}(u,w)\|_{L^1(I;L^2)} \lesssim \int_I d_{j,k}(t)  \|u(t)\|_{\cB^1} \|w(t)\|_{\cB^1} dt \, .
$$
This concludes the proof of  the lemma.
\end{proof}

\medbreak\noindent{\it Continuation of the proof of Proposition~{\rm\refer{existencepetitcB1}}. }
As we have by interpolation,
\beq
\label{prooftheoexistencepetitcB1eq5}
 \|a\|_{\cB^{1}} \leq 
\|a\|_{\cB^{0}}^{\frac12} \|a\|_{\cB^{2}}^{\frac12} \andf
 \|a\|_{\cB^{1}} \leq 
\|a\|_{\cB^{1,-\frac12}}^{\frac12} \|a\|_{\cB^{1,\frac32}}^{\frac12},
\eeq
we infer that the bilinear maps~$B_{{\rm h}, a}$ and~$B_{{\rm v},a}$ map~$L^2(I;\cB^1)\times L^2(I;\cB^1)$ into~$L^2(I;\cB^1)$. A classical fixed point theorem implies the local wellposedness in the space~$L^2(I;\cB^1)$ for initial data in the space~$\cB^0+\cB^{1,-\frac12}$.

\medbreak\noindent Now let us extend this (unique) solution to the whole interval~$\R^+$. 
Given~$\e>0$, to be  chosen small enough later on, let us define~$T_\e$ as
\beq
\label{prooftheoexistencepetitcB1eq2}
T_\e\eqdefa\sup \bigl\{T<T^\star\,,\ \|\rho\|_{L^2([0,T];\cB^1)} \leq \e\bigr\} \, .
\eeq
 As in the proof of Lemma\refer{propagationlemmaheathorizontal}, let us consider the increasing sequence~$(T_m)_{0\leq m\leq M}$ such that~$T_0=0$,~$T_M=\infty$ and for some given~$c_0$ which will be chosen later on
\beq
\label{prooftheoexistencepetitcB1eq3}
\forall m <M-1\,,\ \int_{T_m}^{T_{m+1}} \|U(t)\|_{\cB^1}^2dt = c_0 \andf
\int_{T_{M-1}}^{\infty} \|U(t)\|_{\cB^1}^2dt \leq c_0 \, .
\eeq
Let us recall that from\refeq{propagationlemmaheathorizontaldeoeq2}, we have
\beq
\label{prooftheoexistencepetitcB1eq4}
M\leq \frac 1 {c_0} 
\int_{0}^{\infty} \|U(t)\|_{\cB^1}^2dt \, .
\eeq
Let us define
\beq
\label{prooftheoexistencepetitcB1eq1aa}
\cN_0\eqdefa \|\cL_0\|_{L^2(\R^+;\cB^1)}^2+ \|\cL_0\|_{L^2(\R^+;\cB^1)} \|U\|_{L^2(\R^+;\cB^1)}.
\eeq
Let us consider any~$m$ such that~$T_{m}<T_\e$.
Lemma\refer{demopropexistencepetitcB1lemma1} implies that for any time~$T$ less than~$\min\{T_{m+1};T_\e\}$, we have
\beno
\cR^{\rm h}_m(T)& \eqdefa & \|\rho_{\rm h} \|_{L^\infty([T_m,T];\cB^{0})} +
\|\rho_{\rm h} \|_{L^1([T_m,T];\cB^{2})} \\
& \leq &  
C \|\rho_{\rm h}(T_m)\|_{\cB^{0}} + C\cN_0 \\
&&\quad{}+C\bigl( \|\rho_{\rm h}\|_{L^2([T_m,T];\cB^1)} + \|\cL_0+U\|_{L^2([T_m,T];\cB^1)}\bigr)\|\rho_{\rm h}\|_{L^2([T_m,T];\cB^1)}\\
& \leq &  
C\|\rho_{\rm h}(T_m)\|_{\cB^{0}} + C\cN_0 \\
&&\qquad\qquad\qquad\qquad{}+C\bigl(\e  + \|\cL_0\|_{L^2([T_m,T];\cB^1)}+c_0\bigr)
\|\rho_{\rm h}\|_{L^2([T_m,T];\cB^1)} \,.
\eeno
Choosing~$C_0$ large enough in~\refeq{recalltheo4},~$c_0$ small enough in\refeq{prooftheoexistencepetitcB1eq3}, and~$\e$ small enough  in\refeq{prooftheoexistencepetitcB1eq2} implies that 
\beq\label{rmhissmall}
\cR^{\rm h}_m (T)\leq C \|\rho_{\rm h}(T_m)\|_{\cB^{0}} + C\cN_0
+\frac 1 2  \|\rho_{\rm h}\|_{L^2([T_m,T];\cB^1)} \,.
\eeq
Exactly along the same lines, we get
\beno
\cR^{\rm v}_m (T) & \eqdefa & \|\rho_{\rm v} \|_{L^\infty([T_m,T];\cB^{1,-\frac12})} +
\|\rho_{\rm v} \|_{L^1([T_m,T];\cB^{1,\frac32})} \\
&\leq&  C \|\rho_{\rm v}(T_m)\|_{\cB^{1,-\frac12}} + C\cN_0
+\frac 1 2  \|\rho_{\rm v}\|_{L^2([T_m,T];\cB^1)} \,.
\eeno
We deduce that 
$$
\|\rho_{\rm h}\|_{L^2([T_m,T];\cB^1)} \leq C \bigl(\|\rho_{\rm h}(T_m)\|_{\cB^{0}}  +\cN_0\bigr)\andf
\|\rho_{\rm v}\|_{L^2([T_m,T];\cB^1)} \leq C\bigl(\|\rho_{\rm v}(T_m)\|_{\cB^{1,-\frac12}}+\cN_0\bigr) \, .
$$
This gives, for any~$m$ such that~$T_m<T_\e$ and for all~$T$ in~$[T_m;\min\{T_{m+1},T_\e\}]$,
\beq
\cR^{\rm h}_m(T) +\cR^{\rm v}_m(T)  \leq   C_1
\bigl( \|\rho_{\rm v}(T_m)\|_{\cB^{1,-\frac12}}+  \|\rho_{\rm h}(T_m)\|_{\cB^{0}} + \cN_0\bigr) \, .
\eeq
Let us observe that~$\rho_{|t=0}=0$. Thus exactly as in the proof of Lemma\refer{propagationlemmaheathorizontal}, an iteration process gives, for any~$m$ such that~$T_m<T_\e$ and any $T$ in~$[T_m,\min\{T_{m+1},T_\e\}]$,
\beno
\cR(T) &  \eqdefa & \|\rho_{\rm h} \|_{L^\infty([0,T];\cB^{0})} +
\|\rho_{\rm h} \|_{L^1([0,T];\cB^{2})} 
+ \|\rho_{\rm v} \|_{L^\infty([0,T];\cB^{1,-\frac12})} 
+\|\rho_{\rm h} \|_{L^1([0,T];\cB^{1,\frac32})}\\
& \leq & (C_1 )^{m+1} \cN_0 \,.
\eeno
By definition of~$\cN_0$ given in\refeq{prooftheoexistencepetitcB1eq1aa}, we have in view of Definition \ref{definitionspaces}
$$
\cN_0 \lesssim  \bigl( \|u_0\|_{\cB^0} +\|f\|_{\cF^0}\bigr)
 \bigl(  \|U\|_{L^2(\R^+;\cB^1)} +  \|u_0\|_{\cB^0} +\|f\|_{\cF^0}\bigr)\,.
$$
As claimed in\refeq{prooftheoexistencepetitcB1eq4} the total number of intervals is less than~$ \|U\|^2_{L^2(\R^+;\cB^1)}$. We infer that, for any~$T<T_\e$
$$
\cR(T)\leq C_2 \bigl( \|u_0\|_{\cB^0} +\|f\|_{\cF^0}\bigr)
 \bigl(  \|U\|_{L^2(\R^+;\cB^1)} +  \|u_0\|_{\cB^0} +\|f\|_{\cF^0}\bigr)\exp \bigl(C_2  \|U\|^2_{L^2(\R^+;\cB^1)}\bigr) \, .
$$
Using the interpolation inequality\refeq{prooftheoexistencepetitcB1eq5}  we infer that, for any~$T<T_\e$,
$$
\int_0^T \|\rho(t)\|^2_{\cB^1} dt \leq C_2 \bigl( \|u_0\|_{\cB^0} +\|f\|_{\cF^0}\bigr)
 \bigl(  \|U\|_{L^2(\R^+;\cB^1)} +  \|u_0\|_{\cB^0} +\|f\|_{\cF^0}\bigr)\exp \bigl(C_2  \|U\|^2_{L^2(\R^+;\cB^1)}\bigr)\, .
$$
Choosing 
$$
C_2 \bigl( \|u_0\|_{\cB^0} +\|f\|_{\cF^0}\bigr)
 \bigl(  \|U\|_{L^2(\R^+;\cB^1)} +  \|u_0\|_{\cB^0} +\|f\|_{\cF^0}\bigr)\exp \bigl(C_2  \|U\|^2_{L^2(\R^+;\cB^1)}\bigr) \leq \frac {\e^2} 2
$$
ensures that~$\ds \int_0^T \|\rho(t)\|^2_{\cB^1} dt $   remains less than~$\e^2$,  and thus there is no blow up  for the solution of~$(NS_U)$.  This concludes the proof of Proposition\refer{existencepetitcB1}.\qed

\subsection{Proof of Proposition\refer{propaganisoNS3D}} 
As a warm up, let us observe if~$u$ belongs to~$L^2(\R^+;\cB^1)$, then~$u\otimes u$ belongs to~$L^1(\R^+;\cB^1)$.  Lemma\refer{Bernsteinaniso} implies that the operators~$\cQ_{\rm h} $ and~$\cQ_{\rm v}$ defined in\refeq{prooftheoexistencepetitcB1eq0} satisfy
$$
\|\cQ_{\rm h}(u,u)\|_{L^1(\R^+;\cB^{0})} +\|\cQ_{\rm v}(u,u)\|_{L^1(\R^+;\cB^{1,-\frac12})}
\lesssim \|u\|_{L^2(\R^+;\cB^1)}^2 \, .
$$
Using the Duhamel formula and the action of the heat flow described in Lemma\refer{anisoheat}, we deduce that \beno
\|u\|_{L^1(\R^+;\cB^{2})}+\|u\|_{L^1(\R^+;\cB^{1,\frac32})}\lesssim \|u_0\|_{\cB^0}
+ \|u\|_{L^2(\R^+;\cB^1)}^2 \, ,
\eeno
which proves~(\ref{propaganisoNS3Deq1}).
Let us prove the second inequality of the proposition  which is a propagation type inequality.
 Once an appropriate (para)linearization of the terms~$\cQ_{\rm h}$ and~$\cQ_{\rm v}$  is done, the proof is quite similar to the proof of Proposition\refer{existencepetitcB1}.  Follwing the method of\ccite{chemin20},  let us observe that 
\beno
\dive (u\otimes u)^\ell & = & \dive_{\rm h} (u^\ell u_{\rm h}) +\partial_3(u^\ell u^3)\\
& = & (\dive_{\rm h} u^{\rm h} )u^\ell +  u^{\rm h}\cdot\nabla_{\rm h} u^\ell + \partial_3 \bigl( T^{\rm v}_{u^3} u^\ell
+T^{\rm v}_{u^\ell} u^3 +R^{\rm v}(u^3,u^\ell) \bigr) \, .
\eeno
Now let us define the bilinear operator~$\cT$ by
$$
(\cT_uw)^\ell \eqdefa  (\dive_{\rm h} w^{\rm h} )u^\ell  +u^{\rm h}\cdot\nabla_{\rm h} w^\ell +\partial_3 \bigl( T^{\rm v}_{u^3} w^\ell
+T^{\rm v}_{u^\ell} w^3 +R^{\rm v}(u^3,w^\ell) \bigr) \, .
$$
Let us observe that~$\cT_u u=\dive (u\otimes u)$.
The laws of product of Proposition\refer{productlawsaniso} imply that, for any~$s$ in~$[-1+\mu,1-\mu]$, 
\beq
\label{propaganisoNS3Ddemoeq1}
\|(\dive_{\rm h} w^{\rm h} )u^\ell  +u^{\rm h}\cdot\nabla_{\rm h} w^\ell \|_{\cB^s} \lesssim \|w\|_{\cB^{s+1}}\|u\|_{\cB^1}\, .
\eeq
Lemmas\refer{Bernsteinaniso} and\refer{productlawsanisodemolem} imply  that, for any~$s$ in~$[-1+\mu,1-\mu]$, 
\beq
\label{propaganisoNS3Ddemoeq2}
\bigl\|(\partial_3 \bigl( T^{\rm v}_{u^3} w^\ell
+T^{\rm v}_{u^\ell} w^3 +R^{\rm v}(u^3,w^\ell) \bigr\|_{\cB^{s} }\lesssim \|w\|_{\cB^{s,\frac32}}\|u\|_{\cB^1} \, .
\eeq
Let us notice that for any non negative~$a$, $u$ is  solution of the linear equation 
\beq
\label{propaganisoNS3Ddemoeq3}
w = e^{(t-a)\D} u(a) + L_a\cT_u w\, .
\eeq
The smoothing effect of the heat flow, as described in Lemma\refer{anisoheat}, implies that  for any non negative~$a$, and any~$t$ greater than or equal to~$a$,
\beq
\label{ouf}
\begin{aligned}
&2^{\frac j 2+ks} \|\D_j ^{\rm v}\D_k^{\rm h}L_a\cT_u w(t)\|_{L^2} 
\\
& \quad \lesssim  \int_a^t d_{j,k}(t') e^{-c2^{(2k+2j)}(t-t')}  \|u(t')\|_{\cB^1} 
\bigl(\|w(t')\|_{\cB^{s+1}}+\|w(t')\|_{\cB^{s,\frac 32}}\bigr)dt'.
\end{aligned}
\eeq
This gives, for any~$b$ in~$]a,\infty]$,
$$
\|L_a\cT_uw\|_{L^\infty(I;\cB^s)} +\|L_a\cT_uw\|_{L^1(I;\cB^{s+2}\cap \cB^{s,\frac52})}
\lesssim \|u\|_{L^2(I;\cB^1)} \|w\|_{L^2(I;\cB^{s+1}\cap \cB^{s,\frac32})}
$$
with~$I=[a,b]$.
Now  let us consider the increasing sequence~$(T_m)_{0\leq m\leq M}$ which satisfies\refeq{prooftheoexistencepetitcB1eq3}.
If~$c_0$ is choosen small enough, we have that the linear map~$L_{T_m}\cT_u$ maps the space 
$$
L^2([T_m,T_{m+1}]; \cB^1\cap \cB^{s+1}\cap \cB^{s,\frac32})
$$
into itself with a norm less than~$1$. Thus~¬†$u$ is the unique solution of\refeq{propaganisoNS3Ddemoeq3} and it satisfies, for any~$m$
$$
\|u\|_{L^\infty([T_m,T_{m+1}];\cB^s)}+ \|u\|_{L^2([T_m,T_{m+1}];\cB^{s+1}\cap \cB^{s,\frac 32})} 
\leq C_1 \|u(T_m)\|_{\cB^s} \, .
$$
Arguing as in the proofs of Lemma\refer{propagationlemmaheathorizontal} and Proposition\refer{existencepetitcB1}, we conclude that~$u$ belongs to~$\cA^s$ and that 
$$
\|u\|_{\cA^s} \lesssim  \|u_0\|_{\cB^s} \exp  \bigl(C\|u\|_{L^2(\R^+;\cB^1)}^2\bigr) \, .
$$
Inequality\refeq{propaganisoNS3Deq2} is proved.

\medbreak
\noindent In order to prove Inequality\refeq{propaganisoNS3Deq3}, let us observe that  using Bony's decomposition in the vertical variable, we get
\beno
\dive (u\otimes  u)^\ell  & = & \sum_{m=1}^3 \partial_m (u^\ell u^m)\\
 & = & \sum_{m=1}^3 \partial_m \Bigl (T^{\rm v}_{u^\ell} u^m+T^{\rm v} _{u^m} u^\ell+R^{\rm v} (u^\ell,u^m)\Bigr) \, .
\eeno
Now let us define
$$
(\overline \cT_uw)^\ell  \eqdefa 
%\sum_{m=1}^3 \partial_m (u^\ell u^m)\\
    \sum_{m=1}^3 \partial_m \Bigl (T^{\rm v}_{u^\ell} w^m+T^{\rm v} _{u^m} w^\ell+R^{\rm v} (u^\ell,w^m)\Bigr) \, .
$$
Proposition\refer{productlawsaniso} implies that, if~$m$ equals $1$ or~$2$ then for any ~$s'$ greater than or equal to~$1/2$
\beno
\bigl\| \partial_m \bigl (T^{\rm v}_{u^\ell} w^m+T^{\rm v} _{u^m} w^\ell+R^{\rm v} (u^\ell,w^m)\bigr)\bigr\|_{L^1(\R^+;
\cB^{0,s'})}  & \lesssim  & \|u\|_{L^2(\R^+;\cB^1)} \|w\|_{L^2(\R^+;\cB^{1,s'})}
\andf\\
\bigl\| \partial_3 \bigl (T^{\rm v}_{u^\ell} w^3+T^{\rm v} _{u^3} w^\ell+R^{\rm v} (u^\ell,w^3)\bigr\|_{L^1(\R^+;
\cB^{0,s'})}  & \lesssim & \|u\|_{L^2(\R^+;\cB^1)} \|w\|_{L^2(\R^+;\cB^{0,s'+1})} \, .
\eeno
Thus we get, for any~$a$ in~$\R^+$, any~$b$ in~$I= [a,\infty]$ and any~$r$ in~$ [1,\infty]$, 
$$
 \|L_a\overline \cT_uw\|_{L^r(I;\cB^{\s,\s'+s'})} \lesssim \|u\|_{L^2(I;\cB^1)} 
\bigl(\|w\|_{L^2(I;\cB^{1,s'})}+\|w\|_{L^2(I;\cB^{0,s'+1})}\bigr)
\with \ds \s+\s'=\frac 2 r \cdotp
$$ Then the lines after Inequality\refeq{ouf} can be repeated word for word.
The proposition is proved. \qed

 %%%%%%%%%%%%%%%%%%%%%%%%%%%%%%%%%%%%%%%%%%%%%

\begin{thebibliography}{50}


 \bibitem{adt}
P.~Auscher, S.~Dubois, and P.~Tchamitchian, On the stability of global solutions to {N}avier-Stokes equations in  the space,  {\it Journal de Math\'ematiques Pures et Appliqu\'ees}, {\bf 83}, 2004, pages 673-697.



 \bibitem{BCD} H. Bahouri, J.-Y. Chemin and R. Danchin, {\it Fourier Analysis and Nonlinear
Partial Differential Equations}, {Grundlehren der mathematischen Wissenschaften, Springer}, {\bf 343}, 2011.

\bibitem{bahouricohenkoch} H. Bahouri, A. Cohen and G. Koch,  A general wavelet-based profile decomposition in the critical embedding of function spaces,   {\it Confluentes Mathematici}, {\bf 3}, 2011, pages 1-25.

   \bibitem{bg}   H. Bahouri and I. Gallagher, On the stability in   weak topology  of the set of global solutions to the Navier-Stokes equations,  {\it Archive for Rational Mechanics
and Analysis}, {\bf 209}, 2013, pages 569-629. 


    \bibitem{bahourigerard}
H.~Bahouri and P.~G\'erard, High frequency approximation of
solutions to critical nonlinear wave equations, {\it American  Journal of  Math},
 {\bf 121}, 1999, pages 131-175.

\bibitem{BMM}
H.~Bahouri, M.~Majdoub and N.~Masmoudi,  On  the lack of compactness in the 2D critical Sobolev embedding, {\it  Journal of Functional Analysis}, {\bf 260}, 2011, pages  208-252.

\bibitem{BMM1}
H.~Bahouri, M.~Majdoub and N.~Masmoudi,  Lack of compactness in the 2D critical Sobolev embedding, the general case, 	to appear in {\it Journal de Math\'ematiques Pures et Appliqu\'ees}.


\bibitem{BP}
H.~Bahouri  and G.~Perelman, A Fourier   approach to the profile decomposition in   Orlicz spaces,  {\it submitted}.


\bibitem{bourdaud} G. Bourdaud, La propri\'et\'e de Fatou dans les espaces de Besov homog\`enes, {\it Note aux Comptes Rendus Mathematique de l'Acad\'emie des Sciences}, {\bf 349},2011, pages 837-840.

\bibitem{bourgainpavlovich}
J. Bourgain  and  N.  Pavlovi\' c, 
Ill-posedness of the Navier-Stokes equations in a critical space in 3D,
{\it Journal of Functional Analysis}, {\bf  255}, 2008, pages 2233-2247.

\bibitem{BC}  H. Br\'ezis and J.-M. Coron,  Convergence of solutions of H-Systems or how to blow bubbles,
  {\it  Archive for Rational Mechanics
and Analysis},  {\bf 89}, 1985, pages 21-86.

\bibitem{chemin10}
J.-Y. Chemin, Remarques sur l'existence globale pour le syst\`eme de Navier-Stokes incompressible, {\it  SIAM Journal on Mathematical Analysis},   {\bf 23}, 1992, pages 20-28.

\bibitem{chemin20}
J.-Y. Chemin, 
Th\'eor\`emes d'unicit\'e pour le syst\`eme de Navier-Stokes tridimensionnel.
{\it Journal d'Analyse Math\'ematique}, {\bf 77}, 1999, pages 27-50.

 \bibitem{cg3} J.-Y. Chemin and I. Gallagher, Large, global solutions to the
Navier-Stokes equations, slowly varying in one direction,  {\it Transactions of the American Mathematical Society} {\bf 362}, 2010,  pages 2859-2873.


 \bibitem{cgm}
J.-Y. Chemin,  I. Gallagher and C. Mullaert,
The role of spectral anisotropy in the resolution of the  three-dimensional Navier-Stokes equations, "Studies in Phase Space Analysis with Applications to PDEs", {\it Progress in Nonlinear Differential Equations and Their Applications} {\bf 84}, Birkhauser, pages 53-79, 2013.

 \bibitem{cgz}
J.-Y. Chemin,  I. Gallagher and P. Zhang,
Sums of large global
solutions  to the incompressible  Navier-Stokes equations,  to appear in {\it  Journal f\"ur die reine und angewandte Mathematik}.

\bibitem{cheminlerner}
J.-Y. Chemin and N. Lerner.
Flot de champs de vecteurs non lipschitziens et {\'e}quations de
  {N}avier-{S}tokes,  {\it Journal of Differential Equations}, {\bf 121}, 1995,  pages 314-328.
  
     \bibitem{cheminzhang} J.-Y. Chemin and P. Zhang,   On the global wellposedness to the 3-D incompressible anisotropic Navier-Stokes equations, {\it Communications in Mathematical Physics}, {\bf 272}, 2007, pages 529-566.
   
\bibitem{tintarevandco}
K.-H.  Fieseler and K. Tintarev, {\it Concentration compactness. Functional-analytic grounds and applications} Imperial College Press, London, 2007, pages xii+264 pp.

  \bibitem{gallagherNS} I. Gallagher,   Profile decomposition for solutions of the {N}avier-{S}tokes equations, {\it Bulletin de la Soci\'et\'e Math\'ematique de France}, {\bf 129}, 2001, pages 285-316.

 \bibitem{gip}
I. Gallagher,  D. Iftimie and F. Planchon,  Asymptotics
   and stability for global solutions to the
   Navier--Stokes equations, {\it  Annales de l'Institut Fourier},  {\bf 53}, 2003, pages 1387-1424.
   
 \bibitem{gkp}
I. Gallagher,  G. Koch and F. Planchon,
A profile decomposition approach to the $L^\infty_t(L^{3}_x)$ Navier-Stokes regularity criterion,  {\it Mathematische Annalen}, {\bf 355}, 2013, pages 1527-1559.


% \bibitem{gp}
%  I. Gallagher and F.
% Planchon, On global infinite energy
%  solutions to the Navier--Stokes equations in two dimensions,     {\it
%  Archive for Rational  Mechanics
%  and Analysis}, {\bf 161} (2002), pages 307--337.
\bibitem{pgerard0}
P. G\'erard,  Microlocal defect measures, {\it  Communications in  Partial Differential Equations}, {\bf 16},  1991, pages 1761-1794.

\bibitem{pgerard1}
P. G\'erard,  Description du d\'efaut de compacit\'e de l'injection de Sobolev, {\it  ESAIM Control, Optimisation and Calculus of Variations}, {\bf 3}, 1998, pages  213-233.

\bibitem{pgermainNS}
P. Germain,  The second iterate for the Navier-Stokes equation, {\it 
Journal of Functional Analysis}, {\bf  255}, 2008, pages  2248-2264.


\bibitem{guihuangzhang}
G. Gui, J.  Huang, P. Zhang, 
Large global solutions to 3D inhomogeneous Navier-Stokes equations slowly varying in one variable, {\it Journal of  Functional  Analysis}, {\bf  261}, 2011, pages 3181-3210.

\bibitem{gz} G. Gui and P. Zhang, Stability to the global large solutions of 3-D Navier-Stokes equations, {\it Advances in Mathematics}  {\bf 225}, 2010, pages 1248-1284.

\bibitem{hmidikeraani} T.  Hmidi and S. Keraani,   Blowup theory for the critical nonlinear Schr\"odinger equations revisited, {\it International Mathematics Research Notices},  {\bf  46}, 2005, pages 2815-2828.


\bibitem{dragos} D. Iftimie, Resolution of the Navier-Stokes equations in anisotropic spaces, {\it  Revista Matematica Ibero-americana}, {\bf 15},  1999,  pages 1-36.

\bibitem{jaffard} S. Jaffard,  Analysis of the lack of compactness in the critical Sobolev embeddings,  {\it Journal of Functional Analysis},  {\bf 161}, 1999, pages 384-396.

 \bibitem{jiasverak} H. Jia and V. \c{S}ver\'ak, Minimal $L^3$-initial data for potential Navier-Stokes singularities,   {\it SIAM J. Math. Anal. } {\bf 45},  2013, pages 1448-1459.
 
 
% \bibitem{kato} T. Kato,
%Strong $L^p$-solutions of the Navier-Stokes equation in
%$\R^m$ with applications to weak solutions, {\it Mathematische
%Zeitschrift}, {\bf 187},  1984,  pages 471-480.



\bibitem{kk}
C.E. Kenig and G. Koch,  An alternative approach to the {N}avier-{S}tokes equations in
  critical spaces,  {\it Annales de l'Institut Henri Poincar\'e (C) Non Linear Analysis}, {\bf 28},  2011, Pages 159-187.


\bibitem{km} C. E. Kenig and F. Merle,  Global well-posedness, scattering and blow-up for the energy
critical focusing non-linear wave equation, {\it  Acta Mathematica},  {\bf  201}, 2008,
 pages 147-212.


\bibitem{keraani} S. Keraani,  On the defect of compactness for the Strichartz estimates  of the Schr\"odinger equation, {\it Journal of Differential equations}, {\bf 175},  2001, pages 353-392.

\bibitem{koch}
G. Koch,  Profile decompositions for critical Lebesgue and Besov space embeddings, {\it  Indiana University Mathematical  Journal}, {\bf 59}, 2010, pages 1801-1830.  


\bibitem{kochtataru} H. Koch and D. Tataru, Well--posedness for the
   Navier--Stokes equations, {\it Advances in
Mathematics},  {\bf 157}, 2001, pages~22-35.

\bibitem{lemarie} P.-G. Lemari\'e-Rieusset,  Recent developments in the Navier-Stokes problem,
{\it  Chapman and Hall/CRC Research Notes in Mathematics}, {\bf 43}, 2002.


\bibitem{leray}
J. Leray, Essai sur le mouvement d'un liquide visqueux emplissant l'espace,
{\it Acta Matematica}, {\bf 63}, 1933, pages 193-248.

\bibitem{leray2D} J.~Leray,
\'Etude de diverses \'equations int\'egrales non lin\'eaires et
de quelques probl\`emes que pose l'hydrodynamique.
{\it Journal de  Math\'ematiques Pures et  Appliqu\'ees}, {\bf 12}, 1933,  pages 1-82.

\bibitem{pl2cocomp1}
P.-L. Lions, The concentration-compactness principle in the calculus of variations. The limit case I, {\it   Revista. Matematica Iberoamericana} {\bf 1} (1), 1985, pages 145-201. 

\bibitem{pl2cocomp2}
P.-L. Lions, The concentration-compactness principle in the calculus of variations. The limit case II, {\it   Revista Matematica Iberoamericana} {\bf 1} (2), 1985, pages 45-121. 

\bibitem{meyerNSlivre}
Y. Meyer,  {\it Wavelets, paraproducts and Navier-Stokes equations},  Current developments in mathematics,  International  Press, Boston, MA, 1997. 

\bibitem{merlevega}
 F. Merle and L. Vega,  Compactness at blow-up time for $L^2 $ solutions of the critical nonlinear Schr\"odinger equation in 2D,
 {\it  International Mathematical Research Notices}, {\bf  1998}, pages 399-425. 
 
\bibitem{paicu} M. Paicu, \'Equation anisotrope de Navier-Stokes dans des espaces critiques, {\it Revista Matematica Iberoamericana} {\bf 21} (1), 2005,  pages~179-235.


 \bibitem{Planchon} F.
 Planchon,
Asymptotic behavior of global solutions to the {N}avier-{S}tokes
  equations in {${\bf R}^3$},  {\it Revista Matematica Iberoamericana}, {\bf 14}, 1998, pages 71-93.
  
  
 \bibitem{eugenie1}
 E. Poulon,  Behaviour of Navier-Stokes solutions with data in~$H^s$ with~$1/2<s<\frac32$, {\it in progress}.
 
 \bibitem{RS} T. Runst and W. Sickel:
{\it Sobolev spaces of fractional order, Nemytskij operators, and nonlinear
partial differential equations},
Nonlinear Analysis and Applications, 3. Walter de Gruyter \& Co., Berlin, 1996.
 
  \bibitem{rs} W.
 Rusin and V. \c{S}ver\'ak, 
Minimal initial data for potential Navier-Stokes singularities
{\it Journal of Functional Analysis}  {\bf 260} (3), 2011, pages 879-891.
  
  \bibitem{St} M. Struwe,  A global compactness result for boundary value problems involving limiting
nonlinearities,{\it  Mathematische Zeitschrift}, {\bf 187}, 1984, pages 511-517.
  


\bibitem{tartar}
L.  Tartar,  $H$-measures, a new approach for studying homogenisation, oscillations and concentration effects in partial differential equations, {\it Proceedings of the  Royal Society of Edinburgh}, {\bf  115}, 1990,  pages 193-230.

\bibitem{triebel}
 H. Triebel: {\it Interpolation theory, function spaces,
 differential operators}, Second edition. Johann Ambrosius Barth, Heidelberg,
1995.

\bibitem{triebel1}
 H. Triebel: {\it Theory of function spaces}, Birkh\"auser, Basel, 1983.


\end{thebibliography}
 \end{document}